\documentclass{svproc}
\usepackage{url}

\usepackage{amsmath}
\usepackage{amscd}

\usepackage{xcolor}
\definecolor{columbiablue}{rgb}{0.61, 0.87, 1.0}
\usepackage{wrapfig}
\usepackage[]{pstricks}
\usepackage{fancyhdr}
 \usepackage{xcolor}
\usepackage{flafter}
\usepackage{graphicx}
\usepackage{verbatim}
\definecolor{sandstone}{HTML}{786D5F}
\definecolor{beaublue}{rgb}{0.74, 0.83, 0.9}
\definecolor{cherryblossompink}{rgb}{1.0, 0.72, 0.77}

\newcommand{\mc}[1]{{\mathcal #1}}

\usepackage{tikz}
\usetikzlibrary{arrows,decorations.pathmorphing,backgrounds,positioning,fit,petri}
\usetikzlibrary{shapes}
\usetikzlibrary{arrows.meta}

\newcommand{\tclock}[5]{
\begin{pgflowlevelscope}{\pgftransformscale{#4}}
\begin{scope}[shift={(#1,#2)}]
\shadedraw [inner color=#3!7!gray, outer color=#3!90!black, 
    line width=0.2pt] (0,0) circle (0.5cm);
\foreach \x in {6,12,...,360} {\draw[line width=0.2pt] (\x:0.40cm) -- (\x:0.45cm);}
\foreach \y in {30,60,...,360} {\draw[line width=0.2pt] (\y:0.35cm) -- (\y:0.45cm);}
{\pgfsetarrowsstart{to}
\draw[line width=0.4pt] (0:0.29cm) -- (0.02,0);
\draw[line width=0.4pt]  (90:0.32cm)--(0,0.02);}
\filldraw[fill=black] (-0.055,0.55) rectangle (0.055,0.6);
\filldraw[fill=black] (-0.015,0.51) rectangle (0.015,0.55);
\draw [line width=0.2pt](0,0.61) circle (0.11cm);
\draw [line width=0.2pt](0,0) circle (0.5cm);
\draw [line width=0.2pt](0,0) circle (0.02cm);
\draw [red,thick,domain=30:45] plot ({#5*0.6*cos(\x)}, {#5*0.6*sin(\x)});
\draw [red,thick,domain=20:55] plot ({#5*0.65*cos(\x)}, {#5*0.65*sin(\x)});
\draw [red,thick,domain=10:65] plot ({#5*0.7*cos(\x)}, {#5*0.7*sin(\x)});
\draw [red,thick,domain=135:150] plot ({#5*0.6*cos(\x)}, {#5*0.6*sin(\x)});
\draw [red,thick,domain=125:160] plot ({#5*0.65*cos(\x)}, {#5*0.65*sin(\x)});
\draw [red,thick,domain=170:115] plot ({#5*0.7*cos(\x)}, {#5*0.7*sin(\x)});
\end{scope}
\end{pgflowlevelscope}
}

\begin{document}
\mainmatter              
\title{Hydrodynamics  for symmetric exclusion in contact with reservoirs}
\titlerunning{Hydrodynamics for symmetric exclusion}  
%
\author{Patr\'icia Gon\c calves}
\authorrunning{Patr\'icia Gon\c calves} 

\institute{Center for Mathematical Analysis,  Geometry and Dynamical Systems,
Instituto Superior T\'ecnico, Universidade de Lisboa,
Av. Rovisco Pais, 1049-001 Lisboa, Portugal and   Institut  Henri
Poincar\'e, UMS 839 (CNRS/UPMC), 11 rue Pierre et Marie Curie, 75231 Paris Cedex 05, France,\\
\email{patricia.goncalves@math.tecnico.ulisboa.pt},\\ WWW home page:
\texttt{http://patriciamath.wixsite.com/patricia}
}
\maketitle              

\begin{abstract}
We consider the symmetric exclusion process with jumps  given by a symmetric, translation invariant, transition probability $p(\cdot)$. The process is  put  in contact with stochastic reservoirs whose strength is tuned by  a parameter $\theta\in\bbbr$. Depending on the value of the parameter $\theta$ and the range of the transition probability $p(\cdot)$ we obtain the hydrodynamical behavior of the system. The type of hydrodynamic equation  depends on whether  the underlying  probability $p(\cdot)$ has finite or infinite variance and the type of boundary condition depends on the strength of the stochastic reservoirs, that is, it depends on the value of $\theta$. More precisely, when $p(\cdot)$ has finite  variance we obtain either a reaction or reaction-diffusion equation with Dirichlet boundary conditions or  the heat equation with different types of boundary conditions (of Dirichlet, Robin or Neumann type). When $p(\cdot)$ has infinite variance we obtain a fractional reaction-diffusion equation given by the regional fractional laplacian with Dirichlet boundary conditions but for a particular strength of the reservoirs. 
\end{abstract}
\section{Introduction}
These notes have been written based on material of the articles \cite{Adriana}, \cite{BGJO} and \cite{BGJO2}  which was presented on a mini-course that the author gave while visiting Institut Henri Poincar\'e in Paris in May 2017 for the trimester "Stochastic dynamics out of equilibrium" that held from the 3rd of April to the 7th of July. The slides  and the videos of the mini-course can be seen in\\
 \url{https://indico.math.cnrs.fr/event/844/page/5}.\\

The content of the notes is to explain how to derive partial differential equations with different types of boundary conditions from varied  underlying  microscopic stochastic dynamics  \cite{Spohn,KL}. In the next coming sections  we consider a macroscopic space, namely,  the interval $[0,1]$ and 
 we discretize it according to a scaling parameter $N$ giving  rise to $N$ intervals of size $\frac 1N$. To each $q\in[0,1]$ belonging to the interval $[\frac iN, \frac{i+1}{N})$ we associate to it the point $\tfrac i N$ and in the discrete set of points $\{1,..., i,...,N-1\}$ we will define  a microscopic dynamics of exclusion type which is Markovian. 
 The discrete set of points $\{1,..., i,...,N-1\}$ will be called the bulk and to it we will add two extra points $x=0$ and $x=N$ which will act as reservoirs. 
 The exclusion dynamics \cite{Spitzer} ensures that there is at most one particle per site in the bulk and the Markovian dynamics comes from the fact that each particle waits  for rings of  random clocks exponentially distributed and independent, after which the particle jumps from a site $x$ in the bulk to another site $y$ in the bulk according to a probability transition rate $p:\bbbz\times \bbbz\to [0,1]$, or the particle leaves the system through one of the reservoirs. The reservoirs will be regulated by a parameter which has the role of slowing or fasting the boundary dynamics. More precisely, particles can be injected in the bulk from the site $x=0$ (resp. $x=N$) to the site $y$ at rate $\alpha \kappa N^{-\theta}p(y)$ (resp. $\beta \kappa N^{-\theta}p(N-y)$) and can be removed from the bulk  at the site $y$ to the site $x=0$ (resp. $x=N$) at rate $(1-\alpha) \kappa N^{-\theta}p(y)$ (resp. $(1-\beta) \kappa N^{-\theta}p(N-y)$). Above,  $\alpha, \beta\in[0,1]$, $\theta\in\bbbr$ and $\kappa>0$. \\
 
 The goal in these notes is to derive the partial differential equations which describe the space-time evolution of the density of particles in the system. The type of these  equations will depend on the finiteness of the variance of the underlying transition probability $p(\cdot)$  and the type of boundary conditions  will depend on the strength of the boundary dynamics, namely, the range of the parameter $\theta$. We note that in \cite{FGN,FGS,Jara1,Jara2} similar models have been considered evolving on the full line, that is, without the presence of stochastic reservoirs.\\
 
 The goal is to analyse which type of equation and which type of  boundary conditions we can get and what is their dependence on the strength of the reservoirs. For that purpose,   we split these notes into two main sections to distinguish the case in which jumps are nearest-neighbor or not. Therefore in Section  \ref{chap1}, we consider the dynamics described above but with $p:\bbbz\times\bbbz\to[0,1]$ which satisfies $p(x,y)=p(y-x)=0$ if $|x-y|>1$, $p(0)=0$ so that $p(1)=p(-1)=\frac 1 2$.  This means that in the bulk particles can jump to one of  their nearest-neighbors and particles can be injected/removed in the bulk/from the bulk through the sites $x=1$ or $x=N-1$. For these models we will derive the heat equation with three different types of boundary conditions: non-homogenenous Dirichlet boundary conditions when the reservoirs are fast (which corresponds to $\theta<1$) and Neumann boundary conditions when the reservoirs are slow (which corresponds to $\theta>1$). Linking the aforementioned two types of boundary conditions, for a particular strength of the boundary dynamics (which corresponds to $\theta=1$), we will derive  the heat equation with a type of linear Robin boundary conditions.\\
 
    In Section \ref{chap2},  
we will consider the dynamics described above, but allowing long jumps given by a transition probability  $p:\bbbz\times \bbbz\to[0,1]$ such that $p(x,y)=p(y-x)$, which is symmetric, namely $p(y-x)=p(x-y)$, and we will distinguish two cases: the first one where $p(\cdot)$ has finite variance and then the case where $p(\cdot)$ has infinite variance. In the first case,  we will obtain an extension of the results of the model with only nearest-neighbor jumps, that is we will derive the heat equation with the three types of boundary conditions mentioned above but for a certain choice of the transition probability two new regimes appear when the reservoirs are fast, namely, a reaction-diffusion equation and a reaction equation, both endowed with non-homogeneous Dirichlet boundary conditions. In the case where $p(\cdot)$ has infinite variance and  for a particular strength of the reservoirs (which corresponds to $\theta=0$), we will derive a collection of fractional reaction-diffusion equations with non-homogeneous Dirichlet boundary conditions.  For the interested reader we note that when $p(\cdot)$ has infinite variance and  when the strength of the reservoirs is slow (which corresponds to $\theta>0$), we cannot say anything about the equation nor its boundary conditions. In \cite{BGJO} a similar model has been studied and some conjectures have been presented in the case where the reservoirs are slow. We believe that the same conjecture should be true for this model, but we leave this for a future problem to look at.  We also note that  it would  be very interesting to consider other types of boundary dynamics or even more general type of bulk dynamics than the exclusion  in order to obtain  other partial differential equations with various boundary conditions.\\

These notes are organized as follows: in Section \ref{chap1} we derive  the hydrodynamic limit for the  symmetric  exclusion in contact with stochastic reservoirs but only allowing jumps to nearest-neighbors and in Section \ref{chap2} we derive the hydrodynamics in the case where the system exhibits long jumps. 

More precisely,  in Sections \ref{sec:model}, \ref{sec:illustration} and \ref{sec:generator} we present the dynamics  of the model; in Section \ref{sec:stat_mea} we present its stationary measures; in Section \ref{sec:empirical_prof} we analyse the empirical profile and the two point correlation function; in Section \ref{sec:hyd_eq_ssep} we present the hydrodynamic equations and the notion of their weak solutions; in Section \ref{sec:HL} we state the hydrodynamic limit; in Section \ref{sec:heuri_ssep} we give an heuristic argument to deduce the weak formulation of the solutions by means of auxiliary martingales associated to the process; in Section  \ref{sec:tightness} we prove tightness of the process of empirical measures; in Section \ref{limit point} we characterize the limit point of the tight  sequence  and in Section \ref{sec:hydrostatic} we prove the hydrostatic limit, which is the hydrodynamic limit starting from the invariant measure of the system.

In Section \ref{chap2} we analyse the hydrodynamics for the symmetric exclusion with long jumps given by a transition  probability  which is symmetric. In Section \ref{sec:model_LJ} we describe the model; in Section \ref{sec:fin_var} we present the case in which the underlying transition  probability has finite variance and in Section \ref{sec:inf_var} we analyse the case in which the transition probability has infinite variance. 

In the Appendix we present some of the technical results that are needed along the proofs regarding the derivation of the weak solution of the corresponding hydrodynamic equations.

\section{Symmetric simple  exclusion  in contact with reservoirs}\label{chap1}

\subsection{The model}
\label{sec:model}

In this section we describe the collection of models that we are going to consider in these notes. First we start by fixing the notation which fits all the models and then we particularize our choice of the parameters in such a way that we treat each model, with its special features, separately. For that purpose, let $N$ be a scaling parameter, which will be taken to infinity later on and  denote by  $\Lambda_N=\{1,..., N-1\}$ the discrete set of points to which we call the bulk. 

 The exclusion process in contact with stochastic reservoirs is a Markov process,  denoted by $\{\eta_t:\,t\geq{0}\}$, which has state space $\Omega_N:=\{0,1\}^{\Lambda_N}$. 
 The configurations of the state space $\Omega_N$  are denoted by $\eta$, so that for $x\in\Lambda_N$,  $\eta(x)=0$ means that the site $x$ is vacant while $\eta(x)=1$ means that the site $x$ is occupied. For an illustration of the dynamics let us first take $N=5$ so that the bulk is the discrete  set of points $\{1,2,3,4\}$:
 
 \begin{center}
\begin{tikzpicture}[scale=0.85]
\draw [line width=1] (-1,10.5) -- (2,10.5) ; 
\foreach \x in  {-1,0,1,2} 
\draw[shift={(\x,10.5)},color=black, opacity=1] (0pt,4pt) -- (0pt,-4pt) node[below] {};
\draw[] (-1,10.3) node[below] {1};
\draw[] (0,10.3) node[below] {2};
\draw[] (1,10.3)  node[below] {3};
\draw[] (2,10.3)  node[below] {4};
\end{tikzpicture} 
\end{center}

Now, to describe a possible initial configuration we can do the following. Toss a coin, if we get head we put a particle at the  site $1$ and if we get a tail we leave it empty. Repeat this for each site of the discrete set $\Lambda_5$ and  suppose that we get at the end to the configuration $\eta_0=(0,1,0,0)$ which can be represented as:

\begin{center}
\begin{tikzpicture}[scale=0.85]
\draw [line width=1] (-1,10.5) -- (2,10.5) ; 
\foreach \x in  {-1,0,1,2} 
\draw[shift={(\x,10.5)},color=black, opacity=1] (0pt,4pt) -- (0pt,-4pt) node[below] {};
\draw[] (-1,10.3) node[below] {1};
\draw[] (0,10.3) node[below] {2};
\draw[] (1,10.3)  node[below] {3};
\draw[] (2,10.3)  node[below] {4};
\shade[shading=ball, ball color=black!30!] (-0,10.82) circle (.3);
\end{tikzpicture} 
\end{center}

Now, we start to particularize our choice for the dynamics. We are going to add one reservoir at each end point of the bulk. This means that in our construction, we add the points $x=0$ and $x=N$ to the bulk.  Going back to the picture above, this means that we have now 
the set $\{0,1,2,3,4,5\}$ where particles can be placed, but the sites $x=0$ and $x=5$ will act as reservoirs.

 \begin{center}
\begin{tikzpicture}[scale=0.85]
\draw [line width=1] (-2,10.5) -- (3,10.5) ; 
\foreach \x in  {-2,-1,0,1,2,3} 
\draw[shift={(\x,10.5)},color=black, opacity=1] (0pt,4pt) -- (0pt,-4pt) node[below] {};
\draw[] (-2,10.3) node[below] {0};
\draw[] (-1,10.3) node[below] {1};
\draw[] (0,10.3) node[below] {2};
\draw[] (1,10.3)  node[below] {3};
\draw[] (2,10.3)  node[below] {4};
\draw[] (3,10.3)  node[below] {5};
\end{tikzpicture} 
\end{center}

Note that the bulk stays unchanged, the role of the boundary points $\{0,N\}$  is to allow particles to get in and out of the bulk. So, for example, in the initial configuration given above, now we have the sites $x=0$ and $x=N$ occupied, representing the fact that in $x=0$ and $x=N$ there are particles that can enter to the bulk and that can be removed from the bulk.

\quad 
 \begin{center}
\begin{tikzpicture}[scale=0.85]
\draw [line width=1] (-2,10.5) -- (3,10.5) ; 
\foreach \x in  {-2,-1,0,1,2,3} 
\draw[shift={(\x,10.5)},color=black, opacity=1] (0pt,4pt) -- (0pt,-4pt) node[below] {};
\draw[] (-2,10.3) node[below] {0};
\draw[] (-1,10.3) node[below] {1};
\draw[] (0,10.3) node[below] {2};
\draw[] (1,10.3)  node[below] {3};
\draw[] (2,10.3)  node[below] {4};
\draw[] (3,10.3)  node[below] {5};
\shade[shading=ball, ball color=columbiablue] (-2,10.82) circle (.3);
\shade[shading=ball, ball color=black!30!] (-0,10.82) circle (.3);
\shade[shading=ball, ball color=columbiablue] (3,10.82) circle (.3);
\end{tikzpicture} 
\end{center}

Now we describe the time between jumps. For that purpose, for each pair of sites $(x,y)$ we associate a Poisson process of intensity $p(x,y)=p(y-x)$.  The Poisson processes associated to different bonds are independent. Note that the bonds in the bulk are not oriented. In the first dynamics that we are describing, we consider $p(y-x)=0$ if $|x-y|>1$, $p(1)=p(-1)=\frac 12$ so that jumps can only occur to a nearest-neighbor position and for that reason the exclusion process coins the name \emph{simple exclusion process}.  At the boundary points we associate two Poisson processes to each bond containing a boundary point. More precisely, to the bond $\{0,1\}$  (resp. $\{1,0\}$) we associate a Poisson process of intensity  $\alpha\kappa N^{-\theta}$ (resp. $(1-\alpha)\kappa N^{-\theta}$) and to the bond  $\{N-1,N\}$  (resp. $\{N,N-1\}$) we associate a Poisson process of intensity  $(1-\beta)\kappa N^{-\theta}$ (resp. $\beta \kappa N^{-\theta}$). Above we fix the parameters $\alpha, \beta\in[0,1]$, $\theta\in\bbbr$ and $\kappa>0$.  The role of the parameter $\theta$ is to regulate the slowness/fastness of the reservoirs. If $\theta>0$ and $\theta$ increases then the reservoirs are slower and if $\theta<0$ and $\theta$  decreases then the reservoirs are faster. 

We remark that another interpretation of the previous dynamics  at the boundary could be given as follows. Particles  can either be created or annihilated  at the sites $x=1$ and $x=N-1$ according to the following rates:
\begin{itemize}
\item[-] at site $x=1$: \hspace{3.9cm}- at site $x=N-1$:
\begin{itemize}
\item creation rate $\alpha\kappa N^{-\theta}$,  \hspace{2.5cm}$\bullet$ creation rate $\beta\kappa N^{-\theta}$,
\item annihilation rate $(1-\alpha)\kappa N^{-\theta}$, \hspace{1cm}$\bullet$ annihilation rate $(1-\beta)\kappa N^{-\theta}$.
\end{itemize}
\end{itemize}

  Note that in any case, the exclusion rule has to be respected. At most one particle is allowed at each site of the bulk (recall that the state space is $\{0,1\}^{\Lambda_N}$) so that particles can only be created (resp. removed) at the sites $x=1$ or $x=N-1$ if the corresponding  site is empty (resp. occupied), otherwise nothing happens. Before we proceed let us see an  illustration of a possible realization of the Poisson processes as given in the figure below. 
  
\begin{wrapfigure}{l}{0.52\textwidth}\begin{center}
\begin{tikzpicture}[scale=0.85]
\draw [line width=1] (-2,10.5) -- (3,10.5) ; 
\foreach \x in  {-2,-1,0,1,2,3} 
\draw[shift={(\x,10.5)},color=black, opacity=1] (0pt,4pt) -- (0pt,-4pt) node[below] {};
\draw[] (-2.8,10.5) node[] {$\eta_{0}$};

\shade[shading=ball, ball color=columbiablue] (-2,10.82) circle (.3);
\shade[shading=ball, ball color=black!30!] (-0,10.82) circle (.3);
\shade[shading=ball, ball color=columbiablue] (3,10.82) circle (.3);
\draw[<-|] (-3.5,5) -- (-3.5,10.5) node [left] {0};

\draw (-3.5,9.35) node[] {$-$};		
\draw (-3.5,9.27) node[] {$-$};		
\draw (-3.5,8.09) node[] {$-$};		
\draw (-3.5,6.45) node[] {$-$};		

\draw (-3.5,9.25)  node[] {$-$}; 	
\draw (-3.5,7.45)  node[] {$-$};	
\draw (-3.5,7.13)  node[] {$-$};	
\draw (-3.5,6.65)  node[] {$-$};	

\draw (-3.5,10.25)  node[] {$-$};	
\draw (-3.5,10.18) 	node[] {$-$};	
\draw (-3.5,8.65)  	node[] {$-$};	
\draw (-3.5,6.35) 	node[] {$-$};	

\draw (-3.5,10.48)  node[] {$-$};	
\draw (-3.5,9.68) 	node[] {$-$};	
\draw (-3.5,9.24)  	node[] {$-$};	
\draw (-3.5,6.07)  	node[] {$-$};	

\draw (-3.5,9.92)  	node[] {$-$};	
\draw (-3.5,8.96) 	node[] {$-$};	
\draw (-3.5,8.8585) node[] {$-$};	
\draw (-3.5,6.1985) node[] {$-$};	

\draw (-3.5,4.5) node[above] {Time};	
\draw[<-] (-1.5,5) -- (-1.5,10.5) node [left] {};
\draw (-1.5,9.35) node[] {$\times$};		
\draw (-1.5,9.27) node[] {$\times$};		
\draw (-1.5,8.09) node[] {$\times$};		
\draw (-1.5,6.45) node[] {$\times$};		
\draw[<-] (-0.5,5) -- (-0.5,10.5) node [left] {};
\draw (-0.5,9.92)  	node[] {$\times$};	
\draw (-0.5,8.96) 	node[] {$\times$};	
\draw (-0.5,8.8585) node[] {$\times$};	
\draw (-0.5,6.1985) node[] {$\times$};	
\draw[<-] (0.5,5) -- (0.5,10.5) node [left] {};
\draw (0.5,10.48) 	node[] {$\times$};	
\draw (0.5,9.68) 	node[] {$\times$};	
\draw (0.5,9.24) 	node[] {$\times$};	
\draw (0.5,6.07) 	node[] {$\times$};	
\draw[<-] (1.5,5) -- (1.5,10.5) node [left] {};
\draw (1.5,9.25)  node[] {$\times$}; 	
\draw (1.5,7.45)  node[] {$\times$};	
\draw (1.5,7.13)  node[] {$\times$};	
\draw (1.5,6.65)  node[] {$\times$};	
\draw[<-] (2.5,5) -- (2.5,10.5) node [left] {};
\draw (2.5,10.25)  node[] {$\times$};	
\draw (2.5,10.18) 	node[] {$\times$};	
\draw (2.5,8.65)  	node[] {$\times$};	
\draw (2.5,6.45) 	node[] {$\times$};	

\end{tikzpicture}
\end{center}
\end{wrapfigure}

In the figure at the left hand side we represent by "$\times$" each mark of a possible realization of  the Poisson processes associated to the bonds. At the left hand side we put an arrow going down which is  representing the evolution of  time  and each sign "$-$" means that a clock has rung according to some Poisson clock, so that at the corresponding time, a jump from a particle  might have occurred. 

We note that in this figure  we did not distinguish the marks of the Poisson processes associated to the oriented bonds at the boundary because we believe that it is simpler to analyse the dynamics at the boundary by allowing particles to get in or get out according to the Poisson marks but also taking into account  the exclusion rule. 

In order to give an example,  let us see now all the configurations that we obtain starting the dynamics from the configuration $\eta_0=(0,1,0,0)$ represented above and the realization of the Poisson processes given in the previous figure. 

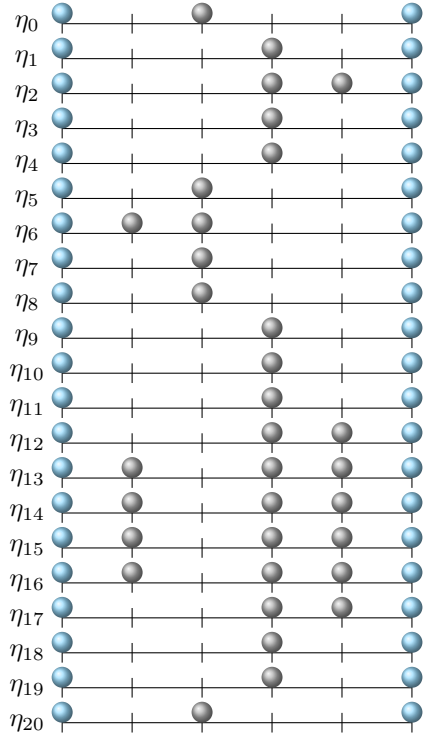
\begin{wrapfigure}{l}{0.45\textwidth}
\begin{center}
\begin{tikzpicture}[scale=0.93]
\draw (5,10.4) -- (10,10.4) ; 
\foreach \x in  {5,6,7,8,9,10} 
\draw[shift={(\x,10.4)},color=black, opacity=1] (0pt,4pt) -- (0pt,-4pt) node[below] {};
\draw[] (4.5,10.4) node[] {$\eta_{0}$};

\shade[shading=ball, ball color=columbiablue] (5,10.55) circle (.15);
\shade[shading=ball, ball color=black!30!] (7,10.55) circle (.15);
\shade[shading=ball, ball color=columbiablue] (10,10.55) circle (.15);
\draw (5,9.9) -- (10,9.9) ; 
\foreach \x in  {5,6,7,8,9,10} 
\draw[shift={(\x,9.9)},color=black, opacity=1] (0pt,4pt) -- (0pt,-4pt) node[below] {};
\draw[] (4.5,9.9) node[] {$\eta_{1}$};

\shade[shading=ball, ball color=columbiablue] (5,10.05) circle (.15);
\shade[shading=ball, ball color=black!30!] (8,10.05) circle (.15);
\shade[shading=ball, ball color=columbiablue] (10,10.05) circle (.15);
\draw (5,9.4) -- (10,9.4) ; 
\foreach \x in  {5,6,7,8,9,10} 
\draw[shift={(\x,9.4)},color=black, opacity=1] (0pt,4pt) -- (0pt,-4pt) node[below] {};
\draw[] (4.5,9.4) node[] {$\eta_{2}$};

\shade[shading=ball, ball color=columbiablue] (5,9.55) circle (.15);
\shade[shading=ball, ball color=black!30!] (8,9.55) circle (.15);
\shade[shading=ball, ball color=black!30!] (9,9.55) circle (.15);
\shade[shading=ball, ball color=columbiablue] (10,9.55) circle (.15);
\draw (5,8.9) -- (10,8.9) ; 
\foreach \x in  {5,6,7,8,9,10} 
\draw[shift={(\x,8.9)},color=black, opacity=1] (0pt,4pt) -- (0pt,-4pt) node[below] {};
\draw[] (4.5,8.9) node[] {$\eta_{3}$};

\shade[shading=ball, ball color=columbiablue] (5,9.05) circle (.15);
\shade[shading=ball, ball color=black!30!] (8,9.05) circle (.15);
\shade[shading=ball, ball color=columbiablue] (10,9.05) circle (.15);
\draw (5,8.4) -- (10,8.4) ; 
\foreach \x in  {5,6,7,8,9,10} 
\draw[shift={(\x,8.4)},color=black, opacity=1] (0pt,4pt) -- (0pt,-4pt) node[below] {};
\draw[] (4.5,8.4) node[] {$\eta_{4}$};
\shade[shading=ball, ball color=columbiablue] (5,8.55) circle (.15);
\shade[shading=ball, ball color=black!30!] (8,8.55) circle (.15);
\shade[shading=ball, ball color=columbiablue] (10,8.55) circle (.15);
\draw (5,7.9) -- (10,7.9) ; 
\foreach \x in  {5,6,7,8,9,10} 
\draw[shift={(\x,7.9)},color=black, opacity=1] (0pt,4pt) -- (0pt,-4pt) node[below] {};
\draw[] (4.5,7.9) node[] {$\eta_{5}$};
\shade[shading=ball, ball color=columbiablue] (5,8.05) circle (.15);
\shade[shading=ball, ball color=black!30!] (7,8.05) circle (.15);
\shade[shading=ball, ball color=columbiablue] (10,8.05) circle (.15);
\draw (5,7.4) -- (10,7.4) ; 
\foreach \x in  {5,6,7,8,9,10} 
\draw[shift={(\x,7.4)},color=black, opacity=1] (0pt,4pt) -- (0pt,-4pt) node[below] {};
\draw[] (4.5,7.4) node[] {$\eta_{6}$};

\shade[shading=ball, ball color=columbiablue] (5,7.55) circle (.15);
\shade[shading=ball, ball color=black!30!] (6,7.55) circle (.15);
\shade[shading=ball, ball color=black!30!] (7,7.55) circle (.15);
\shade[shading=ball, ball color=columbiablue] (10,7.55) circle (.15);
\draw (5,6.9) -- (10,6.9) ; 
\foreach \x in  {5,6,7,8,9,10} 
\draw[shift={(\x,6.9)},color=black, opacity=1] (0pt,4pt) -- (0pt,-4pt) node[below] {};
\draw[] (4.5,6.9) node[] {$\eta_{7}$};

\shade[shading=ball, ball color=columbiablue] (5,7.05) circle (.15);
\shade[shading=ball, ball color=black!30!] (7,7.05) circle (.15);
\shade[shading=ball, ball color=columbiablue] (10,7.05) circle (.15);
\draw (5,6.4) -- (10,6.4) ; 
\foreach \x in  {5,6,7,8,9,10} 
\draw[shift={(\x,6.4)},color=black, opacity=1] (0pt,4pt) -- (0pt,-4pt) node[below] {};
\draw[] (4.5,6.4) node[] {$\eta_{8}$};

\shade[shading=ball, ball color=columbiablue] (5,6.55) circle (.15);
\shade[shading=ball, ball color=black!30!] (7,6.55) circle (.15);
\shade[shading=ball, ball color=columbiablue] (10,6.55) circle (.15);
\draw (5,5.9) -- (10,5.9) ; 
\foreach \x in  {5,6,7,8,9,10} 
\draw[shift={(\x,5.9)},color=black, opacity=1] (0pt,4pt) -- (0pt,-4pt) node[below] {};
\draw[] (4.5,5.9) node[] {$\eta_{9}$};

\shade[shading=ball, ball color=columbiablue] (5,6.05) circle (.15);
\shade[shading=ball, ball color=black!30!] (8,6.05) circle (.15);
\shade[shading=ball, ball color=columbiablue] (10,6.05) circle (.15);
\draw (5,5.4) -- (10,5.4) ; 
\foreach \x in  {5,6,7,8,9,10} 
\draw[shift={(\x,5.4)},color=black, opacity=1] (0pt,4pt) -- (0pt,-4pt) node[below] {};
\draw[] (4.5,5.4) node[] {$\eta_{10}$};

\shade[shading=ball, ball color=columbiablue] (5,5.55) circle (.15);
\shade[shading=ball, ball color=black!30!] (8,5.55) circle (.15);
\shade[shading=ball, ball color=columbiablue] (10,5.55) circle (.15);
\draw (5,4.9) -- (10,4.9) ; 
\foreach \x in  {5,6,7,8,9,10} 
\draw[shift={(\x,4.9)},color=black, opacity=1] (0pt,4pt) -- (0pt,-4pt) node[below] {};
\draw[] (4.5,4.9) node[] {$\eta_{11}$};

\shade[shading=ball, ball color=columbiablue] (5,5.05) circle (.15);
\shade[shading=ball, ball color=black!30!] (8,5.05) circle (.15);
\shade[shading=ball, ball color=columbiablue] (10,5.05) circle (.15);
\draw (5,4.4) -- (10,4.4) ; 
\foreach \x in  {5,6,7,8,9,10} 
\draw[shift={(\x,4.4)},color=black, opacity=1] (0pt,4pt) -- (0pt,-4pt) node[below] {};
\draw[] (4.5,4.4) node[] {$\eta_{12}$};

\shade[shading=ball, ball color=columbiablue] (5,4.55) circle (.15);
\shade[shading=ball, ball color=black!30!] (8,4.55) circle (.15);
\shade[shading=ball, ball color=black!30!] (9,4.55) circle (.15);
\shade[shading=ball, ball color=columbiablue] (10,4.55) circle (.15);
\draw (5,3.9) -- (10,3.9) ; 
\foreach \x in  {5,6,7,8,9,10} 
\draw[shift={(\x,3.9)},color=black, opacity=1] (0pt,4pt) -- (0pt,-4pt) node[below] {};
\draw[] (4.5,3.9) node[] {$\eta_{13}$};

\shade[shading=ball, ball color=columbiablue] (5,4.05) circle (.15);
\shade[shading=ball, ball color=black!30!] (6,4.05) circle (.15);
\shade[shading=ball, ball color=black!30!] (8,4.05) circle (.15);
\shade[shading=ball, ball color=black!30!] (9,4.05) circle (.15);
\shade[shading=ball, ball color=columbiablue] (10,4.05) circle (.15);
\draw (5,3.4) -- (10,3.4) ; 
\foreach \x in  {5,6,7,8,9,10} 
\draw[shift={(\x,3.4)},color=black, opacity=1] (0pt,4pt) -- (0pt,-4pt) node[below] {};
\draw[] (4.5,3.4) node[] {$\eta_{14}$};

\shade[shading=ball, ball color=columbiablue] (5,3.55) circle (.15);
\shade[shading=ball, ball color=black!30!] (6,3.55) circle (.15);
\shade[shading=ball, ball color=black!30!] (8,3.55) circle (.15);
\shade[shading=ball, ball color=black!30!] (9,3.55) circle (.15);
\shade[shading=ball, ball color=columbiablue] (10,3.55) circle (.15);
\draw (5,2.9) -- (10,2.9) ; 
\foreach \x in  {5,6,7,8,9,10} 
\draw[shift={(\x,2.9)},color=black, opacity=1] (0pt,4pt) -- (0pt,-4pt) node[below] {};
\draw[] (4.5,2.9) node[] {$\eta_{15}$};

\shade[shading=ball, ball color=columbiablue] (5,3.05) circle (.15);
\shade[shading=ball, ball color=black!30!] (6,3.05) circle (.15);
\shade[shading=ball, ball color=black!30!] (8,3.05) circle (.15);
\shade[shading=ball, ball color=black!30!] (9,3.05) circle (.15);
\shade[shading=ball, ball color=columbiablue] (10,3.05) circle (.15);
\draw (5,2.4) -- (10,2.4) ; 
\foreach \x in  {5,6,7,8,9,10} 
\draw[shift={(\x,2.4)},color=black, opacity=1] (0pt,4pt) -- (0pt,-4pt) node[below] {};
\draw[] (4.5,2.4) node[] {$\eta_{16}$};

\shade[shading=ball, ball color=columbiablue] (5,2.55) circle (.15);
\shade[shading=ball, ball color=black!30!] (6,2.55) circle (.15);
\shade[shading=ball, ball color=black!30!] (8,2.55) circle (.15);
\shade[shading=ball, ball color=black!30!] (9,2.55) circle (.15);
\shade[shading=ball, ball color=columbiablue] (10,2.55) circle (.15);
\draw (5,1.9) -- (10,1.9) ; 
\foreach \x in  {5,6,7,8,9,10} 
\draw[shift={(\x,1.9)},color=black, opacity=1] (0pt,4pt) -- (0pt,-4pt) node[below] {};
\draw[] (4.5,1.9) node[] {$\eta_{17}$};

\shade[shading=ball, ball color=columbiablue] (5,2.05) circle (.15);
\shade[shading=ball, ball color=black!30!] (8,2.05) circle (.15);
\shade[shading=ball, ball color=black!30!] (9,2.05) circle (.15);
\shade[shading=ball, ball color=columbiablue] (10,2.05) circle (.15);
\draw (5,1.4) -- (10,1.4) ; 
\foreach \x in  {5,6,7,8,9,10} 
\draw[shift={(\x,1.4)},color=black, opacity=1] (0pt,4pt) -- (0pt,-4pt) node[below] {};
\draw[] (4.5,1.4) node[] {$\eta_{18}$};

\shade[shading=ball, ball color=columbiablue] (5,1.55) circle (.15);
\shade[shading=ball, ball color=black!30!] (8,1.55) circle (.15);
;\shade[shading=ball, ball color=columbiablue] (10,1.55) circle (.15);
\draw (5,0.9) -- (10,0.9) ; 
\foreach \x in  {5,6,7,8,9,10} 
\draw[shift={(\x,0.9)},color=black, opacity=1] (0pt,4pt) -- (0pt,-4pt) node[below] {};
\draw[] (4.5,0.9) node[] {$\eta_{19}$};

\shade[shading=ball, ball color=columbiablue] (5,1.05) circle (.15);
\shade[shading=ball, ball color=black!30!] (8,1.05) circle (.15);
\shade[shading=ball, ball color=columbiablue] (10,1.05) circle (.15);
\draw (5,0.4) -- (10,0.4) ; 
\foreach \x in  {5,6,7,8,9,10} 
\draw[shift={(\x,0.4)},color=black, opacity=1] (0pt,4pt) -- (0pt,-4pt) node[below] {};
\draw[] (4.5,0.4) node[] {$\eta_{20}$};

\shade[shading=ball, ball color=columbiablue] (5,0.55) circle (.15);
\shade[shading=ball, ball color=black!30!] (7,0.55) circle (.15);
\shade[shading=ball, ball color=columbiablue] (10,0.55) circle (.15);

\end{tikzpicture}
\caption{Possible configurations starting from $(0,1,0,0)$}
\end{center}
\end{wrapfigure}

\bigskip

By abuse of notation, in the figure at the left hand side,  we numbered the configurations that we obtained by the number of the marks of the Poisson processes (which in the example are equal to  $20$) just to make the presentation simple. We note that the configurations are indexed by time $t$ which is continuous and not discrete. 
Note that the difference between $\eta_0=(0,1,0,0)$ and $\eta_1=(0,0,1,0)$ is only at two sites (this is always  the case when we compare two configurations which differ by a jump of a particle in the bulk, \emph{a jump in the bulk  affects the occupation  variables at two sites}) and $\eta_1$ is obtained from $\eta_0$ by shifting the particle at the site $2$ in $\eta_0$ to the site $3$. This is a consequence of the fact that the first mark of the Poisson process that occurs is associated to  the bond $\{2,3\}$ and that in $\eta_0$ there is a particle at the site $2$.  The next mark we see  is associated  to the bond $\{4,5\}$ and since in $\eta_1=(0,0,1,0)$ there is no particle at the site $x=4$, a particle is injected in the bulk at the site $4$, giving rise to $\eta_2=(0,0,1,1)$ and so on. Note that the boundary dynamics only changes the configuration at one site.

We also note that the ring of a clock does not imply that the configuration of the system has changed. In the example above $\eta_3=\eta_{4}=(0,0,1,0)$ since the corresponding Poisson mark is associated to the bond $\{1,2\}$ and since both sites $x=1$ and $x=2$ are empty, nothing happens and particles wait a new ring of a clock. 

The first dynamics that we are going to consider in these notes, and which is described in this section is completely characterized by now, but we  note that in Section \ref{chap2} we are going to  generalize the previous dynamics by allowing particles to give long jumps according to some probability transition rate $p:\bbbz\times \bbbz\to[0,1]$ such that $p(x,y)=p(y-x)$ and  which is symmetric, that is $p(y-x)=p(x-y)$. In the latter dynamics, there is only one reservoir at each end point of the bulk but particles can be injected from them to any site of the bulk or they can be removed from any site of the bulk to one of the reservoirs.  We will distinguish two cases: when $p(\cdot)$ has finite variance and when $p(\cdot)$ has infinite variance.

\subsection{Illustration of the dynamics}
\label{sec:illustration}

In this section we draw some pictures to illustrate more easily the dynamics that we defined in the previous subsection. The particles at the bulk are coloured in gray  and the particles at the  two reservoirs are coloured in blue. We also added the clocks only at the bonds where there are particles but we note that the clocks are present in all bonds of the form $\{x,x+1\}$. Whenever there is a ring of a clock we see some red lines on top of the corresponding clock and the jump rates are indicated above the corresponding jumps which are represented by arrows. 

In the first picture below, we take $N=11$ and the initial configuration is $\eta_0=(0,0,1,0,0,1,0,0,1,0)$.  Note that this initial  configuration changes only if one of the clocks associated to  bonds containing the sites $x=3,6,9$ rings (which makes the corresponding particle to displace one position to the left or right of it) or if the clocks at the boundary sites $x=0$ (resp. $x=11$) ring   (which makes a particle  get into the system  at the site $x=1$ (resp. $x=10$).

 \begin{center}
 \begin{tikzpicture}[thick, scale=0.85][h!]
 \draw[latex-] (-6.5,0) -- (5.5,0) ;
\draw[-latex] (-6.5,0) -- (5.5,0) ;
\foreach \x in  {-6,-5,-4,-3,-2,-1,0,1,2,3,4,5}
\draw[shift={(\x,0)},color=black] (0pt,0pt) -- (0pt,-3pt) node[below] 
{};
\node[ball color=black!30!,shape=circle,minimum size=0.7cm] (A) at (0,0.4) {};
       \node[shape=circle,minimum size=0.6cm] (N) at (-5,1.2) {};
       \node[shape=circle,minimum size=0.6cm] (P) at (-5,2.0) {};
    \node[ball color=black!30!,shape=circle,minimum size=0.7cm] (C) at (-3,0.4) {};
   \node[ball color=black!30!,shape=circle,minimum size=0.7cm] (E) at (3,0.4) {};
    \node[shape=circle,minimum size=0.7cm] (E) at (3,0.4) {};
         \node[ball color=columbiablue,shape=circle,minimum size=0.7cm] (L) at (-6,0.4) {};
          \node[shape=circle,minimum size=0.7cm] (Q) at (-6,2.0) {};
           \node[ball color=columbiablue,shape=circle,minimum size=0.7cm] (M) at (-6,1.2) {};
           \node[ball color=columbiablue,shape=circle,minimum size=0.7cm] (R) at (5,0.4) {};
          \node[shape=circle,minimum size=0.7cm] (S) at (5,1.2) {};
          \node[ball color=columbiablue,shape=circle,minimum size=0.7cm] (R) at (5,0.4) {};
          \node[ball color=columbiablue,shape=circle,minimum size=0.7cm] (S) at (5,1.2) {};
              \node[shape=circle,minimum size=0.7cm] (T) at (5,2.0) {};
               \node[shape=circle,minimum size=0.7cm] (U) at (4,3.6) {};
                 \node[shape=circle,minimum size=0.7cm] (V) at (5,3.0) {};
                  \node[shape=circle,minimum size=0.7cm] (W) at (4,3.0) {};
\node[ball color=columbiablue,shape=circle,minimum size=0.7cm] (T) at (5,2.0) {};
              \node[shape=circle,minimum size=0.7cm] (U) at (4,2.0) {};
              \path [->] (T) edge[bend right =60] node[above] {${\textcolor{red}{\kappa}}\beta/N^\theta$} (U);
              
                \path [->] (W) edge[bend left=60] node[above] {${\textcolor{red}{\kappa}}(1-\beta)/N^\theta$} (V);
              
               \path [->] (P) edge[bend right=60] node[above] {${\textcolor{red}{\kappa}}(1-\alpha)/N^\theta$} (Q);
               \path [->] (M) edge[bend left=60] node[above] {${\textcolor{red}{\kappa}}\alpha/N^\theta$} (N);
              
    \node[shape=circle,minimum size=0.6cm] (K) at (-1,0.4) {};
        \node[shape=circle,minimum size=0.6cm] (G) at (1,0.4) {};
  \path [->] (A) edge[bend right=60,draw=black] node[above] {$\frac 12$} (K);
    \path [->] (A) edge[bend left=60] node[above] {$\frac 12$} (G);
\tclock{0.6}{-0.8}{columbiablue}{0.8}{1}
\tclock{-0.6}{-0.8}{columbiablue}{0.8}{0}
\tclock{4.4}{-0.8}{columbiablue}{0.8}{0}
\tclock{-4.3}{-0.8}{columbiablue}{0.8}{0}
\tclock{-6.8}{-0.8}{columbiablue}{0.8}{0}
\tclock{5.6}{-0.8}{columbiablue}{0.8}{0}
\tclock{3.2}{-0.8}{columbiablue}{0.8}{0}
\tclock{-3.1}{-0.8}{columbiablue}{0.8}{0}
 \end{tikzpicture}
 \end{center}
 Suppose that the first clock to ring is associated to the bond $\{6,7\}$. Since there is a particle at the site $x=6$ it jumps to the site $x=7$ with rate $1/2$. See the figure below.

  \begin{center}
 \begin{tikzpicture}[thick, scale=0.85][h!]
 \draw[latex-] (-6.5,0) -- (5.5,0) ;
\draw[-latex] (-6.5,0) -- (5.5,0) ;
\foreach \x in  {-6,-5,-4,-3,-2,-1,0,1,2,3,4,5}
\draw[shift={(\x,0)},color=black] (0pt,0pt) -- (0pt,-3pt) node[below] 
{};
\node[ball color=black!30!,shape=circle,minimum size=0.7cm] (A) at (1,0.4) {};
       \node[shape=circle,minimum size=0.6cm] (N) at (-5,1.2) {};
       \node[shape=circle,minimum size=0.6cm] (P) at (-5,2.0) {};
    \node[ball color=black!30!,shape=circle,minimum size=0.7cm] (C) at (-3,0.4) {};
   \node[ball color=black!30!,shape=circle,minimum size=0.7cm] (E) at (3,0.4) {};
    \node[shape=circle,minimum size=0.7cm] (E) at (3,0.4) {};
         \node[ball color=columbiablue,shape=circle,minimum size=0.7cm] (L) at (-6,0.4) {};
          \node[shape=circle,minimum size=0.7cm] (Q) at (-6,2.0) {};
           \node[ball color=columbiablue,shape=circle,minimum size=0.7cm] (M) at (-6,1.2) {};
           \node[ball color=columbiablue,shape=circle,minimum size=0.7cm] (R) at (5,0.4) {};
          \node[shape=circle,minimum size=0.7cm] (S) at (5,1.2) {};
          \node[ball color=columbiablue,shape=circle,minimum size=0.7cm] (R) at (5,0.4) {};
          \node[ball color=columbiablue,shape=circle,minimum size=0.7cm] (S) at (5,1.2) {};
              \node[shape=circle,minimum size=0.7cm] (T) at (5,2.0) {};
               \node[shape=circle,minimum size=0.7cm] (U) at (4,3.6) {};
                 \node[shape=circle,minimum size=0.7cm] (V) at (5,3.0) {};
                  \node[shape=circle,minimum size=0.7cm] (W) at (4,3.0) {};
\node[ball color=columbiablue,shape=circle,minimum size=0.7cm] (T) at (5,2.0) {};
              \node[shape=circle,minimum size=0.7cm] (U) at (4,2.0) {};
              \path [->] (T) edge[bend right =60] node[above] {${\textcolor{red}{\kappa}}\beta/N^\theta$} (U);
              
                \path [->] (W) edge[bend left=60] node[above] {${\textcolor{red}{\kappa}}(1-\beta)/N^\theta$} (V);
              
               \path [->] (P) edge[bend right=60] node[above] {${\textcolor{red}{\kappa}}(1-\alpha)/N^\theta$} (Q);
               \path [->] (M) edge[bend left=60] node[above] {${\textcolor{red}{\kappa}}\alpha/N^\theta$} (N);
              
    \node[shape=circle,minimum size=0.6cm] (K) at (-1,0.4) {};
        \node[shape=circle,minimum size=0.6cm] (G) at (1,0.4) {};
\tclock{0.6}{-0.8}{columbiablue}{0.8}{0}
\tclock{1.9}{-0.8}{columbiablue}{0.8}{0}
\tclock{4.4}{-0.8}{columbiablue}{0.8}{0}
\tclock{-4.3}{-0.8}{columbiablue}{0.8}{0}
\tclock{-6.8}{-0.8}{columbiablue}{0.8}{0}
\tclock{5.6}{-0.8}{columbiablue}{0.8}{0}
\tclock{3.2}{-0.8}{columbiablue}{0.8}{0}
\tclock{-3.1}{-0.8}{columbiablue}{0.8}{0}
 \end{tikzpicture}
 \end{center}

Now let us suppose that the next clock to ring is associated to the oriented bond $\{0,1\}$.

  \begin{center}
 \begin{tikzpicture}[thick, scale=0.85][h!]
 \draw[latex-] (-6.5,0) -- (5.5,0) ;
\draw[-latex] (-6.5,0) -- (5.5,0) ;
\foreach \x in  {-6,-5,-4,-3,-2,-1,0,1,2,3,4,5}
\draw[shift={(\x,0)},color=black] (0pt,0pt) -- (0pt,-3pt) node[below] 
{};
\node[ball color=black!30!,shape=circle,minimum size=0.7cm] (A) at (1,0.4) {};
       \node[shape=circle,minimum size=0.6cm] (N) at (-5,1.2) {};
       \node[shape=circle,minimum size=0.6cm] (P) at (-5,2.0) {};
    \node[ball color=black!30!,shape=circle,minimum size=0.7cm] (C) at (-3,0.4) {};
   \node[ball color=black!30!,shape=circle,minimum size=0.7cm] (E) at (3,0.4) {};
    \node[shape=circle,minimum size=0.7cm] (E) at (3,0.4) {};
         \node[ball color=columbiablue,shape=circle,minimum size=0.7cm] (L) at (-6,0.4) {};
          \node[shape=circle,minimum size=0.7cm] (Q) at (-6,2.0) {};
           \node[ball color=columbiablue,shape=circle,minimum size=0.7cm] (M) at (-6,1.2) {};
           \node[ball color=columbiablue,shape=circle,minimum size=0.7cm] (R) at (5,0.4) {};
          \node[shape=circle,minimum size=0.7cm] (S) at (5,1.2) {};
          \node[ball color=columbiablue,shape=circle,minimum size=0.7cm] (R) at (5,0.4) {};
          \node[ball color=columbiablue,shape=circle,minimum size=0.7cm] (S) at (5,1.2) {};
              \node[shape=circle,minimum size=0.7cm] (T) at (5,2.0) {};
               \node[shape=circle,minimum size=0.7cm] (U) at (4,3.6) {};
                 \node[shape=circle,minimum size=0.7cm] (V) at (5,3.0) {};
                  \node[shape=circle,minimum size=0.7cm] (W) at (4,3.0) {};
\node[ball color=columbiablue,shape=circle,minimum size=0.7cm] (T) at (5,2.0) {};
              \node[shape=circle,minimum size=0.7cm] (U) at (4,2.0) {};
              \path [->] (T) edge[bend right =60] node[above] {${\textcolor{red}{\kappa}}\beta/N^\theta$} (U);
              
                \path [->] (W) edge[bend left=60] node[above] {${\textcolor{red}{\kappa}}(1-\beta)/N^\theta$} (V);
              
               \path [->] (P) edge[bend right=60] node[above] {${\textcolor{red}{\kappa}}(1-\alpha)/N^\theta$} (Q);
               \path [->] (M) edge[bend left=60] node[above] {${\textcolor{red}{\kappa}}\alpha/N^\theta$} (N);
              
    \node[shape=circle,minimum size=0.6cm] (K) at (-1,0.4) {};
        \node[shape=circle,minimum size=0.6cm] (G) at (1,0.4) {};
\tclock{0.6}{-0.8}{columbiablue}{0.8}{0}
\tclock{1.9}{-0.8}{columbiablue}{0.8}{0}
\tclock{4.4}{-0.8}{columbiablue}{0.8}{0}
\tclock{-4.3}{-0.8}{columbiablue}{0.8}{0}
\tclock{-6.8}{-0.8}{columbiablue}{0.8}{1}
\tclock{5.6}{-0.8}{columbiablue}{0.8}{0}
\tclock{3.2}{-0.8}{columbiablue}{0.8}{0}
\tclock{-3.1}{-0.8}{columbiablue}{0.8}{0}
 \end{tikzpicture}
 \end{center}

 Since there is no particle at the site $x=1$, a particle is injected into the system at the site $x=1$ with rate $\alpha\kappa N^{-\theta}$. See the figure below.

  \begin{center}
 \begin{tikzpicture}[thick, scale=0.85][h!]
 \draw[latex-] (-6.5,0) -- (5.5,0) ;
\draw[-latex] (-6.5,0) -- (5.5,0) ;
\foreach \x in  {-6,-5,-4,-3,-2,-1,0,1,2,3,4,5}
\draw[shift={(\x,0)},color=black] (0pt,0pt) -- (0pt,-3pt) node[below] 
{};
\node[ball color=black!30!,shape=circle,minimum size=0.7cm] (A) at (1,0.4) {};
       \node[shape=circle,minimum size=0.6cm] (N) at (-5,1.2) {};
       \node[shape=circle,minimum size=0.6cm] (P) at (-5,2.0) {};
    \node[ball color=black!30!,shape=circle,minimum size=0.7cm] (C) at (-3,0.4) {};
   \node[ball color=black!30!,shape=circle,minimum size=0.7cm] (E) at (3,0.4) {};
    \node[shape=circle,minimum size=0.7cm] (E) at (3,0.4) {};
         \node[ball color=columbiablue,shape=circle,minimum size=0.7cm] (L) at (-6,0.4) {};
          \node[shape=circle,minimum size=0.7cm] (Q) at (-6,2.0) {};
           \node[ball color=columbiablue,shape=circle,minimum size=0.7cm] (M) at (-6,1.2) {};
           \node[ball color=columbiablue,shape=circle,minimum size=0.7cm] (R) at (5,0.4) {};
          \node[shape=circle,minimum size=0.7cm] (S) at (5,1.2) {};
              \node[ball color=black!30!,shape=circle,minimum size=0.7cm] (AA) at (-5,0.4) {};
          \node[ball color=columbiablue,shape=circle,minimum size=0.7cm] (R) at (5,0.4) {};
          \node[ball color=columbiablue,shape=circle,minimum size=0.7cm] (S) at (5,1.2) {};
              \node[shape=circle,minimum size=0.7cm] (T) at (5,2.0) {};
               \node[shape=circle,minimum size=0.7cm] (U) at (4,3.6) {};
                 \node[shape=circle,minimum size=0.7cm] (V) at (5,3.0) {};
                  \node[shape=circle,minimum size=0.7cm] (W) at (4,3.0) {};
\node[ball color=columbiablue,shape=circle,minimum size=0.7cm] (T) at (5,2.0) {};
              \node[shape=circle,minimum size=0.7cm] (U) at (4,2.0) {};
              \path [->] (T) edge[bend right =60] node[above] {${\textcolor{red}{\kappa}}\beta/N^\theta$} (U);
              
                \path [->] (W) edge[bend left=60] node[above] {${\textcolor{red}{\kappa}}(1-\beta)/N^\theta$} (V);
              
               \path [->] (P) edge[bend right=60] node[above] {${\textcolor{red}{\kappa}}(1-\alpha)/N^\theta$} (Q);
               \path [->] (M) edge[bend left=60] node[above] {${\textcolor{red}{\kappa}}\alpha/N^\theta$} (N);
              
    \node[shape=circle,minimum size=0.6cm] (K) at (-1,0.4) {};
        \node[shape=circle,minimum size=0.6cm] (G) at (1,0.4) {};
\tclock{0.6}{-0.8}{columbiablue}{0.8}{0}
\tclock{1.9}{-0.8}{columbiablue}{0.8}{0}
\tclock{4.4}{-0.8}{columbiablue}{0.8}{0}
\tclock{-4.3}{-0.8}{columbiablue}{0.8}{0}
\tclock{-6.8}{-0.8}{columbiablue}{0.8}{0}
\tclock{5.6}{-0.8}{columbiablue}{0.8}{0}
\tclock{3.2}{-0.8}{columbiablue}{0.8}{0}
\tclock{-3.1}{-0.8}{columbiablue}{0.8}{0}
\tclock{-5.6}{-0.8}{columbiablue}{0.8}{0}
 \end{tikzpicture}
 \end{center}

Finally let us suppose that the next clock to ring is associated to the oriented bond $\{N,N-1\}$.  
   \begin{center}
 \begin{tikzpicture}[thick, scale=0.85][h!]
 \draw[latex-] (-6.5,0) -- (5.5,0) ;
\draw[-latex] (-6.5,0) -- (5.5,0) ;
\foreach \x in  {-6,-5,-4,-3,-2,-1,0,1,2,3,4,5}
\draw[shift={(\x,0)},color=black] (0pt,0pt) -- (0pt,-3pt) node[below] 
{};
\node[ball color=black!30!,shape=circle,minimum size=0.7cm] (A) at (1,0.4) {};
       \node[shape=circle,minimum size=0.6cm] (N) at (-5,1.2) {};
       \node[shape=circle,minimum size=0.6cm] (P) at (-5,2.0) {};
    \node[ball color=black!30!,shape=circle,minimum size=0.7cm] (C) at (-3,0.4) {};
   \node[ball color=black!30!,shape=circle,minimum size=0.7cm] (E) at (3,0.4) {};
    \node[shape=circle,minimum size=0.7cm] (E) at (3,0.4) {};
         \node[ball color=columbiablue,shape=circle,minimum size=0.7cm] (L) at (-6,0.4) {};
          \node[shape=circle,minimum size=0.7cm] (Q) at (-6,2.0) {};
           \node[ball color=columbiablue,shape=circle,minimum size=0.7cm] (M) at (-6,1.2) {};
           \node[ball color=columbiablue,shape=circle,minimum size=0.7cm] (R) at (5,0.4) {};
          \node[shape=circle,minimum size=0.7cm] (S) at (5,1.2) {};
              \node[ball color=black!30!,shape=circle,minimum size=0.7cm] (AA) at (-5,0.4) {};
          \node[ball color=columbiablue,shape=circle,minimum size=0.7cm] (R) at (5,0.4) {};
          \node[ball color=columbiablue,shape=circle,minimum size=0.7cm] (S) at (5,1.2) {};
              \node[shape=circle,minimum size=0.7cm] (T) at (5,2.0) {};
               \node[shape=circle,minimum size=0.7cm] (U) at (4,3.6) {};
                 \node[shape=circle,minimum size=0.7cm] (V) at (5,3.0) {};
                  \node[shape=circle,minimum size=0.7cm] (W) at (4,3.0) {};
\node[ball color=columbiablue,shape=circle,minimum size=0.7cm] (T) at (5,2.0) {};
              \node[shape=circle,minimum size=0.7cm] (U) at (4,2.0) {};
              \path [->] (T) edge[bend right =60] node[above] {${\textcolor{red}{\kappa}}\beta/N^\theta$} (U);
              
                \path [->] (W) edge[bend left=60] node[above] {${\textcolor{red}{\kappa}}(1-\beta)/N^\theta$} (V);
              
               \path [->] (P) edge[bend right=60] node[above] {${\textcolor{red}{\kappa}}(1-\alpha)/N^\theta$} (Q);
               \path [->] (M) edge[bend left=60] node[above] {${\textcolor{red}{\kappa}}\alpha/N^\theta$} (N);
              
    \node[shape=circle,minimum size=0.6cm] (K) at (-1,0.4) {};
        \node[shape=circle,minimum size=0.6cm] (G) at (1,0.4) {};
\tclock{0.6}{-0.8}{columbiablue}{0.8}{0}
\tclock{1.9}{-0.8}{columbiablue}{0.8}{0}
\tclock{4.4}{-0.8}{columbiablue}{0.8}{0}
\tclock{-4.3}{-0.8}{columbiablue}{0.8}{0}
\tclock{-6.8}{-0.8}{columbiablue}{0.8}{0}
\tclock{5.6}{-0.8}{columbiablue}{0.8}{1}
\tclock{3.2}{-0.8}{columbiablue}{0.8}{0}
\tclock{-3.1}{-0.8}{columbiablue}{0.8}{0}
\tclock{-5.6}{-0.8}{columbiablue}{0.8}{0}
 \end{tikzpicture}
 \end{center}
 
 \quad 
 
Since there is no particle at the site $x=N-1$, a particle is injected into the system at the site $x=N-1$ with rate $\beta\kappa N^{-\theta}$. See the figure below.
    \begin{center}
 \begin{tikzpicture}[thick, scale=0.85][h!]
 \draw[latex-] (-6.5,0) -- (5.5,0) ;
\draw[-latex] (-6.5,0) -- (5.5,0) ;
\foreach \x in  {-6,-5,-4,-3,-2,-1,0,1,2,3,4,5}
\draw[shift={(\x,0)},color=black] (0pt,0pt) -- (0pt,-3pt) node[below] 
{};
\node[ball color=black!30!,shape=circle,minimum size=0.7cm] (A) at (1,0.4) {};
\node[ball color=black!30!,shape=circle,minimum size=0.7cm] (AAA) at (4,0.4) {};
       \node[shape=circle,minimum size=0.6cm] (N) at (-5,1.2) {};
       \node[shape=circle,minimum size=0.6cm] (P) at (-5,2.0) {};
    \node[ball color=black!30!,shape=circle,minimum size=0.7cm] (C) at (-3,0.4) {};
   \node[ball color=black!30!,shape=circle,minimum size=0.7cm] (E) at (3,0.4) {};
    \node[shape=circle,minimum size=0.7cm] (E) at (3,0.4) {};
         \node[ball color=columbiablue,shape=circle,minimum size=0.7cm] (L) at (-6,0.4) {};
          \node[shape=circle,minimum size=0.7cm] (Q) at (-6,2.0) {};
           \node[ball color=columbiablue,shape=circle,minimum size=0.7cm] (M) at (-6,1.2) {};
           \node[ball color=columbiablue,shape=circle,minimum size=0.7cm] (R) at (5,0.4) {};
          \node[shape=circle,minimum size=0.7cm] (S) at (5,1.2) {};
              \node[ball color=black!30!,shape=circle,minimum size=0.7cm] (AA) at (-5,0.4) {};
          \node[ball color=columbiablue,shape=circle,minimum size=0.7cm] (R) at (5,0.4) {};
          \node[ball color=columbiablue,shape=circle,minimum size=0.7cm] (S) at (5,1.2) {};
              \node[shape=circle,minimum size=0.7cm] (T) at (5,2.0) {};
               \node[shape=circle,minimum size=0.7cm] (U) at (4,3.6) {};
                 \node[shape=circle,minimum size=0.7cm] (V) at (5,3.0) {};
                  \node[shape=circle,minimum size=0.7cm] (W) at (4,3.0) {};
\node[ball color=columbiablue,shape=circle,minimum size=0.7cm] (T) at (5,2.0) {};
              \node[shape=circle,minimum size=0.7cm] (U) at (4,2.0) {};
              \path [->] (T) edge[bend right =60] node[above] {${\textcolor{red}{\kappa}}\beta/N^\theta$} (U);
              
                \path [->] (W) edge[bend left=60] node[above] {${\textcolor{red}{\kappa}}(1-\beta)/N^\theta$} (V);
              
               \path [->] (P) edge[bend right=60] node[above] {${\textcolor{red}{\kappa}}(1-\alpha)/N^\theta$} (Q);
               \path [->] (M) edge[bend left=60] node[above] {${\textcolor{red}{\kappa}}\alpha/N^\theta$} (N);
              
    \node[shape=circle,minimum size=0.6cm] (K) at (-1,0.4) {};
        \node[shape=circle,minimum size=0.6cm] (G) at (1,0.4) {};
\tclock{0.6}{-0.8}{columbiablue}{0.8}{0}
\tclock{1.9}{-0.8}{columbiablue}{0.8}{0}
\tclock{4.4}{-0.8}{columbiablue}{0.8}{0}
\tclock{-4.3}{-0.8}{columbiablue}{0.8}{0}
\tclock{-6.8}{-0.8}{columbiablue}{0.8}{0}
\tclock{5.6}{-0.8}{columbiablue}{0.8}{0}
\tclock{3.2}{-0.8}{columbiablue}{0.8}{0}
\tclock{-3.1}{-0.8}{columbiablue}{0.8}{0}
\tclock{-5.6}{-0.8}{columbiablue}{0.8}{0}
 \end{tikzpicture}
 \end{center}
\quad

  We note that the bulk dynamics conserves the total number of particles in the bulk, but the boundary dynamics destroys this quantity since it injects/removes particles in/from the bulk.
  \subsection{Infinitesimal generator}
\label{sec:generator}
 The dynamics described above is Markovian and can be completely characterized by mean of its infinitesimal generator, see \cite{Li,Liggett2}.
 The Markov process  $\{\eta_t\,:\, t\geq 0\}$ whose dynamics we have just defined has infinitesimal generator denoted by $\mathcal L_N$ which  is expressed as 
\begin{equation}\label{generator_ssep}
\mathcal L_{N}=\mathcal L_{N,0}+\mathcal L_{N,b},
\end{equation}
where  $\mathcal L_{N,0}$ and $\mathcal L_{N,b}$ are given on functions $f:\Omega_N\rightarrow \bbbr$ by  
\begin{equation*}\label{lnb}
\begin{split}
(\mathcal L_{N,0}f)(\eta)\;=\;
\sum_{x=1}^{N-2}\frac 12 \Big(f(\eta^{x,x+1})-f(\eta)\Big)\,, 
\end{split}
\end{equation*}
\begin{equation}\label{generator_ssep_boundary}
\mathcal L_{N,b}=\mathcal L_{N,b}^1+\mathcal L_{N,b}^{N-1},
\end{equation}
where  for $x\in\{1,N-1\}$
\begin{equation*}\label{lno}
(\mathcal L_{N,b}^xf)(\eta)\;=\;\frac{{\kappa}}{N^\theta}c_{x}(\eta,r(x))\Big(f( \eta^x)-f(\eta)\Big)\,,
\end{equation*}
  $r(1)=\alpha$ and  $r(N-1)=\beta$,
\begin{equation}\label{tranformations}
(\eta^{x,y})(z) = 
\begin{cases}
\eta(z), \; z \ne x,y,\\
\eta(y), \; z=x,\\
\eta(x), \; z=y
\end{cases}
, \quad (\eta^x)(z)= 
\begin{cases}
\eta(z), \; z \ne x,\\
1-\eta(x), \; z=x,
\end{cases}
\end{equation}
and for $x\in\{1,N-1\}$
\begin{equation}\label{rate_c}
c_{x} (\eta;r(x)) :=\frac 12\left[ \eta(x)  \left(1-r(x) \right) + (1-\eta(x))r(x)\right].
\end{equation}

Note that the generator above splits into the sum of the generator $\mathcal L_{N,0}$ (which is related to the jumps in the bulk) and $\mathcal L_{N,b}$ (which is related to the jumps from the boundary or from the reservoirs). We will refer to the first one as the \emph{exchange dynamics} and the latter one as the \emph{flip dynamics}, because in $\mathcal L_{N,0}$ we exchange the occupation variables $\eta(x)$ and $\eta(x+1)$ and in $\mathcal L_{N,b}^x$ we flip the value of the occupation variable at $\eta(x)$.

We consider  the Markov process  speeded up in the  time scale $\Theta(N)$ and we note that the process  $\{\eta_{t\Theta(N)}\,:\,t\ge 0\} $ has  infinitesimal generator  given by $\Theta(N)\mathcal L_{N}$. To see this relation, let $\tilde{\mathcal L}_N$ be the generator of the process  $\{\eta_{t\Theta(N)}\,:\,t\ge 0\} $. By definition, for $f:\Omega_N\to\bbbr$,  we have that
\begin{equation}
\tilde {\mathcal L}_Nf =\lim_{s\to 0} \frac{\tilde{S}_{s}f-f}{s},
\end{equation}
where $\tilde{S}_s:=S_{s\Theta(N)}$ is the semigroup associated to $\tilde{ \mathcal L}_N$ and $S_s$  is the semigroup associated to $\mathcal L_N$. Then, 
\begin{equation}\begin{split}
\Theta (N)\mathcal L_N f =&\lim_{t\to 0} \Theta(N)\frac{S_{t}f-f}{t}
=\lim_{s\to 0}\Theta(N)\frac{S_{s\Theta(N)}f-f}{s\Theta (N)}=\tilde {\mathcal L}_Nf,
\end{split}
\end{equation} from where  we conclude that $\tilde{\mathcal L}_N:=\Theta(N)\mathcal L_N$.

We note that  $\eta_{t\theta(N)}$ depends on $\alpha$, $\beta$, $\theta$ and $\kappa $ but we will omit these indexes in order to simplify notation.  Fix $T>0$ and $\theta\in \bbbr$. Let $\mu_N$ be a probability measure in $\Omega_N.$ We denote by $\bbbp_{\mu _{N}}$ the probability measure in the Skorohod space $\mathcal D([0,T], \Omega_N)$ induced by the  Markov process $\{\eta_{t\Theta(N)}\,:\,t\geq{0}\}$ and the initial probability measure $\mu_N$ and we denote by $ {E}_{\bbbp_{\mu _{N}}}$ the expectation with respect to $\bbbp_{\mu _{N}}$.

Our goal in these notes is to analyse the impact of changing the strength of the reservoirs (by changing the value of  $\theta$) on the macroscopic behavior of the system. More precisely, we want to obtain the hydrodynamic equations of the process which will have different boundary conditions depending on the range of the parameter $\theta$ which rules the strength of the reservoirs. 
Before proceeding, in the next subsection we analyse the invariant measures for this model.

\subsection{Stationary measures}
 \label{sec:stat_mea}

For $\rho\in(0,1)$ we denote by $\nu^N_\rho$ the Bernoulli product measure in $\Omega_N$ with density $\rho$, that is, for $x\in\Lambda_N$:
\begin{equation}\label{eq:bernoulli_mea}
\nu^N_\rho\{\eta: \eta(x)=1\}=\rho.
\end{equation}
According to this measure the occupation variables $\{\eta(x)\}_{x\in\Lambda_N}$ are independent and for each $x\in\Lambda_N$ the random variable $\eta(x)$ has Bernoulli distribution of parameter $\rho$. 
When we restrict the parameters $\alpha$ and $\beta$ such that $\alpha=\beta=\rho$, then these measures are invariant for the dynamics described above. In fact, a stronger result is true, see the next lemma where we prove that that these measures are reversible.

\begin{lemma}\label{lem:bern_rev}
For $\alpha=\beta=\rho$ the Bernoulli product measures $\nu^N_\rho$ are reversible.
\end{lemma}

\begin{proof}
Fix two functions $f,g:\Omega_N\to\bbbr$. To prove the lemma, we need to show  that
\begin{equation}\label{eq:rev_cond}
\int_{\Omega_N}g(\eta)\mathcal L_Nf(\eta)d\nu^N_{\rho}=\int_{\Omega_N}f(\eta)\mathcal L_Ng(\eta)d\nu^N_{\rho}.
\end{equation}
Let us start with the exchange dynamics given by $\mathcal L_{N,0}$. In this case we need to check that
\begin{equation*}
\sum_{x\in\Lambda_N}\int_{\Omega_N}g(\eta)(f(\eta^{x,x+1})-f(\eta))d\nu^N_{\rho}=\sum_{x\in\Lambda_N}\int_{\Omega_N}f(\eta)(g(\eta^{x,x+1})-g(\eta))d\nu^N_{\rho}.
\end{equation*}
For that purpose note that, for fixed $x\in\Lambda_N$ and performing a change of variables $\xi=\eta^{x,x+1}$, we have that 
\begin{equation*}
\begin{split}
\int_{\Omega_N}g(\eta)f(\eta^{x,x+1})d\nu^N_{\rho}=&\sum_{\eta\in\Omega_N}g(\eta)f(\eta^{x,x+1})\nu^N_\rho(\eta)\\=&\sum_{\xi\in\Omega_N}g(\xi^{x,x+1})f(\xi)\frac{\nu^N_{\rho}(\xi^{x,x+1})}{\nu^N_\rho(\xi)}\nu^N_\rho(\xi).
\end{split}
\end{equation*}
Now note that
\begin{equation*}
\nu^N_\rho(\xi)=\prod_{x\in\Lambda_N}\rho^{\xi(x)}(1-\rho)^{1-\xi(x)}
\end{equation*}
so that 
\begin{itemize}
\item if $\xi(x)=1$ and $\xi(x+1)=0$,  denoting by  $\tilde \xi$ the configuration $\xi$ removing its values at $x$ and $x+1$ so that $\xi=(\tilde \xi, \xi(x),\xi(x+1))$, then $\nu^N_\rho(\xi)=\nu^N_\rho(\tilde \xi)\rho(1-\rho)$ and $\nu^N_\rho(\xi^{x,x+1})=\nu^N_\rho(\tilde \xi)(1-\rho)\rho$, so that 
\begin{equation}\label{price}\frac{\nu^N_{\rho}(\xi^{x,x+1})}{\nu^N_\rho(\xi)}=1.\end{equation}
\item  if $\xi(x)=0$ and $\xi(x+1)=1$, then $\nu^N_\rho(\xi)=\nu^N_\rho(\tilde \xi)(1-\rho)\rho$ and $\nu^N_\rho(\xi^{x,x+1})=\nu^N_\rho(\tilde \xi)\rho(1-\rho)$, so that \eqref{price} is also true.
\end{itemize}
Therefore,  we obtain that 
\begin{equation*}
\begin{split}
\int_{\Omega_N}g(\eta)f(\eta^{x,x+1})d\nu^N_{\rho}=\sum_{\xi\in\Omega_N}g(\xi^{x,x+1})f(\xi)\nu^N_\rho(\xi)=\int_{\Omega_N}g(\eta^{x,x+1})f(\eta)d\nu^N_{\rho},
\end{split}
\end{equation*} which proves \eqref{eq:rev_cond} for $\mathcal L_{N,0}.$ For the flip dynamics given by $\mathcal L_{N,b}$ we note, for the left boundary, that
\begin{equation*}
\begin{split}
\int_{\Omega_N}&g(\eta)c_1(\eta,\alpha)f(\eta^{1})d\nu^N_{\rho}\\=&\sum_{\eta\in\Omega_N}g(\eta)(1-\eta(1))\alpha f(\eta^{1})\nu^N_\rho(\eta)+\sum_{\eta\in\Omega_N}g(\eta)(1-\alpha)f(\eta^{1})\nu^N_\rho(\eta).
\end{split}
\end{equation*}
By the change of variables $\xi=\eta^1$, the previous expression can be written as 
\begin{equation*}
\sum_{\xi\in\Omega_N}f(\xi)\Big\{g(\xi^1)\xi(1)\alpha \frac{\nu^N_\rho(\xi^1)}{\nu^N_\rho(\xi)}+ g(\xi^1)(1-\xi(1))(1-\alpha) \frac{\nu^N_\rho(\xi^1)}{\nu^N_\rho(\xi)}\Big\}\nu^N_\rho(\xi).
\end{equation*}
A simple computation shows that if $\xi(1)=1$, then $\frac{\nu^N_\rho(\xi^1)}{\nu^N_\rho(\xi)}=\frac{1-\rho}{\rho}$ so that the previous expression can be written as
\begin{equation*}
\frac{\kappa}{N^\theta}\sum_{\xi\in\Omega_N}f(\xi)\Big\{g(\xi^1)\xi(1)\alpha \frac{1-\rho}{\rho}+g(\xi^1)(1-\xi(1))(1-\alpha)\frac{\rho}{1-\rho}\Big\}\nu^N_\rho(\xi),
\end{equation*}
from where we get, for $\alpha=\rho$, that
\begin{equation*}
\begin{split}
\int_{\Omega_N}g(\eta)c_1(\eta,\alpha)f(\eta^{1})d\nu^N_{\rho}=\int_{\Omega_N}g(\eta^1)c_1(\eta^1,\rho)f(\eta)d\nu^N_{\rho}.
\end{split}
\end{equation*}
The same computation can be done if $\xi(1)=0$, from where we conclude.
We can repeat the same computation for the right boundary and this  proves \eqref{eq:rev_cond} for $\mathcal{L}_{N,b}.$ This ends the proof of the lemma. \qed
\end{proof}

When $\alpha\neq \beta$, the Bernoulli product measures are not reversible nor invariant. A simple way to check the non-invariance is to argue as follows. Suppose that the measures $\nu_\rho^N$  are invariant. Then we know that for any function $f:\Omega_N\to\bbbr$ we have that
\begin{equation}\label{inv_cond}
\int_{\Omega_N}\mathcal L_N f(\eta)d\nu^N_\rho=0.
\end{equation}
But for $f(\eta)=\eta(1)$, a simple computation shows that
$\mathcal L_{N,0}f(\eta)=\frac 12(\eta(2)-\eta(1))$ and $\mathcal L_{N,b}f(\eta)=\frac{\kappa}{2N^\theta}[\alpha-\eta(1)]$, so that $$\int_{\Omega_N}\mathcal L_N f(\eta)d\nu^N_\rho=\frac{\kappa}{2N^\theta}(\alpha-\rho)$$ and this equals to $0$ iff $\alpha=\rho$. Analogously, repeating the same computations as above for $f(\eta)=\eta(N-1),$ we would conclude \eqref{inv_cond} iff $\beta=\rho$. But this contradicts the fact that $\alpha\neq \beta$.

When $\alpha\neq \beta$, since we have a finite state irreducible Markov process, then there exists a unique stationary measure that we denote by $\mu_{ss}.$ A way to get information about this measure is to use the matrix ansatz method introduced in \cite{Derrida1,Derrida2}.  The idea behind the method is the following. Let  $$f_{N-1}(\eta(1),\cdots,\eta(N-1))$$ denote the weight of the configuration  $\eta:=(\eta(1),\cdots,\eta(N-1))$ with respect to the stationary measure $\mu_{ss}$ and let us suppose that 
\begin{equation*}\label{eq:matrix_function}
f_{N-1}(\eta(1),\eta(2),\cdots,\eta(N-1)) = \textbf{w}^{T} X_{\eta(1)}X_{\eta(2)}\cdots X_{\eta(N-1)}\textbf{v},
\end{equation*}
where
\begin{equation*}\label{eq:matrix}
X_{\eta(x)} = \eta(x)D+(1-\eta(x))E,
\end{equation*}
and $D,E$ are matrices (which in general do not commute) and the vectors $\textbf{w}^{T},\textbf{v}$ are present in order to convert the matrix product into a scalar.  In the figure below we take $N=6 $ and we present a possible configuration $\eta=(0,1,0,1,1)$ whose corresponding weight is given by $f_{N-1}(\eta) = \textbf{w}^{T} EDEDD\textbf{v}.$
\begin{figure}[!htb]
\begin{center}
 \begin{tikzpicture}[thick, scale=0.85]
 \draw (-4.0,0) -- (2.0,0) ; 
\foreach \x in  {-4,-3,-2,-1,0,1,2} 
\draw[shift={(\x,0)},color=black] (0pt,3pt) -- (0pt,-3pt) node[below] {};

\draw[shift={(-4,-0.1)},color=black] (0pt,3pt) -- (0pt,-3pt) node[below] {$\downarrow$};
\draw (-4,-1) node[] {$\textbf{w}^{T}$};
	
\draw[shift={(-3,-0.1)},color=black] (0pt,3pt) -- (0pt,-3pt) node[below] {$\downarrow$};
\draw (-3,-1) node[] {$E$};	

\draw[shift={(-2,-0.1)},color=black] (0pt,3pt) -- (0pt,-3pt) node[below] {$\downarrow$};
\draw (-2,-1) node[] {$D$};	

\draw[shift={(-1,-0.1)},color=black] (0pt,3pt) -- (0pt,-3pt) node[below] {$\downarrow$};
\draw (-1,-1) node[] {$E$};

\draw[shift={(0,-0.1)},color=black] (0pt,3pt) -- (0pt,-3pt) node[below] {$\downarrow$};
\draw (0,-1) node[] {$D$};

\draw[shift={(1,-0.1)},color=black] (0pt,3pt) -- (0pt,-3pt) node[below] {$\downarrow$};
\draw (1,-1) node[] {$D$};

\draw[shift={(2,-0.1)},color=black] (0pt,3pt) -- (0pt,-3pt) node[below] {$\downarrow$};
\draw (2,-1) node[] {$\textbf{v}$};
  
    \node[shape=circle,draw=white,minimum size=0.4cm] 
(K1) at (-3.0,0.4) {};
    \node[ball color=black!30!,shape=circle,minimum size=0.7cm] (B) at (-2.0,0.4) {};
    \node[shape=circle,draw=white,minimum size=0.4cm] 
(K2) at (0.0,0.4) {};
 \node[ball color=black!30!,shape=circle,minimum size=0.7cm] (D) at (0,0.4) {};
    
    \node[ball color=black!30!,shape=circle,minimum size=0.7cm] (D) at (1.0,0.4) {};
\end{tikzpicture}
\end{center}
\end{figure}
Let $P(\eta(1),\eta(2),\cdots,\eta(N-1))$ be the normalized weight of the configuration $\eta:=(\eta(1),\cdots,\eta(N-1))$ with respect to the stationary state $\mu_{ss}$, which is given by
\begin{equation*}\label{eq:m_a_stat}
P(\eta(1),\eta(2),\cdots,\eta(N-1)) = \frac{f_{N-1}(\eta(1),\eta(2),\cdots,\eta(N-1))}{Z_{N-1}},
\end{equation*}
where $Z_{N-1}$ is the sum of the weights of the $2^{N-1}$ possible configurations in  $\Omega_N$:
\begin{equation*}\label{eq:ZN1}
Z_{N-1}= \sum_{\eta(1)\in \{ 0,1\} }\cdots\sum_{\eta(N-1)\in \{0,1\} }f_{N-1}(\eta(1),\eta(2),\cdots,\eta(N-1)).
\end{equation*}
From the definition of $f_{N-1}$, we have that
\begin{equation*}\label{eq:m_a_stat_2}
P(\eta(1), \eta(2),\cdots,\eta(N-1))=\frac{\textbf{w}^{T} X_{\eta(1)}X_{\eta(2)}\cdots X_{\eta(N-1)}\textbf{v}}{Z_{N-1}},
\end{equation*}
and the normalization  can be written as 
\begin{equation}\label{eq:norm}
\begin{split}
Z_{N-1} =& \; \sum_{\eta(1)\in \{ 1,0 \}}\cdots \sum_{\eta(N-1)\in \{ 1,0 \}}\textbf{w}^{T} X_{\eta(1)}X_{\eta(2)}\cdots X_{\eta(N-1)}\textbf{v}\\
=& \sum_{\eta(1)\in \{ 1,0 \}}\cdots \sum_{\eta(N-2)\in \{ 1,0 \}}\textbf{w}^{T} X_{\eta(1)}X_{\eta(2)}\cdots X_{\eta(N-2)}(D+E)\textbf{v} \\
=& \cdots = \textbf{w}^{T}(D+E)^{N-1}\textbf{v}.
\end{split}
\end{equation}
Let us now impose  conditions on the matrices $D$ and $E$. For that purpose, let $C=D+E$. The expectation of the  occupation variable at the site $x$, with respect to the stationary state $\mu_{ss}$, is given by 
\begin{equation}\label{eq:dens_ssep}
\begin{split}
\rho_{ss}^N(x)&=\int_{\Omega_N}\eta(x)d\mu_{ss} \\
&= \frac{\sum_{\eta(1)\in \{1,0\} }\cdots \sum_{\eta(N-1)\in \{1,0\} }\eta(x)f_{N-1}(\eta(1),\cdots,\eta(N-1))}{Z_{N-1}}\\
 &=\frac{1}{Z_{N-1}} \sum_{\eta(1)\in \{ 1,0 \} }\cdots \sum_{\eta(N-1)\in \{ 1,0 \}}\Big[\textbf{w}^{T}\prod_{j=1}^{x-1}X_{\eta(j)}D \prod_{j=x+1}^{N-1}X_{\eta(j)}\textbf{v}\Big] \\ 
&= \frac{\textbf{w}^{T}C^{x-1}DC^{N-1-x}\textbf{v}}{{\textbf{w}^{T}C^{N-1}\textbf{v}}}.
\end{split}
\end{equation} 
The function $\rho_{ss}^N(\cdot)$ is called the stationary empirical density profile since it is the average with respect to the stationary measure $\mu_{ss}$, otherwise we refer to it as the empirical density profile. 
Note that above the sum does not contain the factor  $\eta(x)\in \{1,0\}$ since the expectation is non-zero iff   $\eta(x)=1$.
We can also compute the expectation of the product of two point occupation variables at the sites $x$ and $y$, with respect to the stationary state $\mu_{ss}$, that is, for  $1\leq x < y \leq N-1$, we have that
\begin{equation}\label{eq:esp_matrix}
\begin{split}
\int_{\Omega_N}\eta(x)&\eta(y)d\mu_{ss}  =\\& =\frac{\sum_{\eta(1)\in \{ 0,1 \}}\cdots\sum_{\eta(N-1)\in \{ 0,1 \}}\eta(x)\eta(y)f_{N-1}(\eta(1),\cdots,\eta(N-1))}{Z_{N-1}} \nonumber \\
&=  \frac{\textbf{w}^{T}C^{x-1}DC^{y-x-1}DC^{N-1-y}\textbf{v}}{\textbf{w}^{T}C^{N-1}\textbf{v}}.
\end{split}
\end{equation}
Therefore,  the two point correlation function, with respect to the stationary state $\mu_{ss}$, is given   on $1\leq x < y \leq N-1$   by
\begin{equation}\label{eq:corr_matrix}
\begin{split}
\varphi_{ss}^N(x,y):=&\int_{\Omega_N}(\eta(x)-\rho_{ss}^N(x))(\eta(y)-\rho_{ss}^N(y))d\mu_{ss}\\
=& \;\; \frac{\textbf{w}^{T}C^{x-1}DC^{y-x-1}DC^{N-1-y}\textbf{v}}{\textbf{w}^{T}C^{N-1}\textbf{v}} \\
-& \;\; \frac{\textbf{w}^{T}C^{x-1}DC^{N-1-x}\textbf{v}}{\textbf{w}^{T}C^{N-1}\textbf{v}}.\frac{\textbf{w}^{T}C^{y-1}DC^{N-1-y}\textbf{v}}{\textbf{w}^{T}C^{N-1}\textbf{v}}.
\end{split}
\end{equation} 
A simple computation (see \cite{de_Paula}) shows that for the dynamics that we are considering in this section, the matrices $D,E$ and the vectors $\textbf{w}^{T}, \textbf{v}$ satisfy the  following relations:
\begin{equation}\label{eq:algSLOW}
\begin{split}
&DE-ED=D+E=C,\\
&\textbf{w}^{T}\left[ \frac{\kappa \alpha}{2N^{\theta}}E-\frac{\kappa(1-\alpha)}{2N^{\theta}}D \right] = \textbf{w}^{T}, \\ 
&\left[ \frac{\kappa(1-\beta)}{2N^{\theta}}D - \frac{\kappa \beta}{2N^{\theta}}E \right]\textbf{v} = \textbf{v}.
\end{split}
\end{equation}
We note that the equations above also show that  $$C(D+I)=(D+E)(D+I) = DD+D+ED+E,$$  and that $C(D+I)=DD+DE=DC$. Analogously we have that  $CD = (D-I)C$.
Using \eqref{eq:norm}, we obtain that $Z_{N-1}$  is given by 
\begin{equation*}
Z_{N-1} = \frac{1}{(\alpha-\beta)^{N-1}}\frac{\Gamma(2N^{\theta}+N-1)}{\Gamma(2N^{\theta})},
\end{equation*}
where  $\Gamma(\cdot)$ denotes the Gamma function.
For the details on these computations we refer the interested reader to \cite{de_Paula}. Now, in \eqref{eq:dens_ssep},  by writing $DC^{N-1-x}=DC C^{N-2-x}$ and using the fact that $C(D+I)=DC$ we obtain 
\begin{equation*}
\rho_{ss}^N(x) = \,\,\frac{\textbf{w}^{T}C^{x-1}C(D+I)C^{N-2-x}\textbf{v}}{Z_{N-1}} 
=\frac{\textbf{w}^{T}C^{x}{DC^{N-2-x}}\textbf{v}}{Z_{N-1}} + \frac{\textbf{w}^{T}C^{N-2}\textbf{v}}{Z_{N-1}}.
\end{equation*}
Repeating the  procedure above  and using the explicit expression for $Z_{N-1}$ given above, we obtain a simple expression for $\rho^N_{ss}(x)$ given by 
\begin{equation}\label{eq:dens_SLOW}
\rho_{ss}^N(x) =\beta + (N-x)\frac{\alpha-\beta}{2N^{\theta}+N-2} + (N^{\theta}-1)\frac{\alpha-\beta}{2N^{\theta}+N-2}.
\end{equation}
In fact last identity can be rewritten as 
\begin{equation*}
\rho_{ss}^N(x) =\frac{\kappa(\beta-\alpha)x}{2N^{\theta}+N-2} +\alpha+ \frac{\kappa(\beta-\alpha)x}{2N^{\theta}+N-2}\Big(\frac{N^\theta}{\kappa}-1\Big).
\end{equation*}
Analogously, from a simple, but long computation (see  \cite{de_Paula}), we have that 
\begin{equation*}
\begin{split}
\int_{\Omega_N}\eta(x)\eta(y)d\mu_{ss} =\beta \rho_{ss}^N(x)+ (N-y+N^{\theta}-1)\frac{\alpha-\beta}{2N^{\theta}+N-2}\rho_{ss}^{N-1}(x),
\end{split}
\end{equation*}
and from \eqref{eq:dens_SLOW}, we obtain
\begin{equation*}\label{eq:dens_XY_SLOW}
\begin{split}
\int_{\Omega_N}\eta(x)&\eta(y)d\mu_{ss}= \beta\left[ \frac{\beta(x+N^{\theta}-1)+\alpha(N-x+N^{\theta}-1)}{2N^{\theta}+N-2} \right]\\
+&  \frac{(N-y+N^{\theta}-1)(\alpha-\beta)}{2N^{\theta}+N-2}\left[ \frac{\beta(x+N^{\theta}-1)+\alpha(N-x+N^{\theta}-2)}{2N^{\theta}+N-3} \right].
\end{split}
\end{equation*}
Putting together last expresssions and doing simple, but long, computations we conclude that
\begin{equation}\label{eq:corr_est_stat}
\begin{split}
\varphi_{ss}^N(x,y)=-\frac{(\alpha-\beta)^{2}(x+N^{\theta}-1)(N-y+N^{\theta}-1)}{(2N^{\theta}+N-2)^{2}(2N^{\theta}+N-3)}.
\end{split}
\end{equation}
From the previous identity it follows that
\begin{equation}\label{eq:bound_stat_corr}
\max_{x<y}|\varphi_{ss}^N(x,y)| = 
\begin{cases}
O\Big(\frac{N^\theta}{N^2}\Big), \; \theta<1,\\
O\Big(\frac {1}{N^\theta}\Big), \; \theta\geq 1,
\end{cases}\to_{N\to\infty}0.
\end{equation}
This means that as the size of the bulk tens to infinity, the two point correlation function vanishes. In the next subsection we analyse the empirical profile and the two point correlation function for more general initial measures.

\subsection{Empirical profile and  correlations}
\label{sec:empirical_prof}
Before stating the hydrodynamic limit result we explain here how to have a guess on the form of the hydrodynamic equations by using the \emph{empirical profile}, which was defined above in the case of the measure $\mu_{ss}$. Now we generalize its definition.  For a measure $\mu_N$ in $\Omega_N$ and for each $x\in\Lambda_N$ we denote by $\rho_t^N(x)$ the empirical profile at the site $x$, given by

\begin{equation*}\label{eq:rho_t}
\rho^N_t(x)\;=\;{E}_{\bbbp_{\mu_N}}[\eta_{tN^2}(x)]\,.
\end{equation*}
We extend this definition to the boundary by setting 
 \begin{equation*}\label{eq:ext_bound}
 \rho^N_t(0)\;=\;\alpha\mbox{ and }\rho^N_t(N)\;=\;\beta\,, \mbox{ for all }t\geq 0\,. 
\end{equation*} 
Note that since $\mu_{ss}$ is a stationary measure the stationary empirical profile $\rho_{ss}^N(\cdot)$
 does not depend on time, but now since $\mu_N$ is a general measure the empirical profile $\rho^N_t(\cdot)$ depends on time. 
 From Kolmogorov's backward equation we know that $\rho_t^N(\cdot)$ is a solution of 
 $$ \partial_t \rho_t^N(x)= {E}_{\bbbp_{\mu_N}}[\mathcal L_N\eta_{tN^2}(x)].$$ A simple computation shows that $$\mathcal L_N\eta(x)=j_{x-1,x}(\eta)-j_{x,x+1}(\eta)$$ where 
 for $x\in\Lambda_N$, the quantity $j_{x,x+1}(\eta)$ denotes the microscopic current at the bond $\{x,x+1\}$, which is given by the difference between the jump rate from $x$ to $x+1$ and the jump rate from $x+1$ to $x$. Note that for $x=0$ (resp. $x=N-1$) $j_{x,x+1}$ is equal to the creation rate minus the annihilation rate   at the site $x=1$ (resp. $x=N-1$).  Therefore
 \begin{equation}
 \begin{split}
& j_{0,1}(\eta)=\frac{\kappa}{2N^\theta}(\alpha-\eta(1)),\\
& j_{x,x+1}(\eta)=\frac{1}{2}(\eta(x)-\eta(x+1)), \forall x\in\{1,...,N-2\}\\
& j_{N-1,N}(\eta)=\frac{\kappa}{2N^\theta}(\eta(N-1)-\beta).
 \end{split}
 \end{equation}
  A simple computation shows that $\rho_t^N(\cdot)$ is a solution of the equation
\begin{equation}\label{eq:disc_heat}
\left\{
\begin{array}{ll}
 \partial_t \rho_t^N(x) \;= \; \big(N^2\mathcal B^\theta_N \rho_t^N\big)(x)\,, \;\; x\in\Lambda_N\,,\;\;t \geq 0\,,\\
 \rho_t^N(0)\;=\alpha\, \,,\;\;t \geq 0,\\ \rho^N_t(N)\;=\;\beta\,,\;\;t \geq 0,
\end{array}
\right.
\end{equation}
 where the operator $\mathcal B_N^\theta$ acts on functions $f:\Lambda_N\cup \{0,N\}\to\bbbr$ as
\begin{equation*}\label{disc_heat}
\left\{
\begin{array}{ll}
N^2(\mathcal B^\theta_Nf)(x)\;=\;\frac 12\Delta_Nf(x)\,, ~~\textrm{ for } x\in\{2,...,N-2\},\\
N^2(\mathcal B^\theta_Nf)(1)=\frac{N^2}{2}(f(2)-f(1))+\frac{{\kappa}N^2}{2N^\theta}(f(0)-f(1)),\\
N^2(\mathcal B^\theta_Nf)(N-1)=\frac{N^2}{2}(f(N-2)-f(N-1))+\frac{{\kappa}N^2}{2N^\theta}(f(N)-f(N-1)).
\end{array}
\right.
\end{equation*}
Above $\Delta_Nf$ denotes the discrete Laplacian of $f(\cdot)$ which is given on $x\in\Lambda_N$ by \begin{equation}\label{eq:disc_lap}
\Delta_Nf(x)=f(x+1)+f(x-1)-2f(x).
\end{equation}
Note that for $\theta=0$ the operator $\mathcal B^\theta_N$ is basically the discrete laplacian but when $\theta\neq 0$ we see some distortion at the boundary due to the mechanism of creation and annihilation. 

A simple computation shows that the stationary solution of \eqref{eq:disc_heat} is given by
$$\rho^N_{ss}(x)={E}_{\bbbp_{\mu_{ss}}}[\eta_{tN^2}(x)]=a_Nx+b_N$$
where  the coefficients $a_N$ and $b_N$ are equal to $$a_N =\frac{{\kappa}(\beta-\alpha)}{2N^\theta+{\kappa}(N-2)} \quad \textrm{and }\quad b_N=a_N\Big(\frac{N^\theta}{{\kappa}}-1\Big)+\alpha.$$
From this we get that
 \begin{equation}\label{eq:conve_emp_prof_stat_prof}
\lim_{N\to\infty}\max_{x\in\Lambda_N}\big|\rho^N_{ss}(x)-\bar{\rho}(\tfrac xN)\big|=0,
\end{equation}
where  for $q\in(0,1)$
\begin{equation}\label{eq: stat_sol}
\bar{\rho}(q)=\left\{
\begin{array}{ll}
(\beta-\alpha)q+\alpha\,;\,\theta<1,\\
\frac{{\kappa}(\beta-\alpha)}{2+{\kappa}}q+\alpha+\frac{\beta-\alpha}{2+{\kappa}}\,;\,\theta=1,\\
\frac{\beta+\alpha}{2}\,;\,\theta>1.
\end{array}
\right.
\end{equation}
Note that $\bar{\rho}(\cdot)$  will be a stationary solution of the hydrodynamic equation that we are looking for. See Figure \ref{fig:stat_sol} for a representation of $\bar{\rho}(\cdot)$.

Now we obtain information about the two point correlation function. 
Let $$V_N= \{ (x,y)\in \{ 0,...,N\}^{2}: 0<x<y<N \},$$ and its boundary $$\partial V_N = \{ (x,y)\in \{ 0,...,N \}^{2}: x=0 \,\, \textrm{or} \,\, y=N \}.$$ See Figure \ref{fig:set_V}.

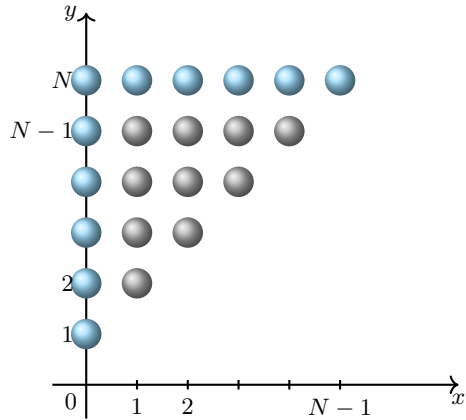
\begin{figure}[!htb]
\centering
\begin{tikzpicture}[thick, scale=0.9]
\draw[->] (-0.5,0)--(5.5,0) node[anchor=north]{$x$};
\draw[->] (0,-0.5)--(0,5.5) node[anchor=east]{$y$};
\begin{scope}[scale=0.75]
\foreach \x in {1,...,4} 
	\foreach \y in {\x,...,4}
		\shade[ball color=black!30!](\x,1+\y) circle (0.3);
		
\foreach \x in {0,...,5} 
	\shade[ball color=columbiablue](\x,6) circle (0.3); 
	 
\foreach \y in {1,...,5} 
	\shade[ball color=columbiablue](0,\y) circle (0.3);  
\end{scope}	
\draw (0,0) node[anchor=north east] {$0$};
\draw (0.75,2pt)--(0.75,-2pt) node[anchor=north] {$1$};
\draw (1.5,2pt)--(1.5,-2pt) node[anchor=north]{$2$};
\draw (2.25,2pt)--(2.25,-2pt);
\draw (3,2pt)--(3,-2pt);
\draw (3.75,2pt)--(3.75,-2pt) node[anchor=north]{\small $N-1$};
\draw (-0.05,.75) node[anchor=east] {$1$};
\draw (-0.05,1.5) node[anchor=east]{$2$};
\draw (-0.05,3.75) node[anchor= east]{$N-1$};
\draw (-0.05,4.5) node[anchor=east]{$N$};
\end{tikzpicture}
\caption{The set $V_N$ and its boundary  $\partial V_N$}
\label{fig:set_V}
\end{figure}

For $x<y\in V_N$, let $\varphi^N_{t}(x,y)$ denote the two point correlation function between the occupation sites at $x<y \in V_N$ which is defined by 
\begin{equation}\label{eq:corr_function}
\varphi_t^N(x,y)={E}_{\bbbp_{\mu_{N}}}[(\eta_{tN^2}(x)-\rho_{t}^N(x))(\eta_{tN^2}(y)-\rho_{t}^N(y))].
\end{equation}
Doing some simple, but long, computations we see that $\varphi^N_{t}$ is a solution of
\begin{equation}\label{eq:disc_eq_corr}
\begin{cases}
\partial_t \varphi_t^N(x,y)=n^2 \mathcal A^\theta_N \varphi^n_t(x,y) +g_t^N(x,y), & \textrm{ for } (x,y)\in V_N,\; t>0,\\
\varphi_t^N(x,y)=0, & \textrm{ for } (x,y)\in \partial  V_N, \;t>0,\\
\varphi_0^N(x,y)= E_{\mu_N}[\eta_0(x)\eta_0(y)]-\rho_0^N(x)\rho_0^N(y), & \textrm{ for } (x,y)\in V_N\cup \partial V_N,\\
\end{cases}
\end{equation}
where $\mathcal A_N^\theta$ is the linear operator that acts on functions $f:V_N\cup \partial V_N \to\bbbr$ as
 \begin{equation*}
 (\mathcal A_N^\theta f)(u)=\sum_{v\in V_N} c^\theta_N(u,v)\big[ f(v)-f(u)\big],
 \end{equation*}
 with
 \begin{equation*}
c^\theta_N(u,v)\;=\; \begin{cases}
 1, & \textrm{ if } \; \Vert u-v\Vert =1 \textrm{ and } \;u, v\in  V_N,\\
 N^{-\theta}, & \textrm{ if }\;  \Vert u-v\Vert=1\textrm{ and } u\in V_N,\; v\in \partial V_N,\\ 
 0,& \textrm{ otherwise, }
 \end{cases}
 \end{equation*}
 for $\theta\geq 0$.
Note that $\mathcal A^\theta_N$ is the generator of a random walk in $V_N\cup \partial V_N$ with jump rates given by $c_N^\theta(u,v)$, which is absorbed at $\partial V_N$. Above   $\Vert \cdot \Vert$ denotes the supremum norm, 
$$g^N_{t}(x,y) = -(\nabla_{N}^{+}\rho_{t}^{N}(x))^{2} \delta_{y=x+1}$$ and  \begin{equation}\label{discrete_gradient_+}\nabla_{N}^{+}\rho_{t}^{N}(x)= N(\rho_{t}^{N}(x+1)-\rho_{t}^{N}(x)).\end{equation}
 In this case, contrarily to the empirical profile, is is quite complicated to obtain an expression for the stationary solution of \eqref{eq:disc_eq_corr}. Nevertheless, we  note that a simple, but long, computation shows that the solution obtained in \eqref{eq:corr_est_stat}, in the case where the starting measure is the stationary state $\mu_{ss}$,  is in fact the stationary solution of  \eqref{eq:disc_eq_corr}. We  also observe that in \cite{FGN_Robin} it was obtained the following bound on the case $\theta=1$ for a general initial measure $\mu_N$. There it was proved that there exists a constant $C>0$ such that
\begin{equation}\label{eq:corr_bound_robin}
\sup_{t\geq 0}\max_{(x,y)\in V_N}|\varphi_t^N(x,y)|\;\leq \; \frac{C}{N}\,,
\end{equation}
but we note that the bounds on the  other regimes of $\theta$ are still open, apart the case $\theta=0$ where the bound above is given by $C/N^2$, see \cite{LMO}.

\subsection{Hydrodynamic equations} \label{sec:hyd_eq_ssep}

From now on up to the rest of these notes we fix a finite time horizon $[0,T]$.  We denote by $\langle \cdot,\cdot\rangle _{\mu}$ the inner product  in $L^{2}([0,1])$ with respect to a measure $\mu$ defined in $[0,1]$ and $\| \cdot\|_{L^2 (\mu) }$ is the corresponding norm. We note that  when $\mu$ is the Lebesgue measure we write $\langle \cdot,\cdot\rangle$ and $\| \cdot\|_{L^2}$ for the corresponding norm. 

We denote by $C^{m,n}([0, T] \times [0,1])$ the set of functions defined on $[0, T] \times [0,1] $ that are $m$ times differentiable on the first variable and $n$ times differentiable  on the second variable. For a function $G:=G(s,q)\in C^{m,n}([0, T] \times [0,1])$ we denote by $\partial _{s}G$  its derivative with respect to the time variable $s$ and and by $\partial _{q}G$  its derivative with respect to the space variable $q$. For simplicity of notation we set $\Delta G:=\partial_{q}^{2}G$. We will also make use of the  set  $C^{m,n}_c ([0,T] \times [0,1])$ of functions $G \in C^{m,n}([0, T] \times[0, 1])$  such that for any time $s$ the function $G_s$ has a compact support included in $(0,1)$ and we denote by $C_{c}^{m}(0,1)$ (resp. $C_c^\infty (0,1)$) the set of all $m$ continuously differentiable (resp. smooth) real-valued  functions defined on $(0,1)$ with compact support.  The supremum norm is denoted by $\| \cdot \|_{\infty}$. Finally,  $C^{m,n}_0 ([0,T] \times [0,1])$ is the set of functions $G \in C^{m,n}([0, T] \times[0, 1])$  such that for any time $s$ the function $G_s$ vanishes at the boundary, that is, $G_s(0)=G_s(1)=0$.

Now we want to define the space where the solutions of the hydrodynamic  equations will live on, namely the Sobolev space $\mathcal H_1$ on $[0,1]$. For that purpose, we define the semi inner-product $\langle \cdot, \cdot \rangle_{1}$ on the set $C^{\infty} ([0,1])$ by 
\begin{equation}
\langle G, H \rangle_{1} = \int_{0}^1 (\partial_q G)(q) \, (\partial_q H)(q)  \, dq,
\end{equation}  
for $G,H\in C^{\infty} ([0,1])$
and  the corresponding semi-norm is denoted by $\| \cdot \|_{1}$. 

\begin{definition}
\label{Def. Sobolev space}
The Sobolev space $\mathcal{H}^{1}$ on $[0,1]$ is the Hilbert space defined as the completion of $C^\infty ([0,1])$ for the norm 
$$\| \cdot\|_{{\mathcal H}^1}^2 :=  \| \cdot \|_{L^2}^2  +  \| \cdot \|^2_{1}.$$
Its elements elements coincide a.e. with continuous functions.  

The space $L^{2}(0,T;\mathcal{H}^{1})$ is the set of measurable functions $f:[0,T]\rightarrow  \mathcal{H}^{1}$ such that 
$$\int^{T}_{0} \Vert f_{s} \Vert^{2}_{\mathcal{H}^{1}}ds< \infty. $$
\end{definition}

We can now give the definition of the weak solutions of the  hydrodynamic equations that will be derived for the symmetric simple exclusion process in contact with stochastic reservoirs. We start by giving the notion of a weak solution to the heat equation with Dirichlet boundary conditions which will be the notion that we will derive in the regime $\theta\in[0,1)$.
In what follows  $g:[0,1]\rightarrow [0,1]$ is a measurable function and it is the initial condition of all the partial differential equations that we define below, that is ${ \rho}_{0}(q)= g(q),$ for all $q\in(0,1)$.
\begin{definition}
\label{Def. Dirichlet source Condition-g_ssep}
We say that  $\rho:[0,T]\times[0,1] \to [0,1]$ is a weak solution of the heat equation with Dirichlet boundary conditions
 \begin{equation}\label{eq:Dirichlet source Equation-g_ssep}
 \begin{cases}
 &\partial_{t}\rho_{t}(q)=\frac 12\Delta\, {\rho} _{t}(q), \quad (t,q) \in [0,T]\times(0,1),\\
 &{ \rho} _{t}(0)=\alpha, \quad { \rho}_{t}(1)=\beta,\quad t \in [0,T],
 \end{cases}
 \end{equation}
 starting from  a measurable function $g:[0,1]\rightarrow [0,1]$, 
if the following two conditions hold:
\begin{enumerate}
\item $\rho \in L^{2}(0,T;\mathcal{H}^{1})$;

\item $\rho$ satisfies the weak formulation:
\begin{equation}\label{eq:Dirichlet_source_ integral-g_ssep}
\begin{split}
F_{Dir}:=\int_0^1 \rho_{t}(q)  G_{t}(q) \,dq & -\int_0^1 g(q)   G_{0}(q) \,dq \\
&- \int_0^t\int_0^1 \rho_{s}(q)\Big(\frac 12\Delta + \partial_s\Big) G_{s}(q)  \,dq \, ds\\
&+ \int_0^t\Big\{\frac \beta 2 \partial_qG_s(1)-\frac \alpha 2 \partial_q G_s(0)\Big\}\, ds=0,
\end{split}   
\end{equation}
for all $t\in [0,T]$ and any function $G \in C_0^{1,2} ([0,T]\times[0,1])$.
\end{enumerate}

\end{definition}

 In the regime $\theta<0$ we will make use of another notion of weak solution to the heat equation with Dirichlet boundary conditions which uses as input for test functions elements in the set $C_c^{1,2} ([0,T]\times[0,1])$. Since functions in that space have compact support, in order to get a proper notion of weak solution  we need to add an extra condition to  Definition \ref{Def. Dirichlet source Condition-g_ssep} (see 3. in Definition \ref{Def. Dirichlet_c source Condition-g_ssep}). 

\begin{definition}
\label{Def. Dirichlet_c source Condition-g_ssep}
We say that  $\rho:[0,T]\times[0,1] \to [0,1]$ is a weak solution of the heat equation with Dirichlet boundary conditions given in \eqref{eq:Dirichlet source Equation-g_ssep}, starting from  a measurable function $g:[0,1]\rightarrow [0,1]$, 
if the following three conditions hold:
\begin{enumerate}
\item $\rho \in L^{2}(0,T;\mathcal{H}^{1})$, 
\item $\rho$ satisfies the weak formulation:
\begin{equation}\label{eq:Dirichlet_c_source_ integral-g_ssep}
\begin{split}
F^c_{Dir}:=\int_0^1 \rho_{t}(q)  G_{t}(q) \,dq  &-\int_0^1 g(q)   G_{0}(q) \,dq \\
&- \int_0^t\int_0^1 \rho_{s}(q)\Big(\frac 12\Delta + \partial_s\Big) G_{s}(q)  \,dq\, ds=0,
\end{split}   
\end{equation}
for all $t\in [0,T]$ and any function $G \in C_c^{1,2} ([0,T]\times[0,1])$, 

\item  $\rho _{t}(0)=\alpha, \quad { \rho}_{t}(1)=\beta$
for all $t \in (0,T]$.
\end{enumerate}
\end{definition}

\begin{remark}
We note that \eqref{eq:Dirichlet_c_source_ integral-g_ssep} coincides with \eqref{eq:Dirichlet_source_ integral-g_ssep} by taking as input a test function $G\in C_c^{1,2} ([0,T]\times[0,1])$, since in this case $\partial_qG_s(0)=\partial_qG_s(1)=0$, so that the  last term in \eqref{eq:Dirichlet_source_ integral-g_ssep} vanishes.
\end{remark}

Now we introduce the notion of weak solution of the hydrodynamic equation that we will derive  in the case $\theta=1$. In this regime the boundary reservoirs are so slow and as a consequence, a different boundary condition appears. In the case of Dirichlet boundary conditions, the value of the profile $\rho_t(\cdot)$ is fixed to be equal to $\alpha$ at $0$ and $\beta$ at $1$. This is no longer the case when $\theta\geq 1$ as we will see later on.

\begin{definition}
\label{Def. Robin Condition-g_ssep}
We say that  $\rho:[0,T]\times[0,1] \to [0,1]$ is a weak solution of the heat equation with Robin boundary conditions 
 \begin{equation}\label{Robin Equation-g_ssep}
 \begin{cases}
 &\partial_{t}\rho_{t}(q)=\frac 12 \Delta\, {\rho} _{t}(q), \quad (t,q) \in [0,T]\times(0,1),\\
 &\partial_{q}\rho _{t}(0)=\kappa (\rho_{t}(0) -\alpha),\quad \partial_{q} \rho_{t}(1)=\kappa(\beta -\rho_{t}(1)),\quad t \in [0,T],
 \end{cases}
 \end{equation}
 starting from  a measurable function $g:[0,1]\rightarrow [0,1]$, 
if the following two conditions hold: 
\begin{enumerate}
\item $\rho \in L^{2}(0,T;\mathcal{H}^{1})$, 

\item $\rho$ satisfies the weak formulation:
\begin{equation}\label{eq:Robin integral-g_ssep}
\begin{split}
&F_{Rob}:=\int_0^1 \rho_{t}(q)  G_{t}(q) \,dq  -\int_0^1 g(q)   G_{0}(q) \,dq \\
& - \int_0^t\int_0^1 \rho_{s}(q) \Big(\frac 12\Delta + \partial_s\Big) G_{s}(q)  \,ds\, dq + \frac 12\int^{t}_{0}   \{\rho_{s}(1)\partial_q G_{s}(1)-\rho_{s}(0) \partial_q G_{s}(0) \} \, ds\\
& \qquad-\frac {\kappa} {2}\int^{t}_{0} \{ G_{s}(0)(\alpha -\rho_{s}(0)) +  G_{s}(1)(\beta -\rho_{s}(1)) \}\,  ds=0,
\end{split}   
\end{equation}
for all $t\in [0,T]$ and  any function $G \in C^{1,2} ([0,T]\times[0,1])$. 
\end{enumerate}
\end{definition}

In the regime $\theta=1$ the boundary reservoirs are so slow so that a type of Robin boundary condition appears. In this case it fixes the value of the flux through the system as being proportional to the difference of concentration. Note that, for example at $q=0$, the value $\partial_q\rho_t(0)$ corresponds to the flux of particles through the left boundary and $\kappa(\rho_t(0)-\alpha)$ corresponds to the difference of the concentration, since in this case, contrarily to what happens in the case of Dirichlet boundary conditions,  it is not true that $\rho_t(0)=\alpha$ (the value of the profile at the boundaries is not fixed!)

\begin{remark} \label{neumann_cond_rem}
Observe that in the case $\kappa=0$ the equation  above is the heat equation with Neumann boundary conditions and it is the hydrodynamic equation that we will derive in the case $\theta>1.$
\end{remark}

\begin{remark}\label{rem:uniq_weak_sol}
We observe that all the partial differential equations  defined above have a unique weak solution in the sense given above. We do not include the proof of this result in these notes but we refer the interested reader to \cite{BGJO} for the proof of the uniqueness in the case of Dirichlet boundary conditions and to \cite{Adriana} for the proof of the uniqueness in the case of Robin boundary conditions. 
\end{remark}
\subsubsection{- Deriving the weak formulation:}
We note that the weak formulation given in all the regimes above can be obtained from the formal expression of the corresponding partial differential equation in the following way. Take a test function  $G \in C^{1,2} ([0,T]\times[0,1])$ and multiply both sides of the equality 
$$\partial_{s}\rho_{s}(q)=\frac 12 \Delta\, {\rho} _{s}(q)$$
 by $G$ and then integrate in the time interval $[0,t]$ and in the space interval $[0,1]$ to get
\begin{equation}\label{eq:int}
\int_0^1\int_0^t  \partial_{s}\rho_{s}(q)G_s(q) \,ds \, dq=\int_0^1\int_0^t\frac 12\Delta\, {\rho} _{s}(q)G_s(q)  \,ds \, dq.
\end{equation}
To treat the term at the left hand side of last display, we perform an integration by parts in the time integral and we get to 
\begin{equation} \label{eq:lhs_pde_int}
\int_0^1 \rho_{t}(q)  G_{t}(q) \,dq  -\int_0^1 g(q)   G_{0}(q) \,dq - \int_0^t\int_0^1 \rho_{s}(q) \partial_sG_{s}(q)  \,ds \, dq.
\end{equation}
The term at the right hand side of \eqref{eq:int} can be treated by doing an integration by parts in the space integral and we get to
\begin{equation*}
\frac 12\int_0^t\partial_q\rho_s(1)G_s(1)-\partial_q\rho_s(0)G_s(0)\, ds-\frac 12\int_{0}^t\int_0^1 \partial_q\rho_s(q)\partial_qG_s(q)\, ds\, dq.
\end{equation*}
Now, we do another integration by parts in the integral in space at the term on the right hand side of last expression and we write the previous display as
\begin{equation}\label{eq:rhs_pde_int}
\begin{split}
&\frac 12\int_0^t\partial_q\rho_s(1)G_s(1)-\partial_q\rho_s(0) G_s(0)\, ds\\-&\frac 12\int_0^t \rho_s(1)\partial_qG_s(1)-\rho_s(0)\partial_qG_s(0)+\frac 12\int_{0}^t\int_0^1 \rho_s(q)\Delta G_s(q)\, ds\, dq.
\end{split}
\end{equation}
Putting together \eqref{eq:rhs_pde_int} and \eqref{eq:lhs_pde_int} we obtain
\begin{equation*}
\begin{split}
\int_0^1 \rho_{t}(q)  G_{t}(q) \,dq  -\int_0^1 g(q)   G_{0}(q) \,dq =&\int_{0}^t\int_0^1 \rho_s(q)\Big(\frac 12\Delta+\partial_s\Big) G_s(q)\, ds\, dq\\+&\frac 12\int_0^t\partial_q\rho_s(1)G_s(1)-\partial_q\rho_s(0) G_s(0)\, ds\\-&\frac 12\int_0^t \rho_s(1)\partial_qG_s(1)-\rho_s(0)\partial_qG_s(0)\, ds.
\end{split}
\end{equation*}
Now we obtain each one of the weak formulations given above. We start with the case where $G \in C_0^{1,2} ([0,T]\times[0,1])$ and we will derive \eqref{eq:Dirichlet_source_ integral-g_ssep}. For that purpose note that since $G$ vanishes at the boundary of $[0,1]$ and since $\rho_s(0)=\alpha$ and $\rho_s(1)=\beta$, the expression in the previous display becomes equivalent to $F_{Dir}=0$. 

On the other hand, if $G \in C_c^{1,2} ([0,T]\times[0,1])$,  then $G$  vanishes at the boundary of $[0,1]$ and $\partial_q G$ also vanishes at the boundary of $[0,1]$, so that for $\rho$ satisfying the Dirichlet boundary conditions of \eqref{eq:Dirichlet source Equation-g_ssep}, the expression in the display above  becomes equivalent to $F^c_{Dir}=0$.

Finally for $G \in C^{1,2} ([0,T]\times[0,1])$ and for $\rho$ satisfying the Robin boundary conditions of \eqref{Robin Equation-g_ssep}, the expression in the previous display becomes equivalent
 to $F_{Rob}=0$.
 \subsubsection{-Stationary solutions:}
Now we deduce the stationary solutions for each one of the equations given above. We start with \eqref{eq:Dirichlet source Equation-g_ssep}. For that purpose note that, denoting by $\bar\rho(\cdot)$ the stationary solution we have that $\Delta \bar\rho(t,q)=0$ implies that $\bar\rho(q)=aq+b$ for $a,b\in\bbbr$ and $q\in(0,1)$. 
Imposing the Dirichlet boundary conditions we arrive at $$a=(\beta-\alpha)\quad \textrm{and}\quad b=\beta,$$ so that 
\begin{equation}\label{eq:stat_sol_dir}
\bar\rho_{Dir}(q)=(\beta-\alpha)q+\alpha.
\end{equation}

On the other hand, imposing the Robin boundary conditions of \eqref{Robin Equation-g_ssep} we arrive at 
$$a=\frac{\kappa(\beta-\alpha)}{2+\kappa}\quad \textrm{and}\quad b=\alpha+\frac{\beta-\alpha}{2+\kappa},$$ so that for $q\in(0,1)$
\begin{equation}\label{eq:stat_sol_rob}\bar\rho_{Rob}(q)=\frac{\kappa(\beta-\alpha)}{2+\kappa}q+\alpha+\frac{\beta-\alpha}{2+\kappa}.
\end{equation}
Finally, if we impose the Neumann boundary conditions, any constant solution is a stationary solution of \eqref{Robin Equation-g_ssep} with $\kappa=0$ (which corresponds to the Neumann regime). In this case we note that the stationary solution is not unique. 
Below we draw the graph of these stationary solutions for a choice of $\alpha=0.2$ and $\beta=0.8$.
\begin{figure}[h]
    \centering
    \includegraphics[width=0.7\textwidth]{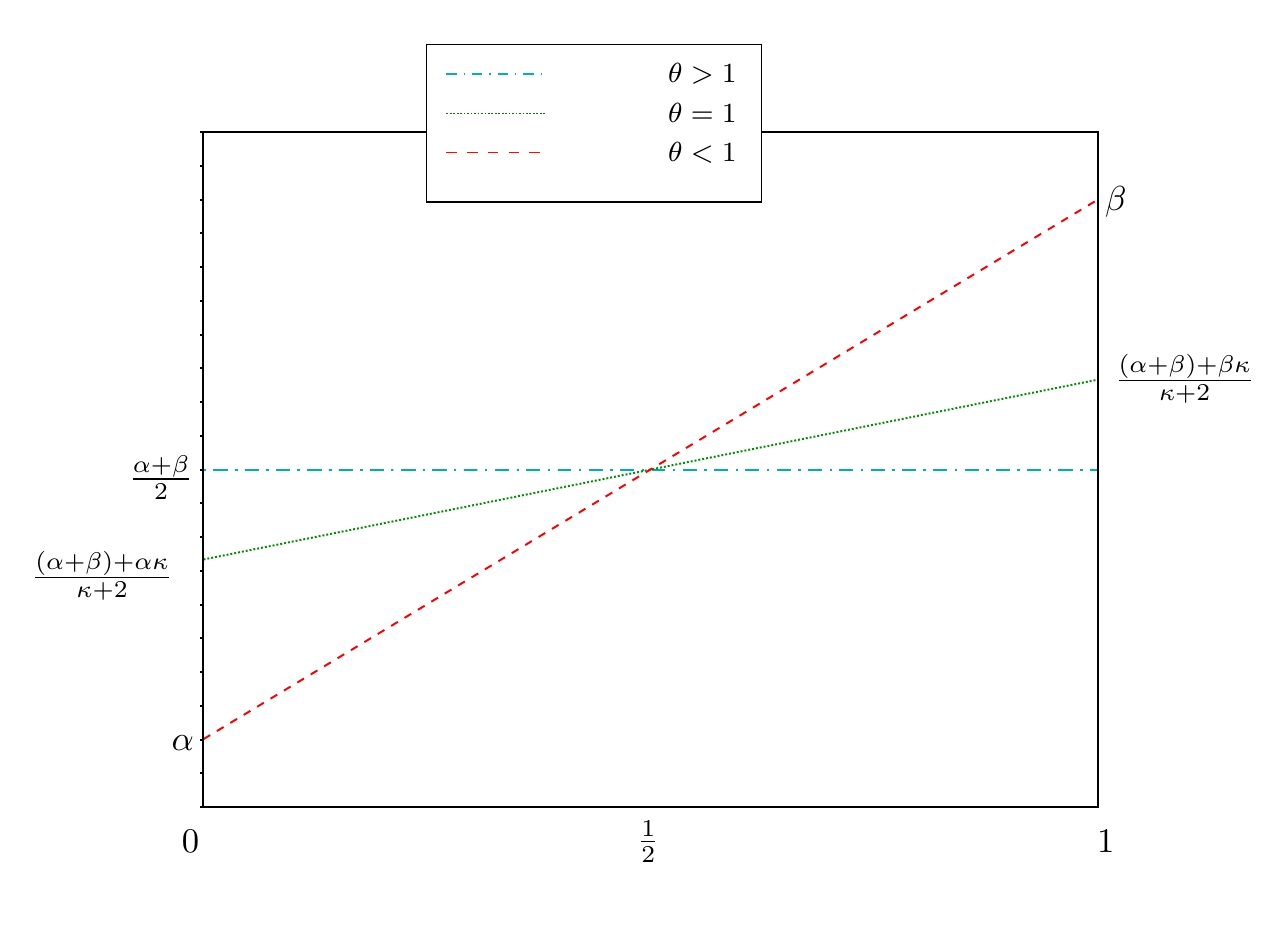}
    \caption{Stationary solutions of the hydrodynamic equations.}
    \label{fig:stat_sol}
\end{figure}

Now we give the explicit expression for the solution of each hydrodynamic equation.
\begin{proposition}
\label{prop:stat}
We have that:
\begin{enumerate}
\item The solution of \eqref{eq:Dirichlet source Equation-g_ssep} with initial condition $g(\cdot)$ is equal to 
 \begin{equation*}
\rho_t (q)= {\bar \rho}_{Dir} (q)+ \sum _{n=1}^{\infty} e^{-\frac{(n\pi)^2}{2} t}2\sin (n\pi  q ).
 \end{equation*}
\item The solution of \eqref{Robin Equation-g_ssep} with initial condition $g(\cdot)$ is equal to 
$$\rho_t(q) ={\bar \rho}_{Rob} (q)+ \sum _{n=1}^{\infty} C_{n}e^{-\frac{\lambda_{n}}{2}t}X_{n}(q),$$
where
\begin{equation}\label{basis}
X_n(q)=A_n\sin (\sqrt{\lambda_n} \,q )+ A_n\kappa \sqrt{\lambda_n} \cos(\sqrt{\lambda_n} \, q),
\end{equation} 
$A_n$ is a normalizing constant in such a way that $X_n$ has unitary $L^2([0,1])$-norm and $$C_{n} = \int_{0}^{1}(g(q)-{\bar \rho} (q))X_{n}(q)dq.$$
\end{enumerate}
\end{proposition}
\begin{proof}
The solution $\rho(\cdot)$ to (\ref{eq:Dirichlet source Equation-g_ssep})  starting from a profile $g(\cdot)$ is such that $u=\rho -\bar \rho$ is solution to  (\ref{eq:Dirichlet source Equation-g_ssep})  with homogeneous boundary conditions $\alpha=\beta=0$, i.e. 
 \begin{equation}\label{eq:Dir_Equation-g_hom}
 \begin{cases}
 &\partial_{t}u_{t}(q)=\frac 12 \Delta\, u _{t}(q), \quad (t,q) \in [0,T]\times(0,1),\\
 &u_{t}(0) =0=u_{t}(1),\quad t \in [0,T].
 \end{cases}
 \end{equation}
It is well known that $u_t(q)$ is given by 
$$ u_t(q) =\sum _{n=1}^{\infty} e^{-\frac{(n\pi)^2}{2} t}2\sin (n\pi  q ).$$
 From the previous computations  we conclude that  the solution $\rho(\cdot)$ of  (\ref{eq:Dirichlet source Equation-g_ssep}) starting from $g(\cdot)$ is given by $$\rho_t(q) = (\beta -\alpha)q+\alpha+ \sum _{n=1}^{\infty} e^{-\frac{(n\pi)^2}{2} t}2\sin (n\pi  q ).$$

On the other hand, the solution $\rho(\cdot)$ of  (\ref{Robin Equation-g_ssep}) starting from $g(\cdot)$ is such that $u=\rho -\bar \rho$ is solution to (\ref{Robin Equation-g_ssep}) with $\alpha=\beta=0$, i.e.
 \begin{equation}\label{Robin Equation-g_hom_ssep}
 \begin{cases}
 &\partial_{t}u_{t}(q)= \frac 12\Delta\, u _{t}(q), \quad (t,q) \in [0,T]\times(0,1),\\
 &\partial_{q}u _{t}(0)=\kappa u_{t}(0) ,\quad \partial_{q} u_{t}(1)=-\kappa u_{t}(1),\quad t \in [0,T].
 \end{cases}
 \end{equation}
It is well known that $u_t(q)$ is given by 
$$ u_t(q) =\sum _{n=1}^{\infty} C_{n}e^{-\frac{\lambda_{n}}{2}t}X_{n}(q),$$
where $X_n(q)$ writes as 
$$X_n(q)=A_n\sin (\sqrt{\lambda_n}  q )+ B_n \cos(\sqrt{\lambda_n} q),$$for some constants $A_n$ and $B_n$.
Then,  the first  boundary condition in \eqref{Robin Equation-g_hom_ssep} gives
$B_n= \sqrt{\lambda_n}\kappa A_n.$ To avoid the null solution we consider  $A_n\neq 0$.
The second boundary condition in \eqref{Robin Equation-g_hom_ssep} gives
\begin{equation}\label{trans}
\tan (\sqrt{\lambda_n}) =\frac{2\kappa\sqrt{\lambda_n}}{\lambda_n\kappa^2-1},
\end{equation}
whose solution $\lambda_n$ satisfying $(n-1)\pi\leq \sqrt{\lambda_n}\leq n\pi$ is such that  $\lambda_n\sim n^2\pi^2$ as $n\to\infty$. From the previous computations we get that
$X_n(q)$ is given by \eqref{basis} and there $A_n$ is a normalizing constant in such a way that $X_n$ has unitary $L^2([0,1])$-norm. Moreover
 $$C_{n} = \int_{0}^{1}(g(q)-\bar\rho(q))X_{n}(q)dq.$$
 From the previous computations  we conclude that  the solution $\rho(\cdot)$ of  (\ref{Robin Equation-g_ssep}) starting from $g(\cdot)$ is given by $$\rho_t(q)= \dfrac{\kappa(\beta -\alpha)}{2+\kappa}  q+ \alpha+\frac{\beta-\alpha}{2+\kappa}+ \sum _{n=1}^{\infty} C_{n}e^{-\frac{\lambda_{n}}{2}t}X_{n}(q).$$
\end{proof}

\subsection{Hydrodynamic Limit}
\label{sec:HL}

In this section we want to state the hydrodynamic limit of the process $\{\eta_{tN^2}\,:\,t\geq{0}\}$ with state space $\Omega_{N}$ and with infinitesimal generator $N^{2}\mc L_{N}$ defined in \eqref{generator_ssep}. Note that here we are going to take  $\Theta(N)=N^2$. Let ${\mc M}^+$ be the space of positive measures on $[0,1]$ with total mass bounded by $1$ equipped with the weak topology. For any configuration  $\eta \in \Omega_{N}$ we define the empirical measure $\pi^{N}(\eta,dq)$ on $[0,1]$ by 
\begin{equation}\label{MedEmp}
\pi^{N}(\eta, dq)=\dfrac{1}{N-1}\sum _{x\in \Lambda_{N}}\eta(x)\delta_{\frac{x}{N}}\left( dq\right),
 \end{equation}
where $\delta_{a}$ is a Dirac mass on $a \in [0,1]$, and
$$\pi^{N}_{t}(\eta, dq):=\pi^{N}(\eta_{tN^2}, dq).$$
This measure gives weight  $\frac{1}{N}$ to each occupied site of the  configuration  $\eta$. 

Fix $T>0$ and $\theta\in \bbbr$. Recall that  $\bbbp_{\mu _{N}}$ is the probability measure in the Skorohod space $\mathcal D([0,T], \Omega_N)$ induced by the  Markov process $\{\eta_{tN^2}\,:\,t\geq{0}\}$ and the initial probability measure $\mu_N$ and we denote by $E _{\bbbp_{\mu_{N}}}$ the expectation with respect to $\bbbp_{\mu _{N}}$.  Now let $\lbrace \bbbq_{N}\rbrace_{N\geq 1}$ be the  sequence of probability measures on $\mathcal D([0,T],\mathcal{M}^{+})$ induced by the  Markov process $\{\pi_{t}^N\,:\,t\geq{0}\}$  and by $\bbbp_{\mu_{N}}$.

At this point we need to fix an initial profile $\rho_0: [0,1]\rightarrow[0,1]$ which is  measurable  and an initial probability measure $\mu_N\in\Omega_N$. We are going to consider the following set of initial measures:  

\begin{definition}\label{def:meas_ass}
A sequence of probability measures $\lbrace\mu_{N}\rbrace_{N\geq 1 }$ in $\Omega_{N}$  is associated to the profile $\rho_{0}(\cdot)$ if for any continuous function $G:[0,1]\rightarrow \bbbr$  and any $\delta > 0$ 
\begin{equation}\label{assoc_mea}
  \lim _{N\to\infty } \mu _{N}\Big( \eta \in \Omega_{N} : \Big| \dfrac{1}{N-1}\sum_{x \in \Lambda_{N} }G\left(\tfrac{x}{N} \right)\eta({x}) - \int_{0}^1G(q)\rho_{0}(q)dq \Big|   > \delta \Big)= 0.
\end{equation}
\end{definition}
Note that \eqref{assoc_mea} states that
\begin{equation}\label{assoc_mea_alternative}
\int G(q)\pi^N(\eta,dq)\longrightarrow_{N\to\infty} \int_{0}^1G(q)\rho_{0}(q)dq,
\end{equation} 
with respect to $\mu_N$, which 
means that the empirical measure at time $t=0$ converges, in probability with respect to $\mu_N$, as $N\to\infty$, to the  deterministic measure $\rho_0(q)dq$, which is absolutely continuous with respect to the Lebesgue measure and the density is the profile $\rho_0(\cdot)$.

The hydrodynamic limit that we want to derive states that the previous result is also true for any $t\in[0,T]$, that is, the empirical measure at time $t$ converges in probability with respect to the distribution of the system at time $t$, as $N\to\infty$, to the  deterministic measure $\rho_t(q)dq$, where $\rho_t(\cdot)$ is a solution (here in the weak sense) to some partial differential equation, \emph{the hydrodynamic equation}.

The first main result of these notes is summarized in the following theorem (see also Figure \ref{fig:hlf}).

\begin{theorem}
\label{th:hyd_ssep}
 Let $g:[0,1]\rightarrow[0,1]$ be a measurable function and let $\lbrace\mu _{N}\rbrace_{N\geq 1}$ be a sequence of probability measures in $\Omega_{N}$ associated to $g(\cdot)$. Then, for any $t\in[0,T]$,
\begin{equation*}\label{limHidreform}
 \lim _{N\to\infty } \bbbp_{\mu _{N}}\Big( \eta_{\cdot} : \left\vert \dfrac{1}{N-1}\sum_{x \in \Lambda_{N} }G\left(\tfrac{x}{N} \right)\eta_{tN^2}(x) - \int_{0}^1G(q)\rho_{t}(q)dq \right\vert    > \delta \Big)= 0,
\end{equation*}
where  $\rho_{t}(\cdot)$ is the unique weak solution of : 
\begin{itemize}
\item [$\bullet$]  \eqref{eq:Dirichlet source Equation-g_ssep} as given in Definition \ref{Def. Dirichlet_c source Condition-g_ssep}, if $\theta<0$; 
\item[$\bullet$] \eqref{eq:Dirichlet source Equation-g_ssep} as given in Definition \ref{Def. Dirichlet source Condition-g_ssep}, if $\theta\in[0,1)$;
\item[$\bullet$]  (\ref{Robin Equation-g_ssep}), if $\theta =1$;
\item[$\bullet$]  (\ref{Robin Equation-g_ssep}) with $\kappa =0$, if $\theta>1$.
\end{itemize}
\end{theorem}

\begin{figure}
\begin{center}
\begin{tikzpicture}[scale=0.25]
\shade[right color=black!30, left color=gray!50] (0,5) -- (0,10) --(20,10)--(20,5)-- cycle;
\shade[right color=columbiablue, left color=gray!50] (0,5) -- (20,5) --(20,-10)--(0,-10)-- cycle;
\draw (0,12) node[right]{$\theta$};
\draw (24,0) node[right]{};
\draw (-1,0) node[left]{$\theta=0$};
\draw (-1,5) node[left]{$\theta=1$};
\draw[->,>=latex] (0,0) -- (23,0);
\draw[->,>=latex] (0,-11) -- (0,11);
\draw[-,=latex,black,ultra thick] (0,5) -- (20, 5) node[midway, sloped, below] {{Heat eq. \& Robin b.c.}};
\node[right, columbiablue] at (0,8) {Heat eq. \& Neumann b.c.} ;
\node[right,gray] at (0,-2) {Heat eq. \& Dirichlet b.c.} ;
\fill[black] (0,5) circle (0.3cm);
\end{tikzpicture}
\end{center}
\caption{The three hydrodynamic equations depending on $\theta$.}
\label{fig:hlf}
\end{figure}
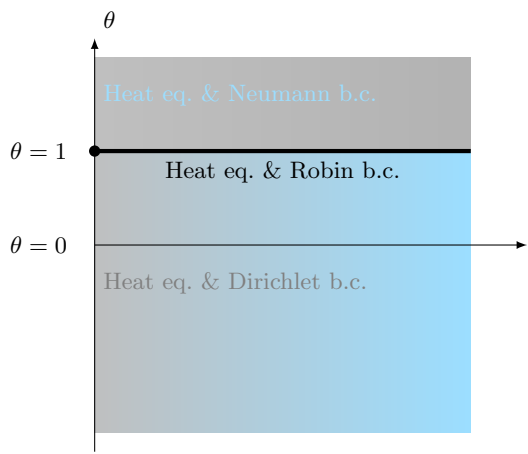

\begin{remark}
We note that in \cite{Adriana} it was studied the case where the reservoirs are slowed (which corresponds to the regime $\theta\geq 0$). In the previous theorem we considered also the case where the reservoirs are fast (which corresponds to $\theta<0$) but we note that the macroscopic behavior of the  system is also given by the heat equation with Dirichlet boundary conditions as happens in the case $\theta\in[0,1)$. To prove this result we note that the notion of weak solution in the case $\theta<0$ is different from the notion of weak solution in the case $\theta\in[0,1)$ since it uses as input functions with compact support. 
\end{remark}

The proof of Theorem \ref{th:hyd_ssep} proceeds as follows:  We split the proof into showing first the tightness of the sequence  $\lbrace \bbbq_{N}\rbrace_{N\geq 1}$ and then we characterize  uniquely the limiting point $\bbbq$ of this sequence. These two results combined together, imply  the convergence of  $\lbrace\bbbq_{N}\rbrace_{N\geq 1}$ to $\bbbq$  as $N\rightarrow \infty$.

The next section is dedicated to the presentation of an heuristic argument to deduce the hydrodynamic equations from the interacting particle system by means of the Dynkin's formula; in Section \ref{sec:tightness} we present the proof of tightness and in Section \ref{limit point} we characterize the limit point $\bbbq$. We note that  in order to characterize the limit point $\bbbq$,  we prove in Section \ref{sec:abs_cont} that all limiting points of  the sequence $\lbrace \bbbq_{N}\rbrace_{N\geq 1}$ are concentrated on trajectories of  measures that are  absolutely continuous with respect to the Lebesgue measure and  that the   density $\rho_{t}(\cdot)$ is a weak solution of  the corresponding hydrodynamic equation. From the  uniqueness of weak solutions of the hydrodynamic equations, see Remark \ref{rem:uniq_weak_sol}, we conclude that $\lbrace \bbbq_{N}\rbrace_{N\geq 1}$ has a unique limit point $\bbbq$, and therefore we conclude the convergence of the sequence  to this limit point.

\subsection{Heuristics for  hydrodynamic equations}

\label{sec:heuri_ssep}
In this section we give the main ideas which are behind the identification of limit points as weak solutions of the partial differential equations given in  Section \ref{sec:hyd_eq_ssep}. Now we argue that the density $\rho_t(\cdot)$ is a weak solution of the corresponding hydrodynamic equation for each regime of $\theta$. We remark that we are not going to prove here that the solution $\rho_t(\cdot)$ belongs to the space $L^2(0,T; \mc H^1)$ but we refer the reader to \cite{Adriana,BGJO} for a complete proof of this fact.  In order to prove  that $\rho_t(\cdot)$ satisfies the weak formulation we use  auxiliary martingales associated to the Markov process $\{\eta_t : t\geq 0\}$. For that purpose, and to make the exposition simpler, we  fix a function $G:[0,1]\to\bbbr$ which does not depend on time and which is two times continuously differentiable. If $\theta <0$ we will assume further that it has a compact support included in $(0,1)$. 
First we recall  Dynkin's formula.
\begin{theorem}\label{Dynkin_formula}
Let $\{ \eta_{t}\,:\,t\geq 0\}$ be a  Markov process with generator  $\mathcal L$  and with countable state space $E$. Let $F: \bbbr^{+}\times E \rightarrow \bbbr$ be a bounded function such that 
\begin{itemize}
\item {$\forall \eta \in E, F(\cdot,\eta)\in C^{2}(\bbbr^{+})$,}
\item { there exists a finite constant $C$, such that for $j=1,2$ $$\sup_{(s,\eta)}|\partial_{s}^{j}F(s,\eta)|\leq C.$$} 
\end{itemize}
For $t \geq 0$, let
\begin{equation*}
\begin{split}
M_{t}^{F} =&F(t,\eta_{t}) - F(0,\eta_{0}) - \int_{0}^{t}(\partial_{s}+\mathcal L)F(s,\eta_{s})ds,\\
N_t^F=&  (M^F_t)^2-\int_0^t \{\mathcal L F(s,\eta_s)^2-2F(s,\eta_s)\mathcal L F(s,\eta_s)\}ds.
\end{split}
\end{equation*} 
Then, $\{ M_{t}^{F}\}_{t\geq 0}$ and  $\{ N_{t}^{F}\}_{t\geq 0}$ are  martingales with respect to  $\mathcal F_s=\sigma(\eta_s\,;\, s\leq t)$. 
\end{theorem}

Let us fix a test function $G:[0,1]\to\bbbr$ and apply Dynkin's formula with 
\begin{equation}\label{eq:F_for_Dynkin}
F(t,\eta_t)=\langle \pi^N_t,G\rangle=\frac{1}{N-1}\sum_{x\in\Lambda_N}\eta_{tN^2}(x)G\Big(\tfrac{x}{N}\Big). \end{equation}
Above  $\left\langle \pi_{t}^{N},G\right\rangle$ represents  the integral of $G$ with respect the measure $ \pi_{t}^{N}$. 
Note that $F$ does not depend on time, only through $\eta_t$. A simple computation shows that
\begin{equation}\label{gen_action_ssep}\begin{split}
N^2\mathcal L_N\langle \pi^N_s,G\rangle&=\langle \pi^N_s,\frac 12\Delta_NG\rangle\\
&+\frac 12\Big(\nabla _N^+G(0)\eta_{sN^2}(1)-\nabla_N^-G(1)\eta_{sN^2}(N-1)\Big)\\
&+\frac {\kappa}{2}\frac{N^{2-\theta}}{N-1}G\Big(\tfrac{1}{N}\Big)(\alpha-\eta_{sN^2}(1))\\&+\frac {\kappa}{2}\frac{N^{2-\theta}}{N-1}G\Big(\tfrac{N-1}{N}\Big)(\beta-\eta_{sN^2}(N-1)),
\end{split}
\end{equation}
from where we obtain that
\begin{equation}\label{Dynkin'sFormula_ssep}
\begin{split}
M_{t}^{N}(G)= \langle \pi_{t}^{N},G\rangle -\langle \pi_{0}^{N},G\rangle&-\int_{0}^{t}\langle \pi^N_s,\frac 12\Delta_NG\rangle\, ds\\
&-\frac 12 \int_0^t\nabla _N^+G(0)\eta_{sN^2}(1)-\nabla_N^-G(1)\eta_{sN^2}(N-1)\, ds\\
&-\frac {\kappa}{2}\int_0^t\frac{N^{2-\theta}}{N-1}G\Big(\tfrac{1}{N}\Big)(\alpha-\eta_{sN^2}(1))\, ds\\&-\frac {\kappa}{2}\int_0^t\frac{N^{2-\theta}}{N-1}G\Big(\tfrac{N-1}{N}\Big)(\beta-\eta_{sN^2}(N-1)) \, ds,
\end{split}
\end{equation}
is a martingale with respect to the natural filtration  $\{\mathcal{F}_{t}\}_{ t\ge 0}$, where for each $t\ge 0$, $\mathcal{F}_t:=\sigma(\eta_s: s < t)$. Above, $\Delta_N$ is the discrete laplacian defined in \eqref{eq:disc_lap}, $\nabla_N^+ $ is defined in \eqref{discrete_gradient_+}
 and $$\nabla_N^-f(x)=N(f(x)-f(x-1)).$$ 
 Now we look at the integral terms in \eqref{Dynkin'sFormula_ssep}.

\subsubsection{- The case $\theta \in[0,1)$:}  In this regime, we take a  test function $G:[0,1]\to\bbbr$  two times continuously differentiable such that $G(0)=G(1)=0$. Then, we can subtract $G(0)$ (resp. $G(1)$) in the fifth term (resp. sixth term) at the right hand side of \eqref{Dynkin'sFormula_ssep} and then doing a Taylor expansion on $G$  we get that
 \begin{equation*}\begin{split}
M^N_t(G)&=\langle \pi^N_t,G\rangle-\langle \pi^N_0,G\rangle-\int_0^t\langle \pi^N_s,\frac 12\Delta_NG\rangle ds\\
&-\frac 12\int_0^t\nabla _N^+G(0)\eta_{sN^2}(1)-\nabla_N^-G(1)\eta_{sN^2}(N-1) ds +{O(N^{-\theta})}.
\end{split}
\end{equation*}
If  we can replace $\eta_{sN^2}(1)$ by $\alpha$ and $\eta_{sN^2}(N-1)$ by $\beta$, which will be a consequence of Lemma \ref{Rep-Dirichlet2} in Appendix \ref{fix_profile} (see Remark \ref{RL_theta_Dir}), then above we have 
\begin{equation*}\begin{split}
M^N_t(G)&=\langle \pi^N_t,G\rangle-\langle \pi^N_0,G\rangle-\int_0^t\langle \pi^N_s,\frac 12\Delta_NG\rangle ds\\
&-\frac 12\int_0^t\nabla _N^+G(0)\alpha-\nabla_N^-G(1)\beta ds  + {O(N^{-\theta})}
\end{split}
\end{equation*}
plus a term that vanishes as $N\to+\infty$.

Taking the expectation with respect to $ \mu_N$ in the expression above we get 
 \begin{equation*}\begin{split}
&\frac{1}{N-1}\sum_{x=1}^{N-1}G\Big(\tfrac{x}{N}\Big)\Big(\rho_t^N(x)-\rho_0^N(x)\Big)-\int_0^t\frac{1}{N-1}\sum_{x=1}^{N-1}\frac 12\Delta_NG\Big(\tfrac{x}{N}\Big)\rho_{s}^N(x) ds\\
&-\frac 12\int_0^t\nabla _N^+G(0)\alpha-\nabla_N^-G(1)\beta ds +{O(N^{-\theta})}=0.
\end{split}
\end{equation*}
Note that above we used the fact that the average of martingales is constant in time and that  $M^N_0(G)=0$. 
Now, assuming that $\rho_t^N(x)\sim\rho_t(\tfrac xN)$ and taking the limit as $N\to\infty$ we get that 
\begin{equation*}\begin{split}
&\int_0^1\rho_t(q)G(q)-\rho_0(q)G(q)dq -\int_0^t\int_0^1 \frac 12\Delta G(q)\rho_{s}(q) dq ds\\
&-\frac 12\int_0^t\partial_qG(0)\alpha-\partial_qG(1)\beta ds =0.
\end{split}
\end{equation*}

Note that the restriction $\theta\geq 0$ comes from the fact that the errors, which arise from the Taylor expansion in $G$, have to vanish as $N\to\infty$ and the restriction $\theta<1$ comes from the replacement of the occupation variables $\eta(1)$ and $\eta(N-1)$ by $\alpha$ and $\beta$, respectively, see Lemma \ref{Rep-Dirichlet2} in Appendix \ref{fix_profile} . At this point compare the previous expression with the weak formulation given in \eqref{eq:Dirichlet_source_ integral-g_ssep} and note that the test function $G$ does not depend on time.

\subsubsection{- The case $\theta<0$:} 

In this regime we take a  function $G:[0,1]\to\bbbr$ with compact support and  we note that the last three terms at the right hand side of \eqref{Dynkin'sFormula_ssep} vanish in this case. From this and the same arguments as above we  get that \begin{equation*}\begin{split}
M^N_t(G)&=\langle \pi^N_t,G\rangle-\langle \pi^N_0,G\rangle-\int_0^t\langle \pi^N_s,\frac 12 \Delta_NG\rangle ds.
\end{split}
\end{equation*}
Taking the expectation with respect to $\bbbp \mu_N$ in the expression above 
and assuming that $\rho_t^N(x)\sim\rho_t(\tfrac xN)$, and then taking the limit as $N\to\infty$ we get that 
\begin{equation*}\begin{split}
&\int_0^1\rho_t(q)G(q)-\rho_0(q)G(q)dq -\int_0^t\int_0^1 \frac 12\Delta G(q)\rho_{s}(q) dq ds=0.
\end{split}
\end{equation*}
Again compare with the weak formulation given in \eqref{eq:Dirichlet_c_source_ integral-g_ssep} and note that the test function $G$ does not depend on time. 

\begin{remark}\label{fix_prof_ssep}
We remark here that in this particular case there is an extra condition in Definition  \ref{Def. Dirichlet_c source Condition-g_ssep} with respect to the other notions of weak solutions where we only have to check the weak formulation and to show that the solution belongs to a Sobolev space. In this case we also need  to show that the value of the profile $\rho_t(\cdot)$ is fixed at the boundary. We leave this issue to Appendix \ref{fix_profile}. 
\end{remark}

\subsubsection{- The case $\theta=1$:}

In this case we consider an arbitrary function $G:[0,1] \to \bbbr$ which is two times continuously differentiable and we get  \begin{equation*}\begin{split}
M^N_t(G)&=\langle \pi^N_t,G\rangle-\langle \pi^N_0,G\rangle-\int_0^t\langle \pi^N_s,\frac 12\Delta_NG\rangle ds\\
&-\frac 12\int_0^t\nabla _N^+G(0)\eta_{sN^2}(1)-\nabla_N^-G(1)\eta_{sN^2}(N-1) ds\\
 &-\frac {\kappa}{2}\frac{N}{N-1}\int_0^tG\Big(\tfrac{1}{N}\Big)(\alpha-\eta_{sN^2}(1))+G\Big(\tfrac{N-1}{N}\Big)(\beta-\eta_{sN^2}(N-1)) ds. \end{split}
\end{equation*}

In this regime Lemma \ref{Rep-Dirichlet2} in Appendix \ref{fix_profile}  is no longer valid. Nevertheless, by Remark \ref{rep_for_ssep_robin}   we can replace $\eta_{sN^2}(1)$ (resp.  $\eta_{sN^2}(N-1)$) by the average in a box around $1$ (resp. $N-1$):
\begin{equation}\label{boxes_ssep}
\overrightarrow{\eta}_{sN^2}^{\varepsilon N}(1):=\frac{1}{\varepsilon N}\sum_{x=1}^{1+\varepsilon N}\eta_{sN^2}(x), \quad \overleftarrow{\eta}_{sN^2}^{\varepsilon N}(N-1):=\frac{1}{\varepsilon N}\sum_{x=N-1}^{N-1-\varepsilon N}\eta_{sN^2}(x).\end{equation}
Here we note that the sum above  goes from $1$ to $1+\lfloor{\varepsilon N\rfloor}$ but for sake of simplicity we write $1+\varepsilon N$. 
By noting that $$\overrightarrow{\eta}^{\varepsilon N}_{sN^2}(1)\sim\rho_s(0)\quad  (\textrm{resp.} \overleftarrow{\eta}_{sN^2}^{\varepsilon N}(N-1)\sim\rho_s(1)),$$  for details on this approximation see for example \cite{Adriana,BGJO} - and repeating the same arguments as above, we  get to
\begin{equation*}
\begin{split}
\int_0^1\rho_t(q)G(q)-\rho_0(q)G(q)dq &-\int_0^t\int_0^1 \frac 12\Delta G(q)\rho_{s}(q) dq ds\\
&-\frac 12\int_0^t\partial_qG(0)\rho_s(0)-\partial_qG(1)\rho_s(1) ds \\&+\frac \kappa 2\int_0^tG(0)(\alpha-\rho_s(0))-G(1)(\beta-\rho_s(1)) ds=0.
\end{split}
\end{equation*}
Again compare with the weak formulation given in \eqref{eq:Dirichlet_c_source_ integral-g_ssep} and note that the test function $G$ does not depend on time.

\subsubsection{- The case $\theta>1$:} 

This regime is quite similar to the previous one. We consider again an arbitrary function $G:[0,1] \to \bbbr$ which is two times continuously differentiable and we note that the last two terms at the right hand side of \eqref{Dynkin'sFormula_ssep} vanish  since $\theta>1$. Then, repeating the same arguments as in the previous section and noting that Remark \ref{rep_for_ssep_robin}  also applies to $\theta>1$ we obtain at the end that
\begin{equation*}\begin{split}
&\int_0^1\rho_t(q)G(q)-\rho_0(q)G(q)dq -\int_0^t\int_0^1\frac 12 \Delta G(q)\rho_{s}(q) dq ds\\
&-\frac 12\int_0^t\partial_qG(0)\rho_s(0)-\partial_qG(1)\rho_s(1) ds =0.
\end{split}
\end{equation*}
Again compare with the weak formulation given in \eqref{eq:Dirichlet_c_source_ integral-g_ssep} and note that the test function $G$ does not depend on time.

\begin{remark}
Note that the parameter $\kappa$ that appears in the boundary dynamics is only seen at the macroscopic level in the case $\theta=1$ which corresponds to the heat equation with Robin boundary conditions. 
\end{remark}
\subsection{Tightness}\label{sec:tightness}

In this section we show that the sequence of probability measures $\{\bbbq_N\}_{N\geq 1} $, defined in the beginning  of Section \ref{sec:HL}, is tight in the Skorohod space $\mathcal D([0,T], \mathcal M_+)$. In order to do that, we invoke the Aldous's criterium  which says that
\begin{lemma}
A sequence $\{P_{N}\}_{N\geq 1}$ of probability measures defined on $\mathcal D([0,T],\mathcal{M}_{+})$ is tight if these two conditions hold:
\begin{description}
\item[a.]
For every $t\in{[0,T]}$ and every $\varepsilon>0$, there exists $K_{\varepsilon}^{t}\subset{\mathcal{M}_{+}}$ compact, such that
\begin{equation*}
\sup_{N\geq 1}P_{N}\Big(\pi_{t}\notin{K_{\varepsilon}^{t}}\Big)\leq{\varepsilon},
\end{equation*}
\item [b.]
For every $\varepsilon>0$
\begin{equation*}
\lim_{\gamma\rightarrow{0}}\limsup_{N\rightarrow{\infty}}\sup_{\substack{\tau\in{\mathcal{T}_{T}}\\ \theta\leq{\gamma}}}
P_{N}\Big(d(\pi_{\tau+\theta},\pi_{\tau})>\varepsilon\Big)=0,
\end{equation*}
\end{description}
where $\mathcal{T}_{T}$ denotes the set of stopping times with respect to the canonical filtration, bounded by $T$ and $d$ is the metric in the
space $\mathcal{M}_{+}$.
\end{lemma}
By Proposition 1.7 of Chapter 4 in \cite{KL} it is enough to show that for every function $G$ in a dense subset of $C([0,1])$, with respect to the uniform topology, the sequence of measures that
corresponds to the real processes $\langle\pi_{t}^{N},G\rangle$ is tight.

In our setting case, the first condition {\bf{a.}} above translates by saying that:
\begin{equation*}
\lim_{A\rightarrow{+\infty}}\lim_{N\rightarrow{+\infty}}\bbbp_{\mu_N}\Big(|\langle\pi_{t}^{N},G\rangle|>A\Big)=0.
\end{equation*}
This is a consequence of Chebychev's inequality and the fact that for the exclusion type dynamics, the number of particles per site is at most one, we leave the details on this to the reader.  So, it remains to show condition \textbf{b}.
In this context and since we are considering the real process $\langle\pi_{t}^{N},G\rangle$, the distance $d$ above is the usual distance in $\bbbr$. Then, we must show that
  for all $\varepsilon > 0$ and any function $G$ in a dense subset of $C([0,1])$, with respect to the uniform topology,  it holds that
\begin{equation}
\label{eq:tight_1}
\displaystyle \lim _{\delta \rightarrow 0} \limsup_{N\rightarrow\infty} \sup_{\tau  \in \mathcal{T}_{T},\bar\tau \leq \delta} {\bbbp}_{\mu _{N}}\Big(\eta_{\cdot}:\Big| \langle\pi^{N}_{\tau+\bar\tau},G \rangle-\langle\pi^{N}_{\tau},G\rangle\Big|> \varepsilon \Big) =0.
\end{equation}
Above we assume that all the stopping times are bounded by $T$, thus, $\tau+ \bar\tau$ should be understood as $ (\tau+ \bar\tau) \wedge T$.

Recall that it is enough to prove the assertion for functions $G$ in a dense subset of $C([0,1])$ with respect to the uniform topology.  We will use two different dense sets, namely the space $C^1([0,1])$ in the case $\theta<1$ and the space $C^2([0,1])$ in the case $\theta\geq 1$, which are both dense in $C([0,1])$ with respect to the uniform topology.  For that purpose, we split the proof according to $\theta\geq 1$ and $\theta<1$. When $\theta\geq 1$ we  prove (\ref{eq:tight_1}) directly for functions  $G \in C^{2}([0,1])$ and we conclude that the sequence is tight. For  $\theta<1$, we prove (\ref{eq:tight_1}) first for functions $G \in C^{2}_{c}(0,1)$ and then we extend it, by a $L^1$ approximation procedure which is explained below, to functions $G\in C^1([0,1])$.

Recall from (\ref{Dynkin'sFormula_ssep})  that $M_{t}^{N}(G)$ is a martingale with respect to the natural filtration  $\{\mathcal{F}_{t}\}_{t\geq 0}$. Then
\begin{equation*}
\begin{split}
 &{\bbbp}_{\mu _{N}}\Big(\eta_{\cdot}:\left\vert \langle\pi^{N}_{\tau+ \bar\tau},G\rangle-\langle\pi^{N}_{\tau},G\rangle\right\vert > \varepsilon \Big)\\
 =&{\bbbp}_{\mu _{N}}\Big(\eta_{\cdot}: \Big| M_{\tau}^{N}(G)- M_{\tau+ \bar\tau}^{N}(G) + \int_{\tau}^{\tau+ \bar\tau} N^2 \mc L_{N} \langle \pi_{s}^{N},G \rangle ds  \Big| > \varepsilon \Big)\\
 \leq &{\bbbp}_{\mu _{N}}\Big(\eta_{\cdot}: \Big| M_{\tau}^{N}(G)- M_{\tau+ \bar\tau}^{N}(G)  \Big| > \frac \varepsilon 2 \Big)\\
 +&{\bbbp}_{\mu _{N}}\Big(\eta_{\cdot}: \Big|  \int_{\tau}^{\tau+ \bar\tau} N^2 \mc L_{N} \langle \pi_{s}^{N},G \rangle ds  \Big| > \frac \varepsilon 2 \Big).
 \end{split}
\end{equation*}
Applying Chebychev's inequality (resp. Markov's inequality) in the first (resp. second) term on the right hand side of last inequality,  we can bound the previous expression from above by
\begin{equation*}
\begin{split}
 \frac{2}{\varepsilon^2}{{E}}_{\bbbp_{\mu _{N}}}\Big[ \Big( M_{\tau}^{N}(G)- M_{\tau+ \bar\tau}^{N}(G)  \Big)^2 \Big]
 +\frac{2}{\varepsilon}{{E}}_{\bbbp_{\mu _{N}}}\Big[ \Big| \int_{\tau}^{\tau+ \bar\tau} N^2 \mc L_{N} \langle \pi_{s}^{N},G \rangle ds\Big|\Big].
 \end{split}
\end{equation*}
Therefore, in order to prove (\ref{eq:tight_1}) 
it is enough to show that
\begin{eqnarray} 
\label{eq:tight_2}
\displaystyle \lim _{\delta\rightarrow 0} \limsup_{N\rightarrow\infty} \sup_{\tau  \in \mathcal{T}_{T},\bar\tau \leq \delta}{{E}}_{\bbbp_{\mu _{N}}}\Big[ \Big| \int_{\tau}^{\tau+ \bar\tau}N^2 \mc L_{N}\langle \pi_{s}^{N},G\rangle ds \Big|\Big] = 0
\end{eqnarray}
and
 \begin{equation} 
 \label{eq:tight_3}
\displaystyle \lim _{\delta\rightarrow 0} \limsup_{N\rightarrow\infty} \sup_{\tau  \in \mathcal{T}_{T},\bar\tau \leq \delta}{{E}}_{\bbbp_{\mu _{N}}}\left[\left( M_{\tau}^{N}(G)- M_{\tau+ \bar\tau}^{N}(G) \right)^{2}  \right]=0.
\end{equation}

Let us start by proving \eqref{eq:tight_2}.  Given a test function $G$, we will show that there exists a constant $C$ such that
\begin{equation}
\label{eq:claim_ssep}
N^2 \mc L_N ( \langle \pi^N_{s}, G \rangle ) \le  C
\end{equation}
for any $s\le T$. We start with the case $\theta\geq1$. For that purpose, recall 
\eqref{gen_action_ssep}. 
Note that, since $|\eta_{sN^2} (x)| \le 1$ for all $s\in[0,t]$ and since $G \in C^{2}([0,1])$, we have that
\begin{equation*}\begin{split}
\Big|\langle \pi^N_s,\Delta_NG\rangle+\nabla _N^+G(0)\eta_{sN^2}(1)-\nabla_N^-G(1)\eta_{sN^2}(N-1)\Big|\leq 2\|G''\|_\infty+2\|G\|_\infty
\end{split}
\end{equation*}
and 
\begin{equation*}\begin{split}
\Big|{\kappa}N^{1-\theta}G\Big(\tfrac{1}{N}\Big)(\alpha-\eta_{sN^2}(1))+{\kappa}N^{1-\theta}G\Big(\tfrac{N-1}{N}\Big)&(\beta-\eta_{sN^2}(N-1))\Big|\\&\leq 4\kappa N^{1-\theta}\|G\|_\infty\\&\leq4\kappa\|G\|_\infty. 
\end{split}
\end{equation*}
This proves \eqref{eq:claim_ssep} for the case $\theta\geq 1.$ In the case $\theta<1$, we take $G \in C_c^{2}([0,1])$ and we see that in this case \eqref{gen_action_ssep} reduces to 
$
\langle \pi^N_s,\frac 12\Delta_NG\rangle
$ whose absolute value is bounded from above by $\|G''\|_\infty$ and this proves \eqref{eq:claim_ssep} for the case $\theta<1.$

Let us now prove  \eqref{eq:tight_3}. Applying Dynkin's formula with $F(\cdot, \cdot)$ given by \eqref{eq:F_for_Dynkin} we get that  
\begin{equation}\label{QV_mart}
\displaystyle\left( M^{N}_{t}(G)\right)^{2}-\int^{t}_{0} N^2\left[\mc  L_{N} \langle\pi^{N}_{s},G \rangle^{2}- 2\langle\pi^{N}_{s},G \rangle \mc L_{N} \langle\pi^{N}_{s},G \rangle\right]ds,
\end{equation} is a martingale with respect to the natural filtration  $\{\mathcal{F}_{t}\}_{t\ge 0}$. A simple computation shows that
\begin{equation*}\begin{split}
&N^2\left[\mc  L_{N,0} \langle\pi^{N}_{s},G \rangle^{2}- 2\langle\pi^{N}_{s},G \rangle \mc L_{N,0} \langle\pi^{N}_{s},G \rangle\right]
\\=&\frac{1}{2N^2}\sum_{x=1}^{N-2}(\eta_{sN^2}(x)-\eta_{sN^2}(x+1))^2(\nabla_N^+G(\tfrac{x}{N}))^2
\end{split}\end{equation*}
and by using the fact that $|\eta_{sN^2} (x)| \le 1$ for all $s\in[0,t]$ last expression is bounded from above by $\frac{2}{N}\|G'\|_\infty$. 
On the other hand, we also have that
\begin{equation*}\begin{split}
&N^2\left[\mc  L_{N,b} \langle\pi^{N}_{s},G \rangle^{2}- 2\langle\pi^{N}_{s},G \rangle \mc L_{N,b} \langle\pi^{N}_{s},G \rangle\right]\\
=&\frac{\kappa}{2N^\theta}\Big[c_1(\eta_{sN^2},\alpha)G(\tfrac{1}{N})^2+c_{N-1}(\eta_{sN^2},\beta)G(\tfrac{N-1}{N})^2\Big]
\end{split}
\end{equation*}
and by using the fact that $|\eta_{sN^2} (x)| \le 1$ for all $s\in[0,t]$ last expression is bounded from above by $\frac{4\kappa}{N^\theta}\|G\|^2_\infty$.

This ends the proof of tightness in the case $\theta\geq 1$, since  $C^{2}([0,1])$ is a dense subset of $C([0,1])$ with respect to the uniform topology. Nevertheless, for $\theta < 1$, since we considered functions $G\in C^{2}_{c}(0,1)$, last display is equal to zero. Therefore,  we have proved (\ref{eq:tight_2}) and (\ref{eq:tight_3}), and thus (\ref{eq:tight_1}), but for functions $G\in C^{2}_{c}(0,1)$ and, as mentioned above, we need to extend this result to functions in $C^1([0,1])$. To accomplish that,  we take a function $G \in C^1([0,1])\subset L^{1}([0,1])$,  and we take  a sequence of functions $\{G_{k}\}_{k \geq 0} \in C^{2}_{c}(0,1)$ converging to $G$, with respect to the $L^{1}$-norm, as $k \to \infty$. Now, since the probability in (\ref{eq:tight_1}) is less or equal than 
\begin{eqnarray*}\label{T7} \nonumber
 &&{\bbbp}_{\mu _{N}}\Big(\eta_{\cdot} :\left\vert \langle\pi^{N}_{\tau+ \bar\tau},G_{k}\rangle-\langle\pi^{N}_{\tau},G_{k}\rangle\right\vert > \dfrac{\varepsilon}{2} \Big)  \\\nonumber
&+&{\bbbp}_{\mu _{N}}\Big(\eta_{\cdot}:\left\vert \langle\pi^{N}_{\tau+ \bar\tau},G-G_{k}\rangle-\langle\pi^{N}_{\tau},G-G_{k}\rangle\right\vert > \dfrac{\varepsilon}{2} \Big) 
\end{eqnarray*}
and  since $G_{k}$ has compact support, from the computation above, it remains only to check that the last probability vanishes as $N \to \infty$ and then $k \to \infty$. 
For that purpose, we use the fact that 
\begin{equation}\label{T8}
\left\vert \langle\pi^{N}_{\tau+ \bar\tau},G-G_{k}\rangle-\langle\pi^{N}_{\tau},G-G_{k}\rangle\right\vert\leq \dfrac{2}{N}\sum _{x\in \Lambda_{N}}\left\vert(G-G_{k})(\tfrac{x}{N})\right\vert, 
\end{equation}
and we use the estimate
\begin{equation*}
\begin{split}
\dfrac{1}{N}\sum _{x\in \Lambda_{N}}\left\vert(G-G_{k})(\tfrac{x}{N})\right\vert & \le \sum _{x\in \Lambda_{N}}  \int_{\frac xN}^{\frac {x+1}{N}} \left\vert(G-G_{k})(\tfrac{x}{N}) -(G-G_k) (q) \right\vert dq \\
&+\,  \int_{0}^{1}\vert (G-G_{k})(q)\vert dq\\
& \le \cfrac{1}{N} \| (G- G_k)' \|_{\infty} + \int_{0}^{1}\vert (G-G_{k})(q)\vert dq.
\end{split}
\end{equation*}
The result follows by first  taking $N\to\infty$ and then $k \to \infty$.

\subsection{The limit point}\label{limit point}

Here, we prove at first that all limit points $\bbbq$ of the sequence $\{\bbbq_{N}\}_{N\geq 1}$ are concentrated on measures absolutely continuous with respect to the
Lebesgue measure, that are equal to $g(q)dq$ at the initial time and finally that $\bbbq$ is concentrated on trajectories of measures satisfying $\pi_t(dq)=\rho_t(q)dq$, where $\rho_t(\cdot)$ is the  weak solution of the corresponding hydrodynamic equation.  Let $\bbbq$ be a limit point of $\{\bbbq_{N}\}_{N\geq 1}$.

\subsubsection{-Characterization of absolutely continuity:}
\label{sec:abs_cont}
We start by showing that $\bbbq$ is concentrated on measures which are absolutely continuous with respect to the Lebesgue measure. Fix a continuous function $G:[0,1]\rightarrow{\bbbr}$. Since
\begin{equation*}
\sup_{t\in{[0,T]}}|\langle\pi_{t}^{N},G\rangle|\leq{\frac{1}{N}\sum_{x\in{\Lambda_N}}|G(\tfrac{x}{N})}|,
\end{equation*}
which is a consequence of the fact  of having at most one particle per site, the function that associates to each trajectory $\pi_{.}$,
$\sup_{t\in{[0,T]}}|\langle\pi_{t},G\rangle|$ is continuous. As a consequence, all limit points are concentrated in trajectories $\pi_{t}$ such that
\begin{equation*}
|\langle\pi_{t},G\rangle|\leq{\int_{0}^1|G(q)|dq}.
\end{equation*}
In order to show that the measure $\pi_t$ is absolutely continuous with respect to the Lebesgue measure, that we denote by $Leb$, we have to show that
for each set $A$ such that $Leb(A)=0$, then $\pi_{t}(A)=0$. With this purpose, we use last estimate for a sequence of continuous functions $\{G_{N}\}_{N\geq 1}$ that converges to the indicator function over the set $A$ and the result follows. Concluding, we have just proved that
\begin{equation*}
\bbbq\Big(\pi_{\cdot}: \pi_{t}(dq)=\pi(t,q)dq, \forall t\in[0,T]\Big)=1
\end{equation*}
i.e. $\pi_{t}(dq)$ is absolutely continuous with respect to the Lebesgue measure with a density $\pi(t,q)$.
\subsubsection{-Characterization of the initial measure:}
\label{sec:charc_initial_measure}
Here we show that $\bbbq$ is concentrated on a Dirac measure equal to $g(q)dq$ at time $0$.  For that purpose, fix $\varepsilon>0$. From the results of Section \ref{sec:tightness}, we know, from the weak convergence over a
subsequence and Portmanteau's Theorem, that:
\begin{equation*}
\begin{split}
&\bbbq\Big(\Big|\frac{1}{N}\sum_{x\in{\Lambda_{N}}}G(\tfrac{x}{N})\eta_{0}(x)-\int_{0}^1G(q)g(q)dq\Big|>\varepsilon\Big)\\
&\leq{\liminf_{K\rightarrow{+\infty}}\bbbq_{N_{k}}
\Big(\Big|\frac{1}{N}\sum_{x\in{\Lambda_{N}}}G(\tfrac{x}{N})\eta_{0}(x)-\int_{0}^1G(q)g(q)dq\Big|>\varepsilon\Big)}\\
&=\liminf_{K\rightarrow{+\infty}}\mu_{N_{k}}
\Big(\Big|\frac{1}{N}\sum_{x\in{\Lambda_{N}}}G(\tfrac{x}{N})\eta(x)-\int_{0}^1G(q)g(q)dq\Big|>\varepsilon\Big).
\end{split}
\end{equation*}
This last limit  is equal to zero, by the hypothesis of $\mu_{N}$ being associated to the profile $g(\cdot)$, see Definition \ref{def:meas_ass}. This shows that
\begin{equation*}
\bbbq\Big(\pi_{\cdot}: \pi_{0}(dq)=g(q)dq\Big)=1.
\end{equation*}

\subsubsection{-Characterization of the density $\pi(t,q)$:}
\label{sec:charac_limit_p}

Up to here  we know that all limit points  $\bbbq $ of the sequence sequence $\{\bbbq_N\}_{N\geq 1}$ are concentrated on trajectories $\pi_{t}(dq)$ which are absolutely continuous with respect to the Lebesgue measure, that is, $\pi_{t}(dq)=\pi(t,q)dq$. Moreover,  we also know  that all limit points $\bbbq$ of the sequence $\{\bbbq_N\}_{N\geq 1}$ are such that the initial trajectory is  a Dirac measure equal to $g(q)dq$. Now we prove that all limit points are concentrated on  trajectories of measures of the form $\rho_t(q)dq$, that is we are going to show that $\pi(t,q)=\rho_t(q)$ and that $\rho_{t}(\cdot)$ is a weak solution of the corresponding hydrodynamic equation. For that purpose, let $\bbbq$ be a limit point of the sequence $\lbrace\bbbq_{N}\rbrace_{N \geq 1}$, whose existence follows from the computations of Section \ref{sec:tightness}  and assume, without lost of generality, that $\lbrace\bbbq_{N}\rbrace_{ N \ge 1}$ converges to $\bbbq$, as $N\to+\infty$.  

\begin{proposition}
\label{prop:weak_sol_car}
If $\bbbq$ is a limit point of $ \{\bbbq_{N}\}_{N\in\bbbn}$  then 
$$\bbbq\left( \pi _{\cdot}: F_\theta= 0,\forall t\in [0,T],\, \forall G \in C_\theta\,\right)=1,$$
where
$$
F_\theta = 
\begin{cases}
F^c_{Dir}, \;\textrm{if }\ \theta<0,\\
F_{Dir}, \;\textrm{if }\theta\in[0,1),\\
F_{Rob}, \; \textrm{if }\theta\geq 1,
\end{cases}
\quad \textrm{and }\quad
C_\theta = 
\begin{cases}
C_c^{1,2} ([0,T]\times[0,1]), \;\textrm{if }\theta<0,\\
C_0^{1,2} ([0,T]\times[0,1]), \;\textrm{if }\theta\in[0,1),\\
C^{1,2} ([0,T]\times[0,1]), \; \textrm{if }\theta\geq 1.
\end{cases}
$$
\end{proposition}
\begin{proof}
We consider the case $\theta\geq 1$.
Note that we need  to verify, for $\delta > 0$ and $ G\in C^{1,2} ([0,T]\times[0,1])$,   that  
\begin{eqnarray}\label{prob_charac}
&&\bbbq\left(\pi _{\cdot}\in \mc D([0,T], \mathcal{M^{+}}): \sup_{0\le t \le T} \left\vert F_{Rob} \right\vert>\delta\right)=0,
\end{eqnarray}
Recall $F_{Rob}$ from \eqref{eq:Robin integral-g_ssep} and note that,  due to the terms  that involve $\rho_s(1)$ and $\rho_s(0)$ and  that appear in $F_{Rob}$, the set inside the  probability  in \eqref{prob_charac} is not an open set in the Skorohod space,  and as a consequence  we cannot use directly Portmanteau's Theorem. To avoid this difficulty, we fix $\varepsilon>0$ and we consider two approximations of the identity given by \begin{equation}\label{eq:aprox_ident}
\iota^0_\varepsilon(q)=\frac{1}{\varepsilon}\textrm{1}_{(0,\varepsilon)}(q)\quad \textrm{and}\quad \iota^1_\varepsilon(q)=\frac{1}{\varepsilon}\textrm{1}_{(1-\varepsilon,1)}(q)
\end{equation} and we sum and subtract  to $\rho_{s}(0)$  and to  $\rho_{s}(1)$ the mean  
\begin{equation}\label{convolutions}
\langle \pi_s, \iota^0_\varepsilon\rangle= \tfrac{1}{\varepsilon}\int_{0}^{\varepsilon}\rho_{s}(q)dq\quad \textrm{and}\quad \langle\pi_s, \iota^1_\varepsilon\rangle=\tfrac{1}{\varepsilon}\int_{1-\varepsilon}^{\varepsilon}\rho_{s}(q)dq,
\end{equation}
 respectively.  Above we used the fact that $\bbbq$ is concentrated on trajectories $\pi_{t}(dq)$ which are absolutely continuous with respect to the Lebesgue measure: $\pi_{t}(dq)=\rho_{t}(q)dq$.
Thus,  we bound the probability in \eqref{prob_charac} from above by the sum of the following terms
\begin{equation}\label{RC6}
\begin{split}
&\bbbq\left(  \sup_{0\le t \le T} \Big|\int_0^1 \rho_{t}(q)  G_{t}(q) \,dq  -\int_0^1  \rho_{0}(q) G_{0}(q) \,dq \right.\\
-&\left.\int_0^t\int_0^1 \rho_{s}(q)  \Big(\frac 12\Delta+\partial_s\Big) G_{s}(q)\,dqds 
   -\frac 12 \int^{t}_{0}  G_{s}(0)\alpha+ G_{s}(1)\beta \,ds  \right.\\
  +& \left.\frac 12\int^{t}_{0}  \langle\pi_s, \iota^1_\varepsilon\rangle\Big( \partial_{q}G_{s}(1)+  G_{s}(1)\Big)\,ds - \frac 12 \int^{t}_{0}\langle \pi_s, \iota^0_\varepsilon\rangle\Big(\partial_{q}G_{s}(0) -  G_{s}(0\Big)\, ds \Big|>\dfrac{\delta}{4}\right),
\end{split}
\end{equation}
\begin{equation}
\label{RC71}
\bbbq \left(  \Big| \int^{1}_{0} (\rho_{0}(q)-g(q)) G_{0}(q)\, dq \Big|>\dfrac{\delta}{4}\right),
\end{equation}
\begin{equation}
\label{RC7}
\sum_{k\in\{0,1\}}\bbbq\left(  \sup_{0\le t \le T}  \Big| \frac 12\int^{t}_{0} ( \rho_{s}(k)-\langle \pi_s, \iota^k_\varepsilon\rangle )\left[ G_{s}(k)- \partial_{q}G_{s}(k)  \right]ds  \Big|>\dfrac{\delta}{4}\right),
\end{equation}
and
we note that the terms  in (\ref{RC7})  converge to $0$ as $\varepsilon \to 0$  since we are comparing $\rho_s(0)$ and $\rho_s(1)$ with the   averages \eqref{convolutions} around  $0$  and  $1$, respectively. Moreover,  (\ref{RC71}) is equal to zero since $\bbbq$ is a limit point of $\{\bbbq_N\}_{N\geq1}$ and $\bbbq_N$ is induced by a measure $\mu_N$ which is associated to the profile $g(\cdot)$.  Note that in  (\ref{RC6}) we still cannot use Portmanteau's Theorem, since the functions $\iota^0_\varepsilon$ and $\iota_\varepsilon^1$ are not continuous. Nevertheless,  by approximating each one of these functions by continuous functions  in such a way that the error vanishes as $\varepsilon \to 0$ then, from Proposition A.3 of \cite{FGN} we can use Portmanteau's Theorem and   bound (\ref{RC6}) from above by
\begin{equation}
\label{RC1111}
\begin{split}
&\liminf_{N\to\infty}\,\bbbq_{N}\Big( \sup_{0\le t \le T} \Big| \int_0^1 \rho_{t}(q)  G_{t}(q) \,dq  -\int_0^1  \rho_{0}(q)  G_{0}(q) \,dq \\
&-\int_0^t\int_0^1 \rho_{s}(q)  \Big(\frac 12\Delta+\partial_s\Big) G_{s}(q)\,dqds -\frac 12\int^{t}_{0}  G_{s}(0)\alpha+ G_{s}(1)\beta \,ds  \\
&  -\frac 12\int^{t}_{0}\langle \pi_s, \iota^0_\varepsilon\rangle\Big(\partial_{q}G_{s}(0) -  G_{s}(0\Big)\, ds + \frac 12\int^{t}_{0}  \langle\pi_s, \iota^1_\varepsilon\rangle\Big(\partial_{q}G_{s}(1)+  G_{s}(1)\Big)\,ds \Big|>\dfrac{\delta}{2^{4}}\Big).
\end{split}
\end{equation} 
Summing and subtracting $\displaystyle\int_{0}^{t} N^{2}\mc L_{N}\langle \pi_{s}^{N},G_{s}\rangle ds$ to the term inside the supremum in  (\ref{RC1111}), recalling \eqref{Dynkin'sFormula_ssep} and \eqref{boxes_ssep}, the definition of $\bbbq_N$,
we bound \eqref{RC1111}  from above by the sum of the next two terms
\begin{equation}\label{RC12}
\liminf_{N\to\infty}\,\bbbp_{\mu_N}\left( \sup_{0\le t \le T} \left\vert M_{t}^{N}(G) \right\vert>\dfrac{\delta}{2^{5}}\right),
\end{equation} 
and
\begin{equation}
\label{RC13}
\begin{split}
&\liminf_{N\to\infty}\,\bbbp_{\mu_N}\left(  \sup_{0\le t \le T} \Big|\int_0^t N^{2}\mc L_{N}\langle \pi_{s}^{N},G_{s}\rangle\,ds-\int_0^t\int_0^1 \rho_{s}(q) \frac 12\Delta G_{s}(q)\,dqds \right.\\
 & -\frac 12 \int^{t}_{0}\overrightarrow{\eta}^{\varepsilon N}_{sN^2}(1)\Big(\partial_{q}G_{s}(0) - G_{s}(0\Big)\, ds +\frac 12 \int^{t}_{0}  \overleftarrow{\eta}^{\varepsilon N}_{sN^2}(N-1)\Big( \partial_{q}G_{s}(1)+   G_{s}(1)\Big)\,ds\\
 &\left.  -\frac 12\int^{t}_{0}  G_{s}(0)\alpha+ G_{s}(1)\beta\,ds  \Big|>\dfrac{\delta}{2^{5}}\right).
\end{split}
\end{equation} 
 Doob's inequality together with the computations right below \eqref{QV_mart} show that (\ref{RC12}) goes to $0$ as $N\to\infty$. Finally, (\ref{RC13}) can be rewritten as  
 \begin{equation}
 \label{RC14}
 \begin{split}
&\liminf_{N\to\infty}\,\bbbp _{\mu _{N}}\left( \sup_{0\le t \le T}  \Big| \int_0^t N^{2}\mc L_{N}\langle \pi_{s}^{N},G_{s}\rangle\,ds-\int_0^t \langle \pi_{s}^{N} ,\frac 12\Delta G_{s}\rangle\,ds \right. \\
& -\frac 12\int^{t}_{0}\overrightarrow{\eta}^{\varepsilon N}_{sN^2}(1)\Big(\partial_{q}G_{s}(0) - G_{s}(0\Big)\, ds + \frac 12\int^{t}_{0}  \overleftarrow{\eta}^{\varepsilon N}_{sN^2}(N-1)\Big( \partial_{q}G_{s}(1)+ G_{s}(1)\Big)\,ds \\
&\left. -\frac 12\int^{t}_{0}  G_{s}(0)\alpha+ G_{s}(1)\beta\,ds   \Big|>\dfrac{\delta}{2^{5}}\right).
\end{split}
\end{equation} 
Now, from (\ref{gen_action_ssep}) we can bound from above the probability in (\ref{RC14}) by the sum of the following terms
\begin{equation}
\label{RC15}
\bbbp_{\mu_{N}}  \left(\sup_{0\le t \le T} \Big| \cfrac{1}{N}\int_{0}^{t} \sum_{x\in \Lambda_N} \frac 12 \Delta_NG_s(\tfrac{x}{N})\eta_{sN^2}(x) ds  -  \int_{0}^{t}\left\langle \pi_{s}^{N},\frac 12\Delta G_{s} \right\rangle   \, ds \Big|>\dfrac{\delta}{2^{6}}\right),
\end{equation}
\begin{equation}
\label{RC16}
\begin{split}
&\bbbp_{\mu_{N}}  \left( \sup_{0\le t \le T}  \Big| \frac 12\int_{0}^{t} \nabla _N^+G_s(0)\eta_{sN^2}(1)- \overrightarrow{\eta}^{\varepsilon N}_{sN^{2}}(1) \partial_{q}G_{s}(0)   \, ds  \Big|>\dfrac{\delta}{2^{6}}\right),
\end{split}
\end{equation}
and
\begin{equation}
\label{RC18}
\begin{split}
 &\bbbp_{\mu_{N}}\left(\sup_{0\le t \le T} \Big| \frac 12 \int_{0}^{t}{\kappa}N^{1-\theta}G_s\Big(\tfrac{1}{N}\Big)(\alpha-\eta_{sN^2}(1))-G_{s}(0)(\alpha-\overrightarrow{\eta}^{\varepsilon N}_{sN^2}(1)) ds\Big| > \dfrac{\delta}{2^{6}}\right) 
 \end{split}
\end{equation}
and two other  terms which are very similar to the two previous ones but related to the  action of the right boundary dynamics given by $\mc L_{N,b}^{N-1}$.   Applying a Taylor expansion on the test function $G$ it is easy to show that  (\ref{RC15}) goes to $0$ as $N\to \infty$. Also  by Taylor expansion,  (\ref{RC16})  can be bounded from above by 
\begin{equation}
\label{RC20}
\bbbp_{\mu_{N}}  \left( \sup_{0\le t \le T} \Big| \int_{0}^{t} \partial_{q}G_{s}(0)(  \eta_{sN^2}(1)- \overrightarrow{\eta}^{\varepsilon N}_{sN^{2}}(1))ds \Big|>\dfrac{\delta}{2^{8}}\right),
\end{equation}
plus a term that vanishes as $N\to\infty$. 
Using Lemma \ref{Rep-Neumann} we see that (\ref{RC20}) vanishes as $N\to \infty$. The term  (\ref{RC18}) can be estimated using exactly the same argument that we just used, that is: Taylor expansion on $G$ plus Lemma \ref{Rep-Neumann}. For the terms related to the right boundary the argument is the same and with this  we finish the proof.

We leave the other case, namely $\theta<1$ for the reader. This case is even simpler than the previous one and for the interested reader we refer to, for example,  \cite{Adriana,BGJO}.

\end{proof}
\subsection{Hydrostatic limit}
\label{sec:hydrostatic}

In this section we prove that the hydrodynamic limit holds when we start the system from the stationary measure $\mu_{ss}$, see Section \ref{sec:stat_mea}.  By looking at the statement of Theorem \ref{th:hyd_ssep} we see that, in fact, to conclude we only need to show  the next result.

\begin{proposition}
\label{th:hyd_stat_ssep}
 Let $\mu _{ss}$ be the stationary measure for the Markov process $\{\eta_{tN^2}\,:\,t\geq0\}$ with generator $N^2\mc L_N$. Then, $\mu_{ss}$ is associated to the profile $\bar\rho:[0,1]\to[0,1]$ given on $q\in{(0,1)}$ by \eqref{eq: stat_sol}, that is
 \begin{equation*}
\bar{\rho}(q)=\left\{
\begin{array}{ll}
(\beta-\alpha)q+\alpha\,;\,\theta<1,\\
\frac{{\kappa}(\beta-\alpha)}{2+{\kappa}}q+\alpha+\frac{\beta-\alpha}{2+{\kappa}}\,;\,\theta=1,\\
\frac{\beta+\alpha}{2}\,;\,\theta>1,
\end{array}
\right.
\end{equation*} which is a stationary solution of the corresponding hydrodynamic equation, see \eqref{eq:stat_sol_dir} and \eqref{eq:stat_sol_rob}.
\end{proposition}
\begin{proof}
Recall from \eqref{assoc_mea}, that we need to prove:
\begin{equation*}
  \lim _{N\to\infty } \mu _{ss}\Big( \eta \in \Omega_{N} : \Big| \dfrac{1}{N}\sum_{x \in \Lambda_{N} }G\left(\tfrac{x}{N} \right)\eta({x}) - \int_{0}^1G(q)\rho_{0}(q)dq \Big|   > \delta \Big)= 0
\end{equation*}
for any continuous function $G:[0,1	]\to\bbbr$. 
By Markov's and triangular inequalities, we bound the previous probability from above by
\begin{equation}\label{eq:stat_hydro_proof}
\begin{split}
&\frac{1}{\delta}{E}_{\bbbp_{\mu _{ss}}}\Big[  \Big| \dfrac{1}{N}\sum_{x \in \Lambda_{N} }G\left(\tfrac{x}{N} \right)\Big(\eta({x}) -\rho_{ss}^N(x)\Big)\Big|\\+&\Big|\frac{1}{N}\sum_{x \in \Lambda_{N} }G\left(\tfrac{x}{N} \right)\rho_{ss}^N(x)- \int_{0}^1G(q)\bar\rho(q)dq \Big| \Big]\\
\leq & \frac{1}{\delta}{E}_{\bbbp_{\mu _{ss}}}\Big[  \Big| \dfrac{1}{N}\sum_{x \in \Lambda_{N} }G\left(\tfrac{x}{N} \right)\Big(\eta({x}) -\rho_{ss}^N(x)\Big)\Big|\Big]\\+&\frac{1}{\delta}\Big|\frac{1}{N}\sum_{x \in \Lambda_{N} }G\left(\tfrac{x}{N} \right)\rho_{ss}^N(x)- \int_{0}^1G(q)\bar\rho(q)dq \Big|.
\end{split}
\end{equation}
The last  term  can be bounded from above by
\begin{equation*}
\begin{split}
&\frac{1}{\delta}\Big|\frac{1}{N}\sum_{x \in \Lambda_{N} }G\left(\tfrac{x}{N} \right)\Big(\rho_{ss}^N(x)-\bar\rho\Big(\tfrac{x}{N}\Big)\Big)\Big|\\+&\frac{1}{\delta}\Big|\frac{1}{N}\sum_{x \in \Lambda_{N} }G\left(\tfrac{x}{N} \right)\bar\rho\Big(\tfrac{x}{N}\Big) -\int_{0}^1G(q)\bar\rho(q)dq \Big|.
\end{split}
\end{equation*}
The term at the left hand side of last expression is bounded from above by 
\begin{equation*}
\begin{split}
\frac{1}{\delta}\frac{1}{N}\sum_{x \in \Lambda_{N} }\Big|G\left(\tfrac{x}{N} \right)\Big|\Big|\rho_{ss}^N(x)-\bar\rho\Big(\tfrac{x}{N}\Big)\Big|\leq \frac{\|G\|_{\infty}}{\delta}\max_{x\in\Lambda_N}\Big|\rho_{ss}^N(x)-\bar\rho\Big(\tfrac{x}{N}\Big)\Big|
\end{split}
\end{equation*}
where from \eqref{eq:conve_emp_prof_stat_prof} it vanishes as  $N\to\infty$, while  the term at the right hand side also vanishes as $N\to\infty$ since we compare the Riemann sum with the corresponding converging integral.

To finish the proof  it remains to analyse the third term in \eqref{eq:stat_hydro_proof}. By the Cauchy-Schwarz's inequality the expectation appearing in that term can bounded  from above by
 \begin{equation*}
\begin{split}
& \Big(  \Big| \dfrac{1}{N^2}\sum_{x \in \Lambda_{N} }G\left(\tfrac{x}{N} \right){E}_{\bbbp_{\mu _{ss}}}\Big[(\eta({x}) -\rho_{ss}^N(x))^2\Big]\\+&\frac{2}{N}\sum_{x<y}G\left(\tfrac{x}{N} \right)G\left(\tfrac{y}{N} \right){E}_{\bbbp_{\mu _{ss}}}\Big[(\eta({x}) -\rho_{ss}^N(x))(\eta({y}) -\rho_{ss}^N(y))\Big]\Big)^\frac 12\\\leq&\Big(\frac{C\|G\|_\infty}{N}+2\|G\|_\infty\max_{x<y}\varphi_{ss}^N(x,y)\Big)^\frac 12.
\end{split}
\end{equation*}
From \eqref{eq:bound_stat_corr} the previous expression vanishes as $N\to\infty$. This finishes the proof.
\end{proof}

Note that the proof presented above uses the information about the two point correlation function which is not always easy to obtain. We refer the reader to \cite{flm} for another proof of this results without using the knowledge on the  correlations.

\section{Symmetric exclusion with long jumps in contact with reservoirs}
 \label{chap2}
 
 \subsection{The model}
\label{sec:model_LJ}
 In this section we want to generalize the results of the previous section to the case where particles can give jumps arbitrarily large.  As in the previous section, the bulk consists in the set of points $\Lambda_N=\{1,..., N-1\}$ and we artificially add two end points $x=0$ and $x=N$. Now, we explain the dynamics of the models we consider  and we start by describing the conditions on the  jump rate.  For that purpose,  let $p:\bbbz\times \bbbz\rightarrow{[0,1]}$ be a transition probability  such that $p(x,y)=p(y-x)$ and which is symmetric. We are going to discuss two cases: the first one, when $p(\cdot)$ has  finite  variance and the second one when $p(\cdot)$ has  infinite variance.   Note that since $p(\cdot)$ is symmetric it has mean zero, that is:
$$\sum_{z\in \bbbz}z p(z)=0.$$
We denote $m=\sum_{z\ge 1} z p(z)$. As an example we consider $p(\cdot)$ given by $p(0)=0$ and  \begin{equation}
\label{eq:choice_p}
p(z) = 
\dfrac{c_{\gamma}}{\vert z\vert^{\gamma+1}},
\end{equation}
 for $z\neq 0$, 
where  $c_{\gamma}$ is a normalizing constant.   For simplicity of the presentation we stick to this choice of $p(\cdot)$ along this section but we note that many of our results are true, in the case where $p(\cdot)$ has finite variance, in a more general setting where we only assume $p(\cdot)$ to be translation invariant and  mean zero.

We consider the process in contact with  stochastic   reservoirs  at the left and the right of the bulk. We fix four parameters $\alpha, \beta\in[0,1]$, $\kappa>0$ and   $\theta \in \bbbr$, so that particles can get in the bulk of the system from the site $x=0$ to any site $y\in\Lambda_N$ at rate  $\alpha\kappa N^{-\theta} p(y)$    or leave the bulk from any site $y\in\Lambda_N$ to the site $x=0$ at rate  $(1-\alpha)\kappa N^{-\theta} p(y)$; and particles can get in   the bulk to any site $y\in\Lambda_N$  from the site $x=N$ at rate $\beta\kappa N^{-\theta} p(N-y)$ or leave the bulk  from any site $y\in\Lambda_N$ to the site $x=N$  at rate  $(1-\beta)\kappa N^{-\theta} p(N-y)$.

We define the  dynamics of the process in the following way.  We start with the bulk dynamics. Each pair of sites of the bulk $\{x,y\} \subset \Lambda_N$ carries a Poisson process of intensity $p(y-x) /2$.  Poisson processes associated to different bonds are independent. If for the configuration $\eta$, the clock associated to the bound $\{x,y\}$ rings, then we exchange the value of the occupation variables $\eta(x)$ and $\eta(y)$ at rate $p(y-x) /2$. Now we explain the dynamics at the boundary. Each pair of sites $\{0,x\}$ with  $x\in\Lambda_N$ carries 	two Poisson processes, all of them being independent. If for the configuration $\eta$, the clock associated to the Poisson process of the oriented bond $\{0,x\}$ (resp. $\{x,0\}$) rings, then we change the value $\eta(x)$ into $1-\eta(x)$ with rate $\kappa N^{-\theta}p(x) \alpha (1-\eta(x))$ (resp. $\kappa N^{-\theta}p(x)(1- \alpha)\eta(x)$).  At the right boundary the dynamics is similar but instead of $\alpha$ the intensity is given by $\beta$. Observe that the reservoirs ($x=0$ and $x=N$) add and remove particles on all the sites of the bulk $\Lambda_N$, and not only at the boundaries $x=1$ and $x=N-1$ as happened in the model of Section \ref{chap1}, but with a rate that decreases as the distance from the corresponding reservoir increases. We remark that as in the previous section, we could do another interpretation of the previous dynamics  at the boundary,  as follows. Particles  can either be created or annihilated  at any site $x\in\Lambda_N$  according to the following rates:
\begin{itemize}
\item from the left reservoir, from $x=0$ to $y\in\Lambda_N$:
\begin{itemize}
\item creation rate: $\alpha\kappa N^{-\theta}p(y)$, 
\item annihilation rate: $(1-\alpha)\kappa N^{-\theta}p(y)$.
\end{itemize}
\vspace{0.1cm}
\item from the right reservoir, from $x=N-1$ to $y\in\Lambda_N$:
\begin{itemize}
\item creation rate: $\beta\kappa N^{-\theta}p(N-y)$, 
\item annihilation rate: $(1-\beta)\kappa N^{-\theta}p(N-y)$.
\end{itemize}
\end{itemize}
Let us see an illustration of the dynamics just described with $N=11$ and the configuration $\eta=(1,1,0,0,0,0,1,0,1,1)$:

 \begin{center}
 \begin{tikzpicture}[thick, scale=0.85][h!]
 \draw[latex-] (-6.5,0) -- (5.5,0) ;
\draw[-latex] (-6.5,0) -- (5.5,0) ;
\foreach \x in  {-6,-5,-4,-3,-2,-1,0,1,2,3,4,5}
\draw[shift={(\x,0)},color=black] (0pt,0pt) -- (0pt,-3pt) node[below] 
{};
 \node[ball color=black!30!,shape=circle,minimum size=0.7cm] (A) at (1,0.4) {};
    \node[ball color=black!30!,shape=circle,minimum size=0.7cm] (B) at (-5,0.4) {};
       \node[shape=circle,minimum size=0.7cm] (N) at (-5,1.2) {};
        \node[shape=circle,minimum size=0.7cm] (NN) at (-6,0.4) {};
       \node[shape=circle,minimum size=0.7cm] (P) at (-5,2.0) {};
    \node[ball color=black!30!,shape=circle, minimum size=0.7cm] (C) at (-4,0.4) {};
   \node[ball color=black!30!,shape=circle, minimum size=0.7cm] (E) at (3,0.4) {};
         \node[ball color=columbiablue,shape=circle,minimum size=0.7cm] (L) at (-6,0.4) {};
          \node[shape=circle,minimum size=0.7cm] (M) at (-6,1.2) {};
          \node[shape=circle,minimum size=0.7cm] (Q) at (-6,2.0) {};
           \node[ball color=columbiablue,shape=circle,minimum size=0.7cm] (M) at (-6,1.2) {};
           \node[shape=circle,minimum size=0.7cm] (A') at (0,0.4) {};
               \node[shape=circle,minimum size=0.7cm] (A'') at (2,0.4) {};
         
          \node[shape=circle,minimum size=0.7cm] (S) at (5,1.2) {};
           \node[ball color=columbiablue,shape=circle,minimum size=0.7cm] (R) at (5,0.4) {};
              \node[ball color=columbiablue,shape=circle,minimum size=0.7cm] (NN) at (-6,0.4) {};
            \node[ball color=columbiablue,shape=circle,minimum size=0.7cm] (S) at (5,1.2) {};
              \node[shape=circle,minimum size=0.7cm] (T) at (5,2.0) {};
               \node[shape=circle,minimum size=0.7cm] (U) at (4,3.6) {};
                 \node[shape=circle,minimum size=0.7cm] (V) at (5,3.0) {};
               
                      \node[shape=circle,minimum size=0.7cm] (Y) at (4,0.4) {};
               
           \node[ball color=black!30!,shape=circle,minimum size=0.7cm] (Y) at (4,0.4) {};
            
                  \node[shape=circle,minimum size=0.7cm] (W) at (4,3.0) {};
               
           \node[ball color=columbiablue,shape=circle,minimum size=0.7cm] (T) at (5,2.0) {};
              \node[shape=circle,minimum size=0.7cm] (U) at (4,2.0) {};
              \path [->] (T) edge[bend right =60] node[above] {${\textcolor{red}{\kappa N^{-\theta}}\beta p(5)}$} (A');
              
                \path [->] (Y) edge[bend left=60] node[above] {$\textcolor{red}{\kappa N^{-\theta}}{(1-\beta)p(1)}$} (S);
              
               \path [->] (C) edge[bend right=60] node[above] {$\textcolor{red}{\kappa N^{-\theta}}{(1-\alpha)p(2)}$} (M);
               \path [->] (M) edge[bend left=60] node[above] {$\textcolor{red}{\kappa N^{-\theta}}\alpha p(8)$} (A'');
              
    \node[shape=circle,minimum size=0.7cm] (K) at (-0,0.4) {};
        \node[shape=circle,minimum size=0.7cm] (G) at (1,0.4) {};
\tclock{1.9}{-0.8}{columbiablue}{0.8}{0}
\tclock{0.65}{-0.8}{columbiablue}{0.8}{0}
\tclock{4.4}{-0.8}{columbiablue}{0.8}{0}
\tclock{3.15}{-0.8}{columbiablue}{0.8}{0}
\tclock{-4.4}{-0.8}{columbiablue}{0.8}{0}
\tclock{-5.65}{-0.8}{columbiablue}{0.8}{0}
\tclock{-6.9}{-0.8}{columbiablue}{0.8}{0}
\tclock{5.65}{-0.8}{columbiablue}{0.8}{0}
 \end{tikzpicture}
 \end{center}

 \bigskip

The infinitesimal generator of the  process is given by  
\begin{equation}
\label{Generator}
 \mc L_{N} = \mc L_{N,0}+ \mc L_{N,b},
\end{equation}
where  $\mc L_{N,0}$ and $\mc L_{N,b}$  act on functions  $f:\Omega_N \to \bbbr$  as
\begin{equation}\label{generators}
\begin{split}
&(\mc L_{N,0} f)(\eta) =\cfrac{1}{2} \, \sum_{x,y \in \Lambda_N} p(x-y) [ f(\eta^{x,y}) -f(\eta)],\\
&(\mc L_{N,b} f)(\eta) = \frac{\kappa}{N^\theta}\sum_{y\in\{0,N\}} \sum_{x \in \Lambda_N} p(y-x)c_{x}(\eta,r(y)) [f(\eta^x) - f(\eta)]
\end{split}
\end{equation} where  the configurations $\eta^{x,y}$ and $\eta^x$ have been defined in \eqref{tranformations}, the rates $c_{x}(\eta,r(y))$ have been defined in \eqref{rate_c} and $r(0)=\alpha$ and  $r(N)=\beta$.

We consider  the Markov process  speeded up in the  time scale $t\Theta(N)$ and  note that $\{ \eta_{t \Theta(N)} \,:\, t\ge 0\} $ has  infinitesimal generator given by $\Theta(N)\mc L_{N}$. Although $\eta_{t \theta(N)}$ depends on $\alpha$, $\beta$ and $\theta$, we shall omit these index in order to simplify notation. 

As in Section \ref{sec:stat_mea} we can prove that the Bernoulli product measures $\nu_\rho^N$ as defined in \eqref{eq:bernoulli_mea} are reversible when we consider $\alpha=\beta=\rho$.  The proof is quite similar to the one given in Lemma \ref{lem:bern_rev} and for that reason it is omitted.

In the next section we analyse the case where $p(\cdot)$ has finite variance and we denote it by $\sigma^2$, so that  $$\sigma^{2}:=\sum_{z\in \bbbz}z^{2}p(z)<\infty.$$  As an example we consider $p(\cdot)$ as in \eqref{eq:choice_p}, that is $p(0)=0$ and  \begin{equation*}
p(z) = 
\dfrac{c_{\gamma}}{\vert z\vert^{\gamma+1}},
\end{equation*}
 for $z\neq 0$, 
where  $c_{\gamma}$ is a normalizing constant and we take $\gamma> 2$, so that $p(\cdot)$ has finite variance.   For simplicity of the presentation we stick to this choice of $p(\cdot)$ whenever we mention to the case where $p(\cdot)$ has finite variance but we note that many of our results are true in the more general setting where we just assume $p(\cdot)$ to be translation invariant, mean zero and with finite variance. 

\begin{remark} We note that for the choice of  $p$ with $p(1)=\frac 12=p(-1)$ the dynamics described above coincides with the one of the first section. In that sense many of the results that we will derive here are a generalization of those obtained before. 
\end{remark}
 
 In Section \ref{sec:inf_var} we analyse the case where $p(\cdot)$ is as in \eqref{eq:choice_p} but we consider $\gamma\in(1,2)$ so that $p(\cdot)$ has mean zero but with infinite variance. We note that in the case $\gamma=2$ the transition probability $p(\cdot)$ also has mean zero  and  infinite variance, but in this case  the results are similar to those when $p(\cdot) $ has finite variance, see Remark \ref{rem:gamma=2}.

\subsection{The finite variance case}
\label{sec:fin_var}

\subsubsection{-Hydrodynamic equations:} \label{sec:hyd_eq_long_jumps} Recall the notation introduced in Section \ref{sec:hyd_eq_ssep}.  We can now give the definition of the weak solutions of the  hydrodynamic equations that will be derived in this section when $p(\cdot)$ is assumed to have finite variance. In what follows  $g:[0,1]\rightarrow [0,1]$ is a measurable function and it is the initial condition of all the partial differential equations that we define below, that is ${ \rho}_{0}(q)= g(q),$ for all $q\in(0,1)$.

\begin{definition}
\label{Def. Dirichlet source Condition-g}
Let   $\hat \sigma \ge 0$ and $\hat \kappa \geq  0$ be some parameters.  We say that  $\rho:[0,T]\times[0,1] \to [0,1]$ is a weak solution of the reaction-diffusion equation with  Dirichlet boundary conditions
 \begin{equation}\label{eq:Dirichlet source Equation-g}
 \begin{cases}
 &\partial_{t}\rho_{t}(q)=\frac{{\hat \sigma}^2}{2}\Delta\, {\rho} _{t}(q)+ {\hat \kappa} \Big\{ \frac{\alpha-\rho_t(q)}{q^{\gamma+1}}+\frac{\beta-\rho_t(q)}{(1-q)^{\gamma+1}}\Big\}, \quad (t,q) \in (0,T]\times(0,1),\\
 &{ \rho} _{t}(0)=\alpha, \quad { \rho}_{t}(1)=\beta,\quad t \in (0,T],
 \end{cases}
 \end{equation} starting from a measurable function $g:[0,1]\rightarrow [0,1]$,
if the following three conditions hold:
\begin{enumerate}
\item \begin{itemize}

\item $\rho \in L^{2}(0,T;\mathcal{H}^{1})$ if $\hat \sigma >0$,

 \item $\int_0^T \int_0^1 \Big\{ \frac{(\alpha-\rho_t(q))^2}{q^{\gamma+1}}+\frac{(\beta-\rho_t(q))^2}{(1-q)^{\gamma+1}}\Big\} \, dq\, dt <\infty$ if $\hat \kappa >0$, \end{itemize}
\item $\rho$ satisfies the weak formulation:
\begin{equation}\label{eq:Dirichlet_source_ integral-g}
\begin{split}
&F_{RD}:=\int_0^1 \rho_{t}(q)  G_{t}(q) \,dq  -\int_0^1 g(q)   G_{0}(q) \,dq \\
&- \int_0^t\int_0^1 \rho_{s}(q)\Big(\dfrac{\hat \sigma^{2}}{2}\Delta + \partial_s\Big) G_{s}(q)  \,dq \, ds\\
&- {\hat \kappa}\int_0^t\int_0^1G_s(q)\left( \frac{\alpha-\rho_s(q)}{q^{\gamma+1}}+\frac{\beta-\rho_s(q)}{(1-q)^{\gamma+1}}\right)\, dq\,ds=0,
\end{split}   
\end{equation}
for all $t\in [0,T]$ and any function $G \in C_c^{1,2} ([0,T]\times[0,1])$, 

\item if $\hat \sigma >0$ then $\rho _{t}(0)=\alpha, \quad { \rho}_{t}(1)=\beta$
for all $t \in [0,T]$.
\end{enumerate}
\end{definition}

\begin{remark}
Observe that in the case $\hat \sigma >0$ and $\hat \kappa=0$ we recover the heat equation with Dirichlet boundary conditions. If $\hat \sigma =0$ the equation does not have the diffusion term. 
\end{remark}

\begin{definition}
\label{Def. Robin Condition-g}
Let $\hat \sigma>0$ and $\hat m\ge 0$ be some parameters.  We say that  $\rho:[0,T]\times[0,1] \to [0,1]$ is a weak solution of the heat equation with Robin boundary conditions 
 \begin{equation}\label{Robin Equation-g}
 \begin{cases}
 &\partial_{t}\rho_{t}(q)= \frac{\hat \sigma^2}{2}\Delta\, {\rho} _{t}(q), \quad (t,q) \in [0,T]\times(0,1),\\
 &\partial_{q}\rho _{t}(0)=\frac{2\hat m}{\hat \sigma^2}(\rho_{t}(0) -\alpha),\quad \partial_{q} \rho_{t}(1)=\tfrac{2\hat m}{\hat \sigma^2}(\beta -\rho_{t}(1)),\quad t \in [0,T],
 \end{cases}
 \end{equation}
starting from a measurable function $g:[0,1]\rightarrow [0,1]$, if the following two conditions hold: 
\begin{enumerate}
\item $\rho \in L^{2}(0,T;\mathcal{H}^{1})$, 

\item $\rho$ satisfies the weak formulation:
\begin{equation}\label{eq:Robin integral-g}
\begin{split}
F_{Rob}:=\int_0^1 &\rho_{t}(q)  G_{t}(q) \,dq  -\int_0^1 g(q)   G_{0}(q) \,dq \\
& - \int_0^t\int_0^1 \rho_{s}(q) \Big(\dfrac{\hat \sigma^{2}}{2}\Delta + \partial_s\Big) G_{s}(q)  \,dq\, ds \\&+ \dfrac{\hat \sigma^{2}}{2}\int^{t}_{0}   \{\rho_{s}(1) \partial_q G_{s}(1)-\rho_{s}(0)  \partial_q G_{s}(0) \} \, ds\\
&-\hat m \int^{t}_{0} \{ G_{s}(0)(\alpha -\rho_{s}(0)) +  G_{s}(1)(\beta -\rho_{s}(1)) \}\,  ds=0,
\end{split}   
\end{equation}
for all $t\in [0,T]$, any function $G \in C^{1,2} ([0,T]\times[0,1])$. 
\end{enumerate}
\end{definition}

\begin{remark} \label{neumann_cond_rem_lj}
Observe that in the case $\hat m =0$ the equation above is the heat equation with Neumann boundary conditions. 
\end{remark}

\subsubsection{-Hydrodynamic Limit:}
\label{sec:HL_long_jumps} Recall the notion of the empirical measure given in Section \ref{sec:hyd_eq_ssep} and note that in this case we have 
$$\pi^{N}_{t}(\eta, dq):=\pi^{N}(\eta_{t\theta(N)}, dq)$$ and we note that, in this case, the time scale $\theta(N)$ will change with the range of $\theta$, contrarily to what happens in the model of Section \ref{chap1}. As before, let $\bbbp _{\mu _{N}}$ be the probability measure in the Skorohod space $\mathcal D([0,T], \Omega_N)$ induced by the  Markov process $\{\eta_{t\theta(N)}\,:\,t\ge 0\}$ and the initial probability measure $\mu_N$ and we denote by $E _{\bbbp_{\mu _{N}}}$ the expectation with respect to $\bbbp_{\mu _{N}}$. Let $\lbrace\bbbq_{N}\rbrace_{N\geq 1}$ be the  sequence of probability measures on $\mathcal D([0,T],\mathcal{M}^{+})$ induced by the  Markov process $\{\pi_{t}^{N}\,;\,t\geq 0\}$ and by $\bbbp_{\mu_{N}}$.

\begin{remark}
We note that due to the presence of  long jumps in  the system, we cannot obtain information about the empirical profile nor the two point correlation function in a simple way as we did in Section \ref{sec:empirical_prof}. We also note that the matrix ansatz method described in Section \ref{sec:stat_mea} in this case does not give us any information about the stationary measures for this model. This study is left for a future work. 
\end{remark}

Let $g: [0,1]\rightarrow[0,1]$ be a measurable function and let  $\lbrace\mu_{N}\rbrace_{N\geq 1 }$ be a sequence of probability measures in $\Omega_{N}$  associated to $g(\cdot)$, see  \eqref{assoc_mea}. The first result in this section   is stated in the following theorem (see Figure \ref{fig:hlf_lj}).

\begin{theorem}
\label{th:hyd_long_jumps}
 Let $g:[0,1]\rightarrow[0,1]$ be a measurable function and let $\lbrace\mu _{N}\rbrace_{N\geq 1}$ be a sequence of probability measures in $\Omega_{N}$ associated to $g(\cdot)$. Then, for any $0\leq t \leq T$,
\begin{equation*}\label{limHidreform_ljfv}
 \lim _{N\to\infty } \bbbp_{\mu _{N}}\Big( \eta_{\cdot} : \left\vert \dfrac{1}{N-1}\sum_{x \in \Lambda_{N} }G\left(\tfrac{x}{N} \right)\eta_{t\theta(N)}(x) - \int_{0}^1G(q)\rho_{t}(q)dq \right\vert    > \delta \Big)= 0,
\end{equation*}
where  the time scale is given by
 \begin{equation}\label{time_scales}
\Theta(N)= \begin{cases}
 N^2, &\quad\textrm{if} \,\,\,  \theta\geq 1-\gamma,\\
 N^{\gamma+\theta+1}, &\quad   \textrm{if} 
\,\,\, \theta<1-\gamma,\\
 \end{cases}
 \end{equation}
and $\rho_{t}(\cdot)$ is the unique weak solution of : 
\begin{itemize}
\item[$\bullet$] \eqref{eq:Dirichlet source Equation-g} with $\hat \sigma=0$ and $\hat \kappa=\kappa c_\gamma $, if $\theta<1-\gamma$;
\item [$\bullet$]  \eqref{eq:Dirichlet source Equation-g} with $\hat \sigma=\sigma$ and $\hat \kappa=\kappa c_\gamma $, if $\theta =1-\gamma$; 
\item[$\bullet$]  \eqref{eq:Dirichlet source Equation-g} with $\hat \sigma=\sigma$ and $\hat \kappa=0$, if $\theta \in (1-\gamma,1)$;
\item[$\bullet$]  (\ref{Robin Equation-g}) with $\hat \sigma=\sigma$ and $\hat m=\tfrac\kappa 2 $, if $\theta =1$;
\item[$\bullet$]  (\ref{Robin Equation-g}) with $\hat \sigma=\sigma$ and $\hat m=0$, if $\theta>1$.
\end{itemize}
\end{theorem}

\begin{remark} We note that for  a transition probability transition $p(\cdot)$ which is symmetric and with finite variance the last three regimes obtained above are in force (however \eqref{eq:Dirichlet source Equation-g} with $\hat\kappa=0$ is obtained for $\theta\in[0,1)$). We note that the two first regimes depend on the specific choice of the transition probability $p(\cdot)$ that we have assumed in \eqref{eq:choice_p}.  We also note that if we impose that the higher moments of $p(\cdot)$ are finite then the regime \eqref{eq:Dirichlet source Equation-g} with $\hat\kappa=0$ can be reached for  $ \theta\in[v,1)$ where $v<0$ depends on the finiteness of the moments of $p(\cdot)$.
\end{remark}

\begin{remark}\label{rem:gamma=2}
Despite, in the case $\gamma=2$, the transition probability $p(\cdot)$ has infinite variance, we obtain a very similar behavior to the one described above     but the time scale that one has to consider is $N^2/\log(N)$ instead of $N^2$. We leave the adaptation of the proof in this case as an exercise to the reader.
\end{remark}

\begin{remark} We note that the solution of the hydrodynamic equation depends on the parameter $\kappa$ which appears at the boundary dynamics in two different regimes of $\theta$, namely $\theta=1-\gamma$ and $\theta=1$.\end{remark}

\begin{figure}
\begin{center}
\begin{tikzpicture}[scale=0.25]
\shade[right color=black!30, left color=gray!50] (0,5) -- (0,10) --(20,10)--(20,5)-- cycle;
\shade[right color=columbiablue, left color=gray!50] (0,5) -- (20,5) --(20,-20)--(0,0)-- cycle;
\shade[top color=black!30, bottom color=black!80] (0,0) -- (20,-20) --(0,-20)-- cycle;
\draw (0,12) node[right]{$\theta$};
\draw (24,0) node[right]{$\gamma$};
\draw (-1,0) node[left]{$\theta=-1, \gamma=2$};
\draw (-1,5) node[left]{$\theta=1, \gamma=2$};
\draw[->,>=latex] (0,0) -- (23,0);
\draw[->,>=latex] (0,-21) -- (0,11);
\draw[-,=latex,white,ultra thick] (0,0) -- (20, -20) node[midway, above, sloped] {{Reaction-Diffusion eq. \& Dirichlet b.c. }};
\draw[-,=latex,black,ultra thick] (0,5) -- (20, 5) node[midway, sloped, below] {{Heat eq. \& Robin b.c.}};
\node[right, columbiablue] at (4,8) {Heat eq. \& Neumann b.c.} ;
\node[right,gray] at (5,-2) {Heat eq. \&  Dirichlet b.c.} ;
\node[right, columbiablue] at (0.5,-18) {Reaction eq. \& Dirichlet b.c.} ;
\fill[black] (0,0) circle (0.3cm);
\fill[black] (0,5) circle (0.3cm);
\draw[<-,black,ultra thick] (19,-18) -- (21, -15) node[right] {$\theta=1-\gamma$};
\end{tikzpicture}
\end{center}
\caption{The five different hydrodynamic regimes in terms of $\gamma$ and $\theta$.}
\label{fig:hlf_lj}
\end{figure}
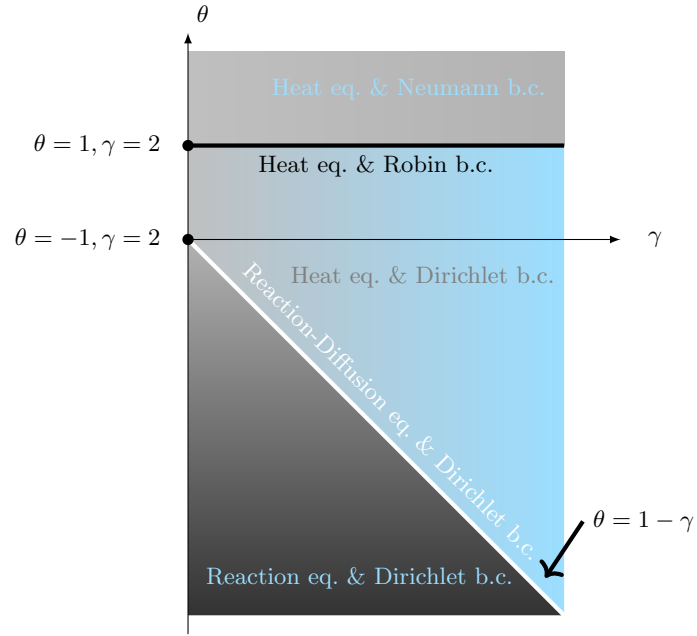

Now note that as before, the stationary solutions of the hydrodynamic limits in the  case $\theta>1-\gamma$  are standard and for that reason they are ommited. On the other hand, the form and properties of the stationary solutions in the case $\theta \le 1-\gamma$  are more complicated to obtain in the case $\theta= 1-\gamma$. This problem is studied in more details in \cite{JO-V} for a slighlty different dynamics. In Figure \ref{fig:stat_sol_lj} we only present some graphs of the stationary solutions and refer the interested reader to \cite{JO-V} for a complete description on the behavior of those solutions. 
Below we draw the graph of these stationary solutions for a choice of $\alpha=0.2$ and $\beta=0.8$.
\begin{figure}
    \centering
    \includegraphics[width=0.7\textwidth]{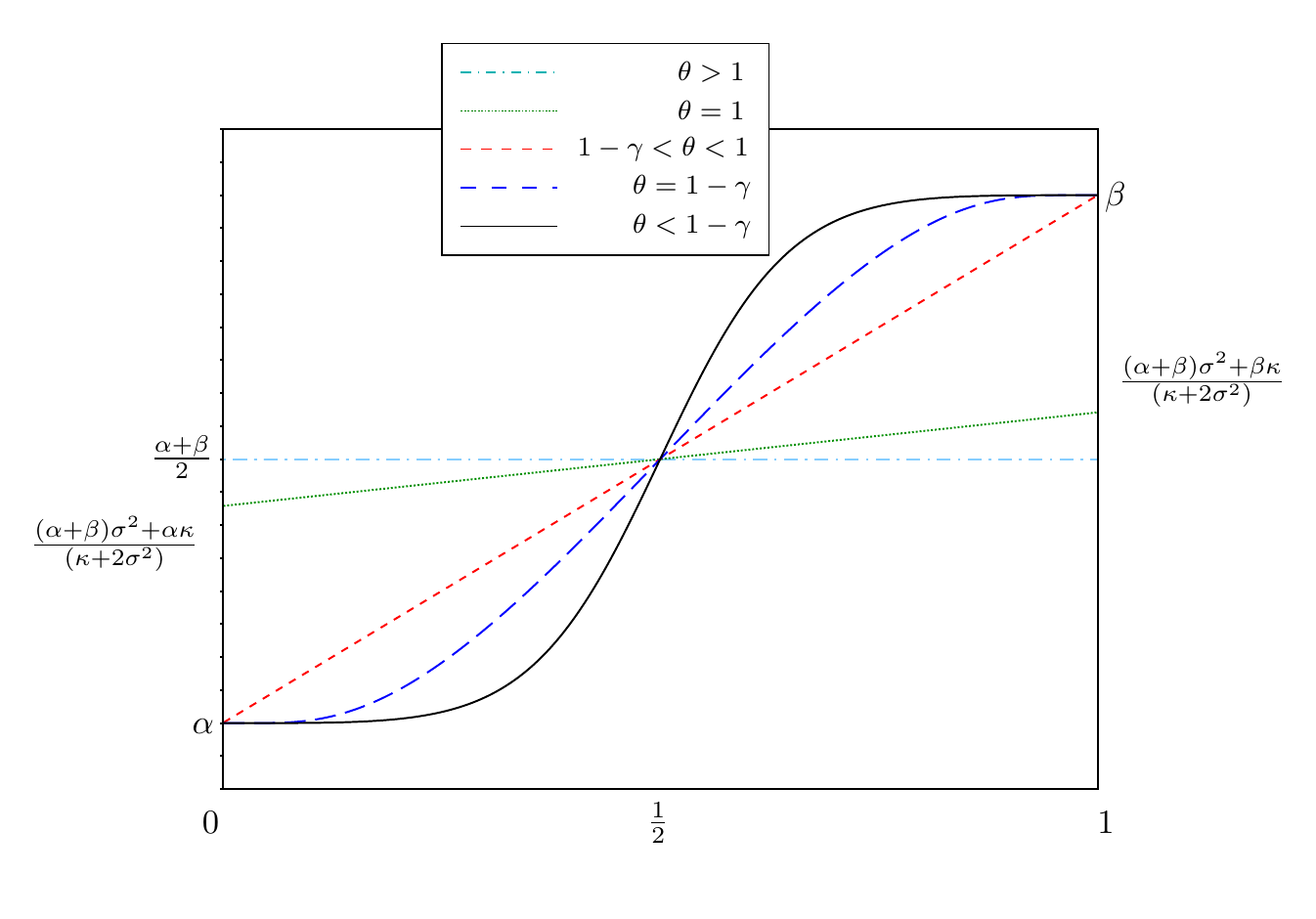}
    \caption{Stationary solutions of the hydrodynamic equations.}
    \label{fig:stat_sol_lj}
\end{figure}

The proof of Theorem \ref{th:hyd_long_jumps} is described in  Section \ref{sec:HL} below Figure \ref{fig:hlf} and for that reason many steps now are omitted. The proof of tightness  of the sequence  $\lbrace\bbbq_{N}\rbrace_{N\geq 1}$ is quite similar to the one given in Section \ref{sec:tightness}. The characterization of limit points is also close to the one given in Section \ref{limit point}, the only difference comes at the level of the identification of the density as a weak solution of the corresponding  partial differential equation. For that purpose, 
the next section is dedicated to the presentation of an heuristic argument to deduce the weak formulation for the  solution of the corresponding hydrodynamic equation. The adaptation of the  rest of the arguments to this new dynamics is left to the reader. 

\subsubsection{-Heuristics for hydrodynamic equations:} \label{sec:heuri_long_jumps} As in Section \ref{sec:heuri_ssep}, the identification of the density $\rho_t(\cdot)$ as a weak solution of the corresponding hydrodynamic equation is obtained by using auxiliary martingales. Fix then a function $G:[0,1]\to\bbbr$ which does not depend on time and which  is two times continuously differentiable. As in Section \ref{sec:heuri_ssep}, we use Dynkin's formula and we note that 
\begin{equation}\label{genaction}
\begin{split}
\int_0^t \Theta(N)& \mc L_N ( \langle \pi^N_{s}, G \rangle ) \, ds =\\& \cfrac{\Theta(N)}{N-1} \int_0^t  \sum_{x\in \Lambda_N}  \tilde{\mc L_N}G(\tfrac{x}{N}) \eta_{s\theta(N)}(x) \, ds \\+& \cfrac{ \kappa \Theta(N)}{(N-1)N^\theta} \int_0^t \sum_{y\in\{0,N\}}\sum_{x \in \Lambda_N}  G(\tfrac{x}{N})p(y-x) (r(y)-\eta_{s\theta(N)}(x)) \, ds,
\end{split}
\end{equation}
where for all $x\in \Lambda_N$
\begin{equation}
\label{eq:mcln}\begin{split}
 (\tilde{\mc L}_N G) (\tfrac{x}{N}) &= \sum_{y \in \Lambda_N} p(y-x) \left[ G(\tfrac{y}{N}) -G(\tfrac{x}{N})\right].
\end{split}
\end{equation}   
Now, we extend the first sum in \eqref{genaction} to all the integers so that we extend the function $G$ to $\bbbr$ in such a way that it remains two times continuously differentiable. By the definition of $\tilde {\mc L_{N}}$, we get that
\begin{equation}\label{bulkaction}
\begin{split}
  \cfrac{\Theta(N)}{N-1}\int_0^t \sum_{x\in \Lambda_N}& \tilde{ \mc L_N}G(\tfrac{x}{N}) \eta_{s\theta(N)}(x) \, ds
 \\
 =&\cfrac{\Theta(N)}{N-1} \int_0^t\sum_{x\in \Lambda_N} (K_N G) (\tfrac{x}{N})\eta_{s\theta(N)}(x)\, ds \\
 -& \cfrac{\Theta(N)}{N-1}  \int_0^t\sum_{x\in \Lambda_N}\sum_{y\leq 0}\left[G(\tfrac{y}{N})- G(\tfrac{x}{N})\right] p(x-y)\eta_{s\theta(N)}(x) \, ds \\
-& \cfrac{\Theta(N)}{N-1} \int_0^t \sum_{x\in \Lambda_N}\sum_{y\geq N}\left[G(\tfrac{y}{N})- G(\tfrac{x}{N})\right] p(x-y)\eta_{s\theta(N)}(x)\, ds,
\end{split}
\end{equation}
where 
\begin{equation}
\label{eq:kappan}
 ({K}_N G) (\tfrac{x}{N}) = \sum_{y \in\bbbz} p(y-x) \left[ G(\tfrac{y}{N}) -G(\tfrac{x}{N})\right].
\end{equation}
Now, we are going to analyse how the different boundary conditions appear on the hydrodynamic equations given in Section \ref{sec:hyd_eq_long_jumps} from this dynamics.

\subsubsection{- The case $\theta <1-\gamma$:} 
\label{sec:CLPDC0} 
Take a function $G: (0,1) \to\bbbr$ two times continuously differentiable  and with compact support  in $(0,1)$, so that we can choose an extension by $0$ outside of the  support of $G$. Since $\Theta(N)=N^{\gamma+\theta+1}$ (see the statement of Theorem \ref{th:hyd_long_jumps}) a simple computation shows that the first term in \eqref{bulkaction}  vanishes for $\theta<1-\gamma$. Indeed, by a Taylor expansion on $G$ and the fact that $p(\cdot)$ is mean zero, we have that
\begin{eqnarray*}
&& N^{\gamma+\theta+1}\sum_{y\in \bbbz}(G(\tfrac{y+x}{N})-G(\tfrac{x}{N}))p(y)
\end{eqnarray*}
is of same order as
$$N^{\gamma+\theta-1}G''(\tfrac{x}{N})\sum_{y\in \bbbz}y^2p(y)$$
and since $\theta<1-\gamma$ last expression vanishes as $N\to\infty. $

Now, the second and third terms in \eqref{bulkaction} vanish as $N\to\infty$, since $\Theta(N)=N^{\gamma+\theta+1}$ and $\theta<1-\gamma$. Note that since $G$ vanishes outside $(0,1)$, those terms can be rewritten as 
\begin{equation}\label{bulkaction_extra}
\begin{split}
& \cfrac{\Theta(N)}{N-1} \int_0^t  \sum_{x\in \Lambda_N} G(\tfrac{x}{N})r_N^-(\tfrac xN)\eta_{sN^{\gamma+\theta+1}}(x) \, ds\\+& \cfrac{\Theta(N)}{N-1} \int_0^t \sum_{x\in \Lambda_N} G(\tfrac{x}{N})r_N^+(\tfrac xN)\eta_{sN^{\gamma+\theta+1}}(x)\, ds,
\end{split}
\end{equation}
where 
\begin{equation}
\label{eq:r_-andr_+}\begin{split}
 r_N^-(\tfrac xN)&=\sum_{y\geq x}p(y),\quad  r_N^+(\tfrac xN)=\sum_{y\leq x-N}p(y).
\end{split}
\end{equation}   
We observe that, for any $a \in (0,1)$, uniformly in $u \in (a,1-a)$, as $N\to \infty$:
\begin{equation}
\label{eq:rs-r+-}
\begin{split}
&N^\gamma r_N^- ( [uN]) \to_{N\to+\infty} c_\gamma \gamma^{-1} u^{-\gamma}:=r^- (u),\\ &N^\gamma r_N^+ ( [uN]) \to_{N\to+\infty} c_\gamma \gamma^{-1} (1-u)^{-\gamma}:=r^+ (u).\end{split}
\end{equation}

Now we note that we can bound from above, for example the  term at the left hand side  in \eqref{bulkaction_extra} by $N^{\theta+1}$ times 
$$\cfrac{1}{N-1} \int_0^t \sum_{x \in \Lambda_N} N^\gamma r_N^- (\tfrac{x}{N})\, |G(\tfrac{x}{N})|$$ 
because  $|\eta_{sN^{\gamma+\theta}}(x)| \le 1$ for all $s>0$. Since $\theta<-1$ and since the previous sum converges to the (finite) integral of $|G| r^-$ on $(0,1)$, by (\ref{eq:rs-r+-}), the previous display vanishes as $N\to\infty$. Now we look at the boundary terms in \eqref{genaction}, which can be written, for the choice of $\Theta(N)=N^{\gamma+\theta+1}$, as:
\begin{equation*}
\frac{\kappa N^{\gamma+1}}{N-1}  \int_0^t\sum_{y\in\{0,N\}}\sum_{x\in\Lambda_N}G \big(\tfrac{x}{N} \big )p(y-x) (r(y)-\eta_{sN^{\gamma+\theta+1}}(x))\, ds 
\end{equation*}
which is equal to 
$$\kappa \, \int_0^t  \langle \alpha- \pi_s^N, G p \rangle+\langle \beta- \pi_s^N, G \tilde{p} \rangle\, ds ,$$ where $\tilde{p}(q)=p(1-q)$,  and
can be replaced, thanks to the fact that $G$ has compact support, by
\begin{equation*}
 \kappa  \int_0^1 G(q) \Big(p (q) (\alpha-\rho_s (q))+\tilde{p}(q)(\beta-\rho_s(q))\Big)  dq 
\end{equation*}
as $N \to \infty$. The last convergence holds because $G$ has compact support included in $(0,1)$ so that $G p$ and $G\tilde p$ are  continuous function. 
From the previous computations we recognize the terms in \eqref{eq:Dirichlet_source_ integral-g} with $\hat{\kappa}=\kappa c_\gamma$ and $\hat \sigma=0$.

\subsubsection{- The case $\theta =1-\gamma$:} 
 \label{sec:CLPDC1}  In this case we also take a function $G:(0,1) \to\bbbr$ two times continuously differentiable and with compact support  in $(0,1)$, so that we can choose an extension by $0$ outside of its support. 
In this case, since $\Theta(N)=N^2$, by Lemma \ref{convergence laplacian},
 which we state below, the first term in \eqref{bulkaction} can be replaced, for $N$ sufficiently big, by 
 \begin{eqnarray*}
\cfrac{1}{N-1} \int_0^t \sum_{x\in \Lambda_N}\tfrac{\sigma^2}{2} \Delta G(\tfrac{x}{N}) \, \eta_{sN^2}(x) \, ds=\int_0^t \langle \pi^N_s, \tfrac {\sigma^2}{2}\Delta G \rangle\, ds.
\end{eqnarray*}
Moreover, a  similar  computation to the one above shows that  the second and third terms in \eqref{bulkaction} vanish as $N\to\infty$ (recall that $\Theta(N)=N^{2}$ and $\gamma>2$). Finally, the second term in  \eqref{genaction} can be rewritten as 
\begin{equation*}\label{gen_action}
\cfrac{ \kappa N^{\gamma+1}}{(N-1)}\int_0^t \sum_{y\in\{0,N\}}\sum_{x \in \Lambda_N}  G(\tfrac{x}{N})p(y-x)\,  (r(y)- \eta_{sN^2}(x)) \, ds
\end{equation*}
and repeating the  analysis we did in the previous case it converges, as $N\to\infty$ to 
\begin{equation*}
 \kappa  \int_0^t  \int_0^1 G(q) \Big(p (q) (\alpha-\rho_s (q))+\tilde{p}(q)(\beta-\rho_s(q))\Big)  dq \, ds.
\end{equation*}
As above, from the previous computations we recognize the terms in  \eqref{eq:Dirichlet_source_ integral-g} with $\hat{\kappa}=\kappa c_\gamma$ and $\hat{\sigma}=\sigma$.

\subsubsection{- The case $\theta \in (1-\gamma,1)$:} 
\label{sec:CLPDC} Take again  a function $G: (0,1) \to\bbbr$ two times continuously differentiable and with compact support  in $(0,1)$ and extend it by $0$ outside  $(0,1)$. As above,  since $\Theta(N)=N^2$, by Lemma \ref{convergence laplacian},
 which we prove below, the first term in \eqref{bulkaction} can be replaced, for $N$ sufficiently big, by 
 \begin{eqnarray*}
\int_0^t \langle \pi^N_s, \tfrac {\sigma^2}{2}\Delta G \rangle\, ds.
\end{eqnarray*}
Now, the second term in  \eqref{gen_action}  equals to
 \begin{equation*}
 \cfrac{\kappa N^{2-\theta}}{N-1}\int_0^t \sum_{y\in\{0,N\}} \sum_{x\in \Lambda_N}G(\tfrac{x}{N}) p(y-x) (r(y)-\eta_{sN^2}(x)) \, ds
\end{equation*}
and vanishes as $N\to\infty$ since $\theta>1-\gamma$. Now, the last two terms in  \eqref{bulkaction}  also vanish because, for example, the second term in \eqref{bulkaction} can be written as 
 \begin{equation*}
\int_0^t \cfrac{N^{2}}{N-1} \sum_{x\in \Lambda_N}G(\tfrac{x}{N}) r_N^-(\tfrac xN) \eta_{sN^2}(x) \, ds
\end{equation*}
which can be bounded from above by a constant times $t N^{2-\gamma}$ times a sum converging to the integral of $| G | r^-$ on $(0,1)$,  and since $\gamma>2$ this term vanishes. 
From this, we see the terms in  \eqref{eq:Dirichlet_source_ integral-g} with $\hat{\kappa}=0$ and $\hat{\sigma}=\sigma$. 

\begin{remark}
We remark here that in the last three cases, similarly to what we have seen in the case $\theta<0$ for the models of Section \ref{chap1} (see Remark \ref{fix_prof_ssep}), there is an extra condition in the definition of the weak solution of \eqref{eq:Dirichlet source Equation-g}. In this notion of solution we need to show that the value of the profile $\rho_t(\cdot)$ is fixed at the boundary. This issue is analysed in  Appendix \ref{fix_profile}. 
\end{remark}

 \subsubsection{- The case $\theta=1$:} 
In this case we consider a function $G:[0,1] \to \bbbr$ which is two times continuously differentiable and we extend it on $\bbbr$ in a two times continuously differentiable function with compact support which strictly contains $[0,1]$. Note that in this case  $G$ can take non-zero values at $0$ and $1$. As above,  since $\Theta(N)=N^2$, by Lemma \ref{convergence laplacian},
 which we state below and which holds for this new space of test functions, the first term in \eqref{bulkaction} can be replaced, for $N$ sufficiently big, by 
 \begin{eqnarray*}
\int_0^t \langle \pi^N_s, \tfrac {\sigma^2}{2}\Delta G \rangle\, ds.
\end{eqnarray*} 
Now we look  at the terms coming from the boundary, namely the  last term in \eqref{genaction}. Then, in the  term for $y=0$ of \eqref{genaction}(resp. for $y=N$) we do at first a Taylor expansion on $G$ and then we  replace $\eta(x)$ by the average $\overrightarrow{\eta}^{\varepsilon N}(1)= \tfrac{1}{\varepsilon N}\sum_{x =1}^{1+\varepsilon N}\eta(x)$ (resp. $\eta(x)$ by $\overleftarrow{\eta}^{\varepsilon N}(N-1)=\tfrac{1}{\varepsilon N}\sum_{x =N-1-\varepsilon N}^{N-1}\eta(x)$), which can be done as a consequence of Lemma \ref{Rep-Neumann} as pointed out in Remark \ref{sec_rep_robin}. Moreover, note that for $y=0$ and $y=N$ it holds that
\begin{equation}\label{mean}
\begin{split}
\sum_{x \in \Lambda_N}   p(y-x)\xrightarrow[N\to+\infty]{} \frac 12.
\end{split}
\end{equation}
Therefore, we can write the last term in \eqref{genaction} as
\begin{equation*}
\frac \kappa 2\int_0^t \{ (\alpha-\overrightarrow{\eta}^{\varepsilon N}_{sN^2}(1)) G(0) + {(\beta-\overleftarrow{\eta}^{\varepsilon N}_{s N^2}(N-1))} G(1) \} \, ds,
\end{equation*}
plus terms that vanish as $N\to+\infty$. Since   
$$\overrightarrow{\eta}^{\varepsilon N} _{sN^2}(1)\sim \rho_s(0)\quad\textrm{and}\quad\overleftarrow{\eta}^{\varepsilon N}_{sN^2}(N-1)\sim\rho_s(1)$$ last term writes as 
\begin{equation}\label{rob_cond1}
\frac \kappa 2  \int_0^t \{ (\alpha-\rho_s (0))G(0) + {(\beta-\rho_s (1))} G(1) \} \, ds.
\end{equation}
Now, we analyse  the two last terms in \eqref{bulkaction}.  Since the function $G$ has been extended into a two times continuously differentiable function on $\bbbr$, by a Taylor expansion on $G$ we can write those terms as 
 \begin{equation}\label{eq:1_0}
   \cfrac{N}{N-1}\int_0^t \sum_{x\in \Lambda_N}G'(\tfrac{x}{N})\Theta^-_x\eta_{sN^2}(x)\, ds-\cfrac{N}{N-1} \int_0^t\sum_{x\in \Lambda_N}G'(\tfrac{x}{N})\Theta^+_x\eta_{s N^2}(x)\, ds
\end{equation}
plus terms that vanish as $N\to+\infty$.
Above for $x\in\Lambda_N$, 
 \begin{eqnarray*}
  \Theta^-_x=\sum_{y\leq 0} (x-y)p(x-y)\quad \textrm{and}\quad \Theta^+_x=\sum_{y\geq N} (y-x)p(x-y).
\end{eqnarray*}
Note that
\begin{equation}
\label{eq:thetaminus}
 \frac{1}{N}\sum_{x\in \Lambda_N}x\Theta^-_x  \xrightarrow[N\to +\infty]{} 0\quad \textrm{and}\quad \frac{1}{N}\sum_{x\in \Lambda_N}x\Theta^+_x  \xrightarrow[N\to +\infty]{} 0.
\end{equation}
Moreover, note that
 \begin{equation}
 \label{eq:sum236}
 \begin{split}
 &\sum_{x\in \Lambda_N}\Theta^-_x= \sum_{x\in \Lambda_N}\sum_{y\geq x}yp(y)\xrightarrow[N\to+\infty]{} \tfrac{\sigma^2}{2},
\\
&\sum_{x\in \Lambda_N}\Theta^+_x= \sum_{x\in \Lambda_N}\sum_{y\geq N-x}yp(y)\xrightarrow[N\to +\infty]{} \tfrac{\sigma^2}{2}.
 \end{split}
 \end{equation}
In order to prove the convergence of  $\,\sum_{x\in \Lambda_N}\Theta^-_x$ (or of $\sum_{x\in \Lambda_N}\Theta^+_x$ in (\ref{eq:sum236})) we use Fubini's theorem  to get that
\begin{equation}
\begin{split}
\sum_{x\in \Lambda_N}\Theta^-_x &= \sum_{y\in\Lambda_{N}}\sum_{x=1}^{y}yp(y)+\sum_{y\geq N}\sum_{x\in \Lambda_{N}}yp(y)\\\nonumber
&= \sum_{y\in\Lambda_{N}}y^{2}p(y)+(N-1)\sum_{y\geq N}yp(y),
\end{split}
\end{equation}
 and since  $\gamma>2$ the result follows. 
By another Taylor expansion on $G$ we can write \eqref{eq:1_0}  as 
 \begin{equation}\label{eq:rob_1}
  \cfrac{N}{N-1} G'(0)\int_0^t\sum_{x\in \Lambda_N}\Theta^-_x\eta_{s N^2}(x)\, ds-\cfrac{N}{N-1}G'(1) \int_s^t \sum_{x\in \Lambda_N}\Theta^+_x\eta_{s N^2}(x)\, ds
\end{equation}
plus terms that vanish as $N\to+\infty$. 
From Lemma \ref{Rep-Neumann} we can replace in the  term on the left (resp. right) hand side of last expression $\eta_{sN^2}(x)$ by $\overrightarrow{\eta}^{\varepsilon N}_{sN^2}(1)$ (resp. $\overleftarrow{\eta}^{\varepsilon N}_{sN^2}(N-1)$). Therefore, \eqref{eq:rob_1} can be replaced, for $N$ sufficiently big and for $\varepsilon$ sufficiently small, by
  \begin{eqnarray*}
  \int_0^t G'(0)\tfrac{\sigma^2}{2}\overrightarrow{\eta}^{\varepsilon N}_{sN^2}(1)-G'(1)\tfrac{\sigma^2}{2}\overleftarrow{\eta}^{\varepsilon N}_{sN^2}(N-1)\, ds.
\end{eqnarray*}
Since $\overrightarrow{\eta}^{\varepsilon N} _{sN^2}(1)\sim \rho_s(0)\quad\textrm{and}\quad\overleftarrow{\eta}^{\varepsilon N}_{sN^2}(N-1)\sim\rho_s(1),$  last term tends to
\begin{eqnarray}
 \label{rob_cond2}
 \int_0^t G'(0)\tfrac{\sigma^2}{2}\rho_s (0)-G'(1)\tfrac{\sigma^2}{2}\rho_s (1)\, ds,
\end{eqnarray} as $N\to\infty$.

Putting together \eqref{rob_cond1} and \eqref{rob_cond2} we see the boundary terms  that appear at the right hand side of  \eqref{eq:Robin integral-g}.
 
\subsubsection{- The case $\theta >1$:} 
 \label{subsec:thetage1}
 In this case we consider an arbitrary function $G:[0,1] \to \bbbr$ which is two times continuously differentiable and we extend it on $\bbbr$ in a two times continuously differentiable function with compact support. Its support strictly contains $[0,1]$ since $G$ can take non-zero values at $0$ and $1$. As in the last case,  since $\Theta(N)=N^2$, by Lemma \ref{convergence laplacian}, the first term in \eqref{bulkaction} can be replaced, for $N$ sufficiently big, by 
 \begin{eqnarray*}
\int_0^t \langle \pi^N_s, \tfrac {\sigma^2}{2}\Delta G \rangle\, ds.
\end{eqnarray*}  
The last term in \eqref{genaction} vanishes, as $N\to\infty$ since, we can bound it   by a constant times
 \begin{equation*}
N^{1-\theta} \sum_{x \in \Lambda_N}   p(x).
\end{equation*}
Since $\gamma>2$ last display vanishes if $\theta>1$, as $N\to+\infty$. 
 Thus, we only need to look at the expression (\ref{bulkaction}). Therefore, in order to see the boundary terms that appear in \eqref{eq:Robin integral-g}, we can use exactly the computations already done in the case $\theta =1$ from  which we obtain (\ref{rob_cond2}).

We finish this section with the statement of the lemma which is used above in order to obtain the diffusion term in the equations above in the cases $\theta\geq 1-\gamma$. Its proof can be seen in \cite{BGJO}.

\begin{lemma}\label{convergence laplacian}
Let $G:\bbbr\to \bbbr$ be a two times continuously differentiable function with compact support. We have
\begin{eqnarray*}
&&\limsup_{N\rightarrow\infty} \sup_{x \in \Lambda_N} \left\vert N^2\sum_{y\in \bbbz}(G(\tfrac{y+x}{N})-G(\tfrac{x}{N}))p(y) - \frac{\sigma^{2}}{2}\Delta G(\tfrac{x}{N})\right\vert= 0.
\end{eqnarray*}
\end{lemma}

\subsection{The infinite variance case}\label{sec:inf_var}

In this section we analyse the case in which $p(\cdot)$ is as in \eqref{eq:choice_p} but now $\gamma\in(1,2)$ so that $p(\cdot)$ has mean zero but infinite variance. We also consider only the case where $\theta=-1$, but we note that in the regime $\theta<-1$ the behavior of the system, when we take the time scale $\Theta(N)=N^{\gamma+\theta+1}$ is the same as when $\theta<1-\gamma$ and when $p(\cdot)$   has finite variance, that is, it is given by the weak solution of \eqref{eq:Dirichlet source Equation-g} with $\hat\sigma=0$ and $\hat \kappa= \kappa c_\gamma$. The other regimes are open and seem to be quite challenging. Recall the infinitesimal generator given in
\eqref{Generator} and \eqref{generators} and since we are restricted to the case $\theta=-1$,  we consider  the Markov process  speeded up in the  time scale $\Theta(N)=N^\gamma$, so that $\{ \eta_{t N^\gamma} \,:\, t\ge 0\} $ has  infinitesimal generator given by $N^\gamma\mc L_{N}$.  As in Section \ref{sec:stat_mea} we can prove that the Bernoulli product measures $\nu_\rho^N$ as defined in \eqref{eq:bernoulli_mea} are reversible when we consider $\alpha=\beta=\rho$.  The proof is quite similar to the one given in Lemma \ref{lem:bern_rev} and for that reason it is omitted.

\subsubsection{-Hydrodynamic equations:} We can now give the definition of the weak solution of the  hydrodynamic equation that will be derived in this section when $p(\cdot)$ is assumed to have infinite variance.

Recall the notations introduced in the beginning of Section \ref{sec:hyd_eq_ssep}.  
We recall the definition of the fractional Laplacian operator of exponent $\gamma/2$  denoted by $(-\Delta)^{\gamma/2}$. It is a non-local operator which  is defined on the set of functions $G:\bbbr \to \bbbr$ such that
\begin{equation}
\label{eq:integ1}
\int_{-\infty}^{\infty} \cfrac{|G(q)|}{(1 +|q|)^{1+\gamma}} dq < \infty
\end{equation}
by
\begin{equation}
(-\Delta)^{\gamma/2} G \, (q) = c_\gamma  \lim_{\varepsilon \to 0} \int_{-\infty}^{\infty} {\bf 1}_{|q-v| \ge \varepsilon} \, \cfrac{G(q) -G(v)}{|q-v|^{1+\gamma}} dv
\end{equation}
provided the limit exists, which is the case, for example, if $G$ is in the Schwartz space. Recall that  $c_{\gamma}$ is fixed in (\ref{eq:choice_p}). Up to a multiplicative constant, $-(-\Delta)^{\gamma/2}$ is the generator of a $\gamma$-L\'evy stable process. 

We define another  operator $L$ whose  action is given on functions $G \in C_c^{\infty} ((0,1))$, by
$$\forall q \in (0,1), \quad ({ L} G)(q) = c_\gamma  \lim_{\varepsilon \to 0} \int_{0}^{1} {\bf 1}_{|q-v| \ge \varepsilon} \, \cfrac{G(v) -G(q)}{|q-v|^{1+\gamma}} dv.$$
The operator ${L}$ is called the \textit{regional fractional Laplacian} on $(0,1)$. The semi inner-product $\langle \cdot, \cdot \rangle_{\gamma/2}$ is defined on the set $C_c^{\infty}((0,1))$ by 
\begin{equation}
\langle G, H \rangle_{\gamma/2} =  \cfrac{c_{\gamma}}{2} \iint_{[0,1]^2} \cfrac{(H(q) -H(v)) (G(q) -G(v))}{|q-v|^{1+\gamma}} \, dq dv.
\end{equation}  
The corresponding semi-norm is denoted by $\| \cdot \|_{\gamma/2}$. Observe that for any $G,H \in C_c^{\infty} ((0,1))$ we have that
\begin{equation*}
-\int_0^1 G(q) { L} H(q)\, dq = -\int_0^1 {L} G(q) H(q) \, dq = \langle G, H \rangle_{\gamma/2}
\end{equation*}
and note that for all $q \in (0,1), $
\begin{equation}\label{Operator_LL}
\quad ({L} G)(q) = -(-\Delta)^{\gamma/2}  G\, (q) +  V_1(q) G(q) 
\end{equation}
where  $V_1(q)=r^-(q)+r^+(q)$, see \eqref{eq:rs-r+-}, that is, $V_1(\cdot)$ is given  on $q\in(0,1)$ by:
\begin{equation}\label{eq: V_1}
V_1(q)=c_\gamma \gamma^{-1}\Big(\frac{1}{q^{\gamma}}+\frac{1}{(1-q)^{\gamma}}\Big).
\end{equation}

\begin{definition}
\label{Def. Sobolev space_gamma}
The Sobolev space $\mathcal{H}^{\gamma/2}$ consists of all square integrable functions $g: (0,1) \rightarrow \bbbr$ such that $\| g \|_{\gamma/2} <\infty$. This is a Hilbert space for the norm $\| \cdot\|_{{\mc H}^{\gamma/2}}$ defined by
$$\Vert g \Vert_{\mathcal{H}^{\gamma/2}}^2:= \Vert g \Vert^{2} + \Vert g \Vert_{\gamma/2}^{2}.$$
Its elements elements coincide a.e. with continuous functions.

The space $L^{2}(0,T;\mathcal{H}^{\gamma/2})$ is the set of measurable functions $f:[0,T]\rightarrow  \mathcal{H}^{\gamma/2}$ such that 
$$\int^{T}_{0} \Vert f_{t} \Vert^{2}_{\mathcal{H}^{\gamma/2}}dt< \infty. $$
\end{definition}
We now extend the definition of the regional fractional Laplacian on $(0,1)$ to the space $\mc H ^{\gamma/2}$.
\begin{definition}\label{def:Dist}
For $\rho \in \mc H^{\gamma/2}$ we define the distribution $L \rho$ by $$\int_0^1 L\rho(q)G(q)\,dq = \int_0^1 \rho(q)L G(q)\, dq,\quad  G\in C_{c}^{\infty}((0,1)). $$
\end{definition}

Let  ${ L}_\kappa$ be the regional fractional Laplacian on $[0,1]$ with zero Dirichlet boundary conditions, indexed by $\kappa$, and taking the form
\begin{equation}\label{def:LLKappa}
{ L}_\kappa= L -\kappa\tilde  {V}_1,
\end{equation}
where for $q\in(0,1)$, 
\begin{equation}\label{eq: V_1_tilde}
\tilde V_1(q)=p(q)+\tilde p(q)=c_\gamma \Big(\frac{1}{q^{\gamma+1}}+\frac{1}{(1-q)^{\gamma+1}}\Big).
\end{equation}
Above $\tilde p(q)=p(1-q)$.
Below  $g:[0,1]\rightarrow [0,1]$ is a measurable function and it is the initial condition of the partial differential equation that we obtain in this section.
\begin{definition}
\label{Def. Dirichlet Condition}
 Let $\kappa>0$ be some parameter. We say that  $\rho^{\kappa}:[0,T]\times[0,1] \to [0,1]$ is a weak solution of the regional fractional reaction-diffusion equation with Dirichlet boundary conditions given by  
\begin{equation}\label{eq:Dirichlet Equation}
\begin{cases}
&\partial_{t} \rho_{t}^{\kappa}(q)= L_{\kappa} \rho_t^{\kappa}(q)+ \kappa \tilde{V}_0(q),  \quad (t,q) \in [0,T]\times(0,1),\\
 &{ \rho^{\kappa}_{t}}(0)=\alpha, \quad { \rho^{\kappa}_{t}}(1)=\beta,\quad t \in [0,T],
 \end{cases}
 \end{equation}
where  $$\tilde {V_0}(q)=\alpha p(q)+\beta \tilde p(q) =c_\gamma\Big(\frac{\alpha}{q^{1+\gamma}}+\frac{\beta}{(1-q)^{1+\gamma}}\Big),$$ and starting from a measurable function  $g:[0,1]\rightarrow [0,1]$, 
 if: 
\begin{enumerate}

\item  $\rho^{\kappa} \in L^{2}(0,T;\mathcal{H}^{\gamma/2})$.
\item   $\int_0^T \int_0^1 \Big\{ \frac{(\alpha-\rho_t^{\kappa}(q))^2}{q^{1+\gamma}}+\frac{(\beta-\rho_t^{\kappa}(q))^2}{(1-q)^{1+\gamma}}\Big\} \, dq\, dt <\infty$ . 
\item For all $t\in [0,T]$ and {all functions} $G \in C_c^{1,\infty} ([0,T]\times (0,1))$ we have that 
\begin{equation}
\label{eq:Dirichlet integral}
\begin{split}
F_{Dir}^{\kappa}&:=\int_0^1 \rho_{t}^{\kappa}(q)  G_{t}(q)\, dq  -\int_0^1 g(q)   G_{0}(q)\, dq\\&- \int_0^t\int_0^1 \rho_{s}^{\kappa}(q)\Big(\partial_s +  L_{\kappa} \Big) G_{s}(q)\, dqds \\
&-\kappa \int^{t}_{0}\int_0^1 G_s(q) \tilde {V}_0 (q)\, dq\,ds=0.
\end{split}   
\end{equation}
\end{enumerate}
\end{definition}

\begin{remark}\label{rem:uniq_weak_sol_iv}
We observe that the partial differential equation above has a unique weak solution in the sense defined above. We do not include the proof of this result in these notes but we refer the interested reader to \cite{BGJO} for the proof of the uniqueness for a  very similar equation. The same proof gives uniqueness in this case.
\end{remark}
\subsubsection{-Hydrodynamic Limit:}\label{subsec:HL_inf_var} Recall the notion of the empirical measure given in Section \ref{sec:hyd_eq_ssep} and note that in this case we have 
$$\pi^{N}_{t}(\eta, dq):=\pi^{N}(\eta_{tN^\gamma}, dq)$$ since the time scale now is equal to $\theta(N)=N^\gamma$. 

The second result of this section is stated in the following theorem.

\begin{theorem}\label{theo:hydro_limit}
 Let $g:[0,1]\rightarrow[0,1]$ be a measurable function and let $\lbrace\mu _{N}\rbrace_{N\geq 1}$ be a sequence of probability measures in $\Omega_{N}$ associated to $g(\cdot)$. Then, for any $0\leq t \leq T$,
\begin{equation*}\label{limHidreform_ljiv}
 \lim _{N\to\infty } \bbbp_{\mu _{N}}\left( \eta_{\cdot} : \left\vert \dfrac{1}{N-1}\sum_{x \in \Lambda_{N} }G\left(\tfrac{x}{N} \right)\eta_{tN^{\gamma}}(x) - \int_{0}^1G(q)\rho_{t}^{\kappa}(q)dq \right\vert    > \delta \right)= 0,
\end{equation*}
where  $\rho_{t}^{\kappa}(\cdot)$ is the unique weak solution of (\ref{eq:Dirichlet Equation}) in the sense of Definition \ref{Def. Dirichlet Condition}.
\end{theorem}

\subsubsection{-Heuristics for hydrodynamic equations:} \label{sec:CL} 
Fix $G:[0,1]\to\bbbr$ which does not depend on time and has compact support included in $(0,1)$.  
Recall \eqref{genaction} and \eqref{bulkaction} and recall  that we assumed  $\theta=-1$, so that \eqref{gen_action} now writes as 
\begin{equation}\label{mart_dec_inf_var}
\begin{split}
\int_0^tN^\gamma \mc L_N ( \langle& \pi^N_{s}, G \rangle ) \, ds = \cfrac{N^\gamma}{N-1}  \int_0^t \sum_{x\in \Lambda_N} (\tilde{\mc L}_N G) (\tfrac{x}{N})\eta_{sN^\gamma}(x)  \\+& \cfrac{ \kappa N^{\gamma+1}}{(N-1)} \int_0^t \sum_{y\in\{0,N\}}\sum_{x \in \Lambda_N}  G(\tfrac{x}{N})p(y-x) (r(y)-\eta_{sN^\gamma}(x)) \, ds.
\end{split}
\end{equation}
Note that the first term   on the right hand side in last display is equal to  $$\int_0^t\langle \pi_{s}^{N}, \tilde{\mc L}_N G \rangle \, ds.$$ Since  from Lemma 3.3 in  \cite{BJ}, we can deduce  that 
\begin{equation}\label{F_convergenceLL }
\lim _{N\to \infty}N^{\gamma}(\tilde{\mc L}_N G)(q)  = (L G)(q)
\end{equation}
uniformly in  $[a,1-a]$, for all functions $G$ with compact support included in $[a,1-a]$. Therefore, the first term   on the right hand side of \eqref{mart_dec_inf_var} can be replaced by 
\begin{equation}\label{eq:first_term}
\int_0^t \int_{0}^{1}(L G)(q) \rho_{s}^{\kappa}(q)\, dq\, ds,
\end{equation}
for $N$ sufficiently big. Now, the second term on the right hand side in \eqref{mart_dec_inf_var} is equal to $$\kappa\int_0^t \langle \alpha -\pi_{s}^{N},  Gp \rangle \, ds+\kappa \int_0^t\langle \beta -\pi_{s}^{N},  G \tilde{p} \rangle \, ds$$ and converges, as $N\to\infty$, to 
\begin{equation}\label{eq:second_term}
\begin{split}
 & \kappa\int_0^t\int_{0}^{1}(\alpha -\rho_{t}^{\kappa}(q))G(q)p(q)du +\kappa\int_0^t\int_{0}^{1}(\beta-\rho_{t}^{\kappa}(q))G(q)\tilde{p}(q)dq\\
&=-  \kappa\int_0^t \int_{0}^{1}\rho_{t}^{\kappa}(q)G(q)\tilde{V}_{1}(q)dq+\kappa\int_0^t\int_{0}^{1}G(q)\tilde {V}_{0}(q)dq.\\
\end{split}
\end{equation}
Putting together \eqref{eq:first_term} and \eqref{eq:second_term} and using \eqref{def:LLKappa} we recognize the corresponding  terms in \eqref{eq:Dirichlet integral}.

We finish this section  by noting that in \cite{BGJO2}  it was studied a similar dynamics to the one described above. There we considered the same bulk dynamics with long jumps given by $p(\cdot)$ with the choice \eqref{eq:choice_p} and $\gamma\in(1,2)$ but the boundary dynamics was different. In \cite{BGJO2} instead of considering just one boundary at each end point of the bulk, it was added  infinitely many reservoirs at the left and at the right of the bullk. As in the dynamics described above, particles can be injected and removed from the system at any point of the bulk by any of the reservoirs located at $y\leq 0$ or $y\geq N$. We note that in the case of this new dynamics the results obtained in \cite{BGJO2} are similar to those presented here, except that the transitions occur for a different value of $\theta$ and for that reason, the potential $\tilde{V}_1$ that appears in the definition of $L_\kappa$ in  the reaction-diffusion equation \eqref{eq:Dirichlet Equation} has a different power than the one that appears in the hydrodynamic  equation in \cite{BGJO2}. 
It would be very interesting to analyse other types of boundary dynamics superposed to the bulk dynamics that we defined above in order to see if we can come up with new fractional reaction-diffusion equations with more tricky boundary conditions than the Dirichlet boundary conditions that we obtained here. And  it would be very interesting to look at  the case where $\theta>-1$, the slow boundary regime, when $p(\cdot)$ is given as above with $\gamma\in(1,2)$, see the area coloured in rose in the figure below.  This is a subject to pursue in the near future. 
In the figure below we summarize the scenario of the hydrodynamic limit  for the models of this section.

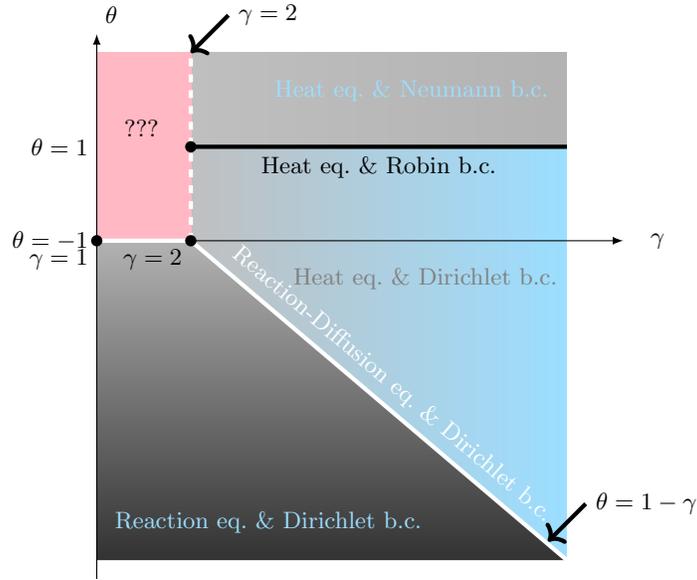
\begin{figure}[!htb]
\begin{center}
\begin{tikzpicture}[scale=0.25]
\shade[right color=black!30, left color=gray!50] (0,5) -- (0,10) --(20,10)--(20,5)-- cycle;
\shade[right color=columbiablue, left color=gray!50] (0,5) -- (20,5) --(20,-17)--(0,0)-- cycle;
\shade[top color=black!30, bottom color=black!80] (-5,0) -- (0,0) --(20,-17)--(-5,-17)--cycle;
\shade[top color=cherryblossompink, bottom color=cherryblossompink] (-5,0) -- (0,0) --(0,10)--(-5,10)--cycle;
\draw (-4,6) node[right]{???};
\draw (-5,12) node[right]{$\theta$};
\draw (24,0) node[right]{$\gamma$};
\draw (-5,0) node[left]{$\theta=-1$};
\draw (-5,5) node[left]{$\theta=1$};
\draw (-5,-1) node[left]{$\gamma=1$};
\draw (0,-1) node[left]{$\gamma=2$};
\draw[->,>=latex] (-5,0) -- (23,0);
\draw[->,>=latex] (-5,-18) -- (-5,11);
\draw[-,=latex,white,ultra thick, dashed] (0,0) -- (0, 10);
\draw[<-,black,ultra thick] (0,10) -- (2, 12) node[right] {$\gamma=2$};
\draw[-,=latex,white,ultra thick] (0,0) -- (20, -17) node[midway, above, sloped] {{Reaction-Diffusion eq. \& Dirichlet b.c. }};
\draw[-,=latex,white,ultra thick] (-5,0) -- (0, 0);
\draw[-,=latex,black,ultra thick] (0,5) -- (20, 5) node[midway, sloped, below] {{Heat eq. \& Robin b.c.}};
\node[right, columbiablue] at (4,8) {Heat eq. \& Neumann b.c.} ;
\node[right,gray] at (5,-2) {Heat eq. \&  Dirichlet b.c.} ;
\node[right, columbiablue] at (-4.5,-15) {Reaction eq. \& Dirichlet b.c.} ;
\fill[black] (-5,0) circle (0.3cm);
\fill[black] (0,0) circle (0.3cm);
\fill[black] (0,5) circle (0.3cm);
\draw[<-,black,ultra thick] (19,-16) -- (21, -14) node[right] {$\theta=1-\gamma$};
\end{tikzpicture}
\end{center}
\caption{Hydrodynamical behavior of the symmetric exclusion with long jumps.}
\label{fig:hlf_lj}
\end{figure}

\appendix
\label{sec:RL}

\section{Auxiliary results}
In this section we establish some technical results that are needed in order to prove the  hydrodynamic limit for the models discussed in the previous sections.

\subsection{Entropy bound}
 From now on, we suppose that $\alpha \leq \beta$. 
Let  $\rho : [0, 1] \rightarrow [0, 1]$ be a   function such that  $\alpha \leq \rho(q)\leq\beta$, for all $q\in [0,1]$. Let $\nu_{\rho(\cdot)}^{N}$ be the Bernoulli product measure on 
$\Omega_{N}$ with marginals given by
\begin{equation}\label{eq:ber_mea_prof}
\nu_{\rho(\cdot)}^{N}\lbrace \eta: \eta_{x} = 1 \rbrace = \rho\left( \tfrac{x}{N}\right).\end{equation}

Given two functions $f,g:\Omega_{N} \to \bbbr$ and a probability measure $\mu$ on $\Omega_{N}$, we denote here by $\langle f, g\rangle_{\mu }$ the scalar product between $f$ and $g$ in $L^2 (\Omega_N, \mu)$, that is,
$$ \langle f, g\rangle_{\mu } = \int _{\Omega_{N}} f(\eta)g(\eta) \, d\mu .$$
Let $H_{N}(\mu\vert \nu_{\rho(\cdot)}^{N})$ be the relative entropy of a probability measure $\mu$ on 
$\Omega_{N}$ with respect to the probability measure $\nu_{\rho(\cdot)}^{N} $ on $\Omega_{N}$. We claim that there exists a constant $C_0 :=C(\alpha,\beta)$, such that
\begin{equation}\label{H}
H_{N}(\mu\vert \nu_{\rho(\cdot)}^{N})\leq C_0N.
\end{equation}
For that purpose note that, since  $\nu^N_{\rho(\cdot)}$ is product we have that
$$\nu^N_{\rho(\cdot)}(\eta)=\prod_{x=1}^{N-1}\rho(\tfrac{x}{N})^{\eta(x)}(1-\rho(\tfrac xN))^{1-\eta(x)}\geq({\alpha \wedge (1-\beta)})^N$$
from where we obtain that 
\begin{eqnarray*}
H(\mu|\nu_{\rho(\cdot)}^{N}) &=& \sum_{\eta\in \Omega_{N}}\mu(\eta)\log\left(\dfrac{\mu(\eta)}{\nu_{\rho(\cdot)}^{N}(\eta)}\right)
\leq  \sum_{\eta\in \Omega_{N}}\mu(\eta)\log\left(\dfrac{1}{\nu_{\rho(\cdot)}^{N}(\eta)}\right)\\
&\leq & \log\left(\left[\dfrac{1}{\alpha \wedge (1-\beta)}\right]^{N} \right)\sum_{\eta\in \Omega_{N}}\mu(\eta)
\leq  N\log\left(\dfrac{1}{\alpha \wedge (1-\beta)}\right)\leq C_0N.
\end{eqnarray*}

We remark here that below when we use as reference measure the Bernoulli product measure given in \eqref{eq:ber_mea_prof} we have to  restrict to $\alpha\neq 0$ and $\beta\neq 1$ since in last estimate  the constant $C_0=-\log(\alpha \wedge(1-\beta))$. We also note that when we use the Bernoulli product measure with a constant parameter we do not need to impose that restriction.

\subsection{Estimates on Dirichlet forms}
In this section we consider the model described in Section \ref{chap2} since the results for the model of Section \ref{chap1} can be obtained easily from the ones we derive below. In any case we present some remarks along the text about the corresponding results  for the model of Section \ref{chap1}.

For a probability measure $\mu$ on $\Omega_N$, $x,y \in \Lambda_N$ and a density function $f:\Omega_N \to [0,\infty)$ with respect to $\mu$ we introduce 
\begin{eqnarray*}
I_{x,y}(\sqrt f,\mu)&:=& \int_{\Omega_N} \left(\sqrt {f(\eta^{x,y})}-\sqrt {f(\eta)}\right)^{2} d\mu,\\
I_{x}^{r(y)}(\sqrt f,\mu)&:=& \int_{\Omega_N}  c_{x}(\eta;r(y))\left(\sqrt {f(\eta^x)}-\sqrt {f(\eta)}\right)^{2} d\mu.
\end{eqnarray*}
In last identity $y\in\{0,N\}$ and $r(0)=\alpha$ and $r(N)=\beta$.
We define
\begin{eqnarray*}
\mc D_{N}(\sqrt{f},\mu )&:=& (\mc D_{N,0}+\mc D_{N, b})(\sqrt{f},\mu) \end{eqnarray*}
where 
\begin{eqnarray} 
\label{left_rig_form}
\mc D_{N,0}(\sqrt{f},\mu):=\cfrac{1}{2}\sum_{x,y\in\Lambda_N}p(y-x)\, I_{x,y}(\sqrt{f},\mu), 
\end{eqnarray}
\begin{eqnarray}
\label{left_dir_form}
\mc D_{N,b}(\sqrt{f},\mu):=\frac{\kappa}{N^\theta}\sum_{y\in\{
0,N\}}\sum_{x\in\Lambda_N}p(y-x)\, I^{r(y)}_{x}(\sqrt{f},\mu).
\end{eqnarray} 
 Note that for the models of Section \ref{chap1} the expressions above simplify to 
 \begin{eqnarray} 
\label{left_rig_form_ssep}
\mc D^{NN}_{N,0}(\sqrt{f},\mu):=\sum_{x\in\Lambda_N}\, I_{x,x+1}(\sqrt{f},\mu), 
\end{eqnarray}
\begin{equation}
\label{left_dir_form_ssep}
\begin{split}
\mc D^{NN}_{N,b}(\sqrt{f},\mu):=&\frac{\kappa}{N^\theta}\Big( I^{\alpha}_{1}(\sqrt{f},\mu)+I^{\beta}_{N-1}(\sqrt{f},\mu)\Big).
\end{split}
\end{equation} 
  
Our first goal is to express, for the measure $\mu=\nu_{\rho(\cdot)}^{N}$,  a relation between the Dirichlet form defined by $-\langle \mc L_N\sqrt{f},\sqrt{f} \rangle_{\nu_{\rho(\cdot)}^{N}}$ and $
\mc D_{N}(\sqrt{f},\nu_{\rho(\cdot)}^{N} )$. 
We claim that for any positive constant $B$, there exists a constant $C>0$ such that
\begin{equation}\label{dir_est}
\begin{split}
\frac{1}{BN}\langle \mc L_{N}\sqrt{f},\sqrt{f} \rangle_{\nu_{\rho(\cdot)}^N} &\leq -\dfrac{1}{4BN}\mc D_{N}(\sqrt{f},\nu_{\rho(\cdot)}^N) \\&+ \frac{C}{BN}\sum_{x,y\in\Lambda_N}p(y-x)\Big(\rho(\tfrac xN)-\rho(\tfrac yN)\Big)^2
\\ &+ \frac{C\kappa}{BN^{1+\theta}}\sum_{y \in {\{0,N\}}} \sum_{x\in\Lambda_N} \Big(\rho(\tfrac xN)-r(y)\Big)^2p(y-x).
\end{split}
\end{equation}
Our aim is then to choose $\rho(\cdot)$ in order to minimize the error term, i.e. the two last terms at  the right hand side of the previous inequality. 

\begin{remark}\label{prices_for_DF}

\quad 

\begin{enumerate}
\item If $p(\cdot)$ has finite variance $\sigma^2$, then:

\begin{itemize}
\item for $\rho(\cdot)$ Lipschitz and such that $\rho(0)=\alpha$ and $\rho(1) =\beta$, we get
\begin{equation}
\label{dir_est_lip}
\begin{split}
\frac{1}{BN}\langle\mc  L_{N}\sqrt{f},\sqrt{f} \rangle_{\nu_{\rho(\cdot)}^N} &\leq -\dfrac{1}{4BN}\mc D_{N}(\sqrt{f},\nu_{\rho(\cdot)}^N) + \frac{C}{BN^2}\sigma^2\\& + \frac{C\kappa}{BN^{3+\theta}}\sum_{y \in {\{0,N\}}}\sum_{x\in\Lambda_N}\big( y-x\big)^2p(y-x)\\
&\leq -\dfrac{1}{4BN}\mc D_{N}(\sqrt{f},\nu_{\rho(\cdot)}^N) + \frac{C}{BN^2}\sigma^2 + \frac{C\kappa}{BN^{3+\theta}}.
\end{split}
\end{equation}
\item for $\rho(\cdot)$  such that $\rho(0)=\alpha$, $\rho(1) =\beta$, H\"older of parameter $\frac \gamma 2$ at the boundaries and Lipschitz inside, we get 
\begin{equation}
\label{dir_est_holder}
\begin{split}
\frac{1}{BN}\langle \mc L_{N}\sqrt{f},\sqrt{f} \rangle_{\nu_{\rho(\cdot)}^N} \leq& -\dfrac{1}{4BN}\mc D_{N}(\sqrt{f},\nu_{\rho(\cdot)}^N) + \frac{C}{BN^2} \sigma^2 + \frac{C\kappa \log(N)}{BN^{\gamma+\theta+1}}.
\end{split}
\end{equation}
\item for $\rho(\cdot)$ such that $\rho(0)=\alpha$, $\rho(1) =\beta$, H\"older of parameter $\frac{1+\gamma}{2}$ at the boundaries and Lipschitz inside, we get 
\begin{equation}
\label{dir_est_holder_alt}
\begin{split}
\frac{1}{BN}\langle \mc L_{N}\sqrt{f},\sqrt{f} \rangle_{\nu_{\rho(\cdot)}^N} \leq& -\dfrac{1}{4BN}\mc D_{N}(\sqrt{f},\nu_{\rho(\cdot)}^N) + \frac{C}{BN^2} \sigma^2 + \frac{C\kappa}{BN^{\gamma+\theta+1}}.
\end{split}
\end{equation}
\item for $\rho(\cdot)$  constant, equal to $\alpha$ or to $\beta$,  we have 
\begin{equation}\label{dir_est_const}
\begin{split}
\frac{1}{BN}\langle \mc L_{N}\sqrt{f},\sqrt{f} \rangle_{\nu_{\alpha}^N} &\leq  -\dfrac{1}{4BN}\mc D_{N}(\sqrt{f},\nu_{\alpha}) + \frac{C\kappa 
}{BN^{\theta+1}}.
\end{split}
\end{equation}
\end{itemize}
\item If $p(\cdot)$  is such that $p(1)=p(-1)=\frac 12$, then:  
\begin{itemize}
\item  for  $\rho(\cdot)$ Lipschitz and such that $\rho(0)=\alpha$, $\rho(1) =\beta$ and  locally constant at $0$ and $1$, we get
\begin{equation}
\label{dir_est_lip_ssep}
\begin{split}
\frac{1}{BN}\langle\mc  L_{N}\sqrt{f},\sqrt{f} \rangle_{\nu_{\rho(\cdot)}^N} &\leq -\dfrac{1}{4BN}\mc D^{NN}_{N}(\sqrt{f},\nu_{\rho(\cdot)}^N) + \frac{C}{BN^2}.
\end{split}
\end{equation}
Note that the choice of asking $\rho(\cdot)$ to be locally constant at $0$ and $1$ turns the errors coming from the boundary dynamics to vanish. 
\item for  $\rho(\cdot)$  constant, equal to $\alpha$ or to $\beta$, then we have 
exactly the same error as in \eqref{dir_est_const}.
\end{itemize}
\item If $p(\cdot)$ has infinite variance, then:
\begin{itemize}
\item for  $\rho(\cdot)$ Lipschitz and such that $\rho(0)=\alpha$ and $\rho(1) =\beta$, we get
\begin{equation}
\label{dir_est_lip_iv}
\begin{split}
\frac{1}{BN}\langle\mc  L_{N}\sqrt{f},\sqrt{f} \rangle_{\nu_{\rho(\cdot)}^N} &\leq -\dfrac{1}{4BN}\mc D_{N}(\sqrt{f},\nu_{\rho(\cdot)}^N) \\&+ \frac{C}{BN^3}\sum_{x,y\in\Lambda_N}\frac{1}{|x-y|^{\gamma-1}}\\& + \frac{C\kappa}{BN^{3+\theta}}\sum_{y \in {\{0,N\}}}\sum_{x\in\Lambda_N}\big( y-x\big)^2p(y-x)\\
&\leq -\dfrac{1}{4BN}\mc D_{N}(\sqrt{f},\nu_{\rho(\cdot)}^N) + \frac{C}{BN^\gamma}\sigma^2 + \frac{C\kappa}{BN^{\gamma+\theta+1}}.
\end{split}
\end{equation}
\end{itemize}
\end{enumerate}
\end{remark}
In order to prove \eqref{dir_est} we need some intermediate results. For that purpose we recall from \cite{BGJO} the following two lemmas.

\begin{lemma}
\label{lemmaleft0}
Let $T : \eta \in  \Omega_N \to T(\eta) \in \Omega_N$ be a transformation in the configuration space and $c: \eta\in\Omega_N \to c(\eta)$ be a positive local function. Let $f$ be a density with respect to a probability  measure $\mu$ on $\Omega_N$. Then, we have that
\begin{eqnarray}
\label{use_comp}
\begin{split}
&\left\langle c (\eta) [ \sqrt{f(T (\eta))} -\sqrt{f(\eta)}]\; ,\;\sqrt{f (\eta)} \right\rangle_{\mu}  \\
&\le  -\dfrac{1}{4}\int c (\eta)\left(\left[ \sqrt{f(T (\eta))}\right]-\left[ \sqrt{f(\eta)}\right]\right)^{2}  d\mu\\
& + \dfrac{1}{16}\int \dfrac{1}{c (\eta)}\left[c (\eta)-c (T(\eta)) \dfrac{\mu (T(\eta))}{\mu(\eta)} \right]^{2}\left(\left[ \sqrt{f(T(\eta))}\right]+\left[ \sqrt{f(\eta)}\right]\right)^{2}  d\mu.
\end{split}
\end{eqnarray}
\end{lemma}

\begin{lemma}
\label{lem:densaf}
There exists a constant $C:=C(\rho)$ such that for any $N \ge 1$ and density $f$ be a density with respect to $\nu_{\rho(\cdot)}^N$
\begin{equation*}
\sup_{x\ne y \in \Lambda_N} \int_{\Omega_N} f ( \eta^{x,y}) \; d\nu_{\rho(\cdot)}^N \; \le \; C, \quad \quad  \sup_{x \in \Lambda_N} \int_{\Omega_N} f ( \eta^x) \; d\nu_{\rho(\cdot)}^N \; \le \; C.
\end{equation*} 
\end{lemma} 
 A simple consequence of the previous lemmas is the next two corollaries.  
Recall the bulk generator $\mc L_{N,0}$ given in \eqref{generators}.
\begin{corollary}
\label{lemmabulk}
There exists a constant $C>0$ (independent of $f(\cdot)$ and $N$) such that
\begin{eqnarray*}
\left\langle \mc L_{N,0}\sqrt{f},\sqrt{f}\right \rangle_{\nu_{\rho(\cdot)}^N}  &\le&  -\dfrac{1}{4}\mc D_{N,0}(\sqrt{f},\nu_{\rho(\cdot)}^N) + C\sum_{x,y\in\Lambda_N}p(y-x)\Big(\rho(\tfrac xN)-\rho(\tfrac yN)\Big)^2
\end{eqnarray*}
for any density $f(\cdot)$ with respect to $\nu_{\rho(\cdot)}^N$.
\end{corollary}
Now we look at the generator of the boundary dynamics given in \eqref{generators}.
\begin{corollary}
\label{lemmaleft1} Let $\theta \in \bbbr$ be fixed. There exists a constant $C>0$ (independent of $f(\cdot)$ and $N$) such that
\begin{equation}
\begin{split}
 \langle \mc L_{N, b}\sqrt{f},\sqrt{f} \rangle_{\nu_{\rho(\cdot)}^N}& \le  -\dfrac{1}{4}\mc D_{N,b}(\sqrt{f},\nu_{\rho(\cdot)}^N) \\
&+ \frac{C\kappa }{N^{\theta}}\sum_{x\in\Lambda_N}\Big(\rho(\tfrac xN)-\alpha\Big)^2p(x)\\&+ \frac{C\kappa}{N^{\theta}}\sum_{x\in\Lambda_N}\Big(\rho(\tfrac xN)-\beta\Big)^2p(N-x)
\end{split}
\end{equation}
for any density $f(\cdot)$ with respect to $\nu_{\rho(\cdot)}^N$.
\end{corollary}
To prove the first corollary take $c \equiv1$, $T(\eta)=\eta^{x,y}$ and note that $$| \theta^{x,y} (\eta) -1 |^2 \leq C (\rho(\tfrac xN) -\rho (\tfrac yN))^2.$$ To prove  the second corollary  we take for each $y\in\{0,N\}$,  $c (\eta) = c_{x}(\eta;r(y))$  and $T(\eta) = \eta^{x}$. 
From the two previous corollaries the claim (\ref{dir_est}) follows easily. We leave the details of the gaps to the reader.

\subsection{Replacement Lemmas}
In this section we prove rigorously all the replacements that were mentioned along the Sections \ref{sec:heuri_ssep} and \ref{sec:heuri_long_jumps}.  We first recall Lemma 5.5  of \cite{BGJO} adapted to our situation (with just one reservoirs at each end point of the bulk).
\begin{lemma}
 \label{bound}
For any density $f(\cdot)$ with respect to $\nu_{\rho(\cdot)}^N$,  any $x\in \Lambda_{N}$, any $y\in\{0,N\}$ and any positive constant $A_x$, there exists a constant $C$ such that
$$\left\vert \left\langle \eta(x)-r(y),f\right\rangle_{\nu_{\rho(\cdot)}^{N}} \right\vert \; \leq\; \dfrac{C}{A_{x}}  I_{x}^{r(y)}(\sqrt{f},\nu_{\rho(\cdot)}^{N})+ CA_{x}+C\Big|\rho(\tfrac{x}{N})- r(y)\Big|.$$
\end{lemma}

 The first replacement lemma that we prove is the one that is needed for the model of Section \ref{chap2} when $p(\cdot)$ has finite variance for the case $\theta\geq 1$. 
 \begin{lemma} \label{Rep-Neumann}
For any $t>0$, for $\gamma>2$ and for any   $ \theta \geq 1$ we have that
\begin{equation*}
\begin{split}
&\lim_{\varepsilon \to 0}\lim _{N\rightarrow \infty} E_{\bbbp_{\mu_N}}\left[\Big|\int_0^t\sum_{x\in\Lambda_N}\Theta_x^- \, (\eta_{sN^2}(x)-\overrightarrow{\eta}^{\varepsilon N}_{sN^2}(1))\, ds\Big|\right] =0,\\
&\lim_{\varepsilon \to 0}\lim _{N\rightarrow \infty}  E_{\bbbp_{\mu_N}}\left[\Big|\int_0^t\sum_{x\in\Lambda_N}\Theta_x^+ \, (\eta_{sN^2}(x)-\overleftarrow{\eta}^{\varepsilon N}_{sN^2}(N-1))\, ds\Big|\right] =0.
\end{split}
\end{equation*}
\end{lemma}
\begin{proof}
Below $C$ is a constant than can change from line to line.
Note that since $\theta\geq 1$ we have $\theta(N)=N^2$. We present the proof for the first term, but we note that the proof for the second one is analogous. Here we take as reference measure the Bernoulli product measure with constant parameter (for example $\alpha$) and we recall  \eqref{dir_est_const}, from where we see that  
\begin{equation}\label{eq:errors_df}
\begin{split}
\frac{N}{B}\langle \mc L_{N}\sqrt{f},\sqrt{f} \rangle_{\nu_{\alpha}} &\leq  -\dfrac{N}{4B}\mc D_{N}(\sqrt{f},\nu_{\alpha}^N) + \frac{C\kappa}{B} N^{1-\theta}
\end{split}
\end{equation}
so that the error to change the Dirichlet form vanishes as $N\to\infty$ for $\theta>1$ and for $\theta=1$ it vanishes when $B\to+\infty$.

By the entropy  and Jensen's  inequalities, the first expectation in the statement of the lemma is bounded from above, for any constant $B>0$, by
\begin{equation*}
\begin{split}
\dfrac{H(\mu _{N}\vert \nu^{N}_{\alpha})}{BN}
+ \dfrac{1}{BN}\log {E}_{\bbbp_{\nu^{N}_{\alpha}}}\left[ e^{BN\Big| \int _{0}^{t}\sum _{x \in \Lambda_{N}}\Theta_x^-(\eta_{sN^2}(x)-\overrightarrow{\eta}^{\varepsilon N}_{sN^2}(1))\, ds\Big|}\right]\\
\end{split}.
\end{equation*}
We can remove the absolute value inside the exponential since $e^{\vert x\vert} \leq e^{x}+e^{-x}$ and  \begin{equation}\label{Log bounded}
\limsup_{N\rightarrow\infty} N^{-1}\log(a_{N}+b_{N})\leq \max \left\lbrace \limsup_{N\rightarrow\infty} N^{-1}\log(a_{N}), \limsup_{N\rightarrow\infty} N^{-1}\log(b_{N}) \right\rbrace.
\end{equation}
By \eqref{H}, the Feynman-Kac's formula and  \eqref{dir_est_const}, last expression can be estimated from above by 
\begin{equation} 
\label{repla_term}
\frac {C_0}{ B} +t \sup _{f}\Big\{\sum _{x \in \Lambda_{N}}\Theta_x^-\langle \eta(x)-\overrightarrow{\eta}^{\varepsilon N}(1),f\rangle_{\nu^{N}_\alpha}   -\dfrac{N}{4B}\mc D_{N}(\sqrt{f},\nu_{\alpha}) + \frac{C\kappa}{B} N^{1-\theta}\Big\},
\end{equation}
where the supremum is carried over all the densities $f(\cdot)$ with respect to $\nu_{\alpha}^N$. 

Now we have to split the sum in $x$, depending on whether  $N-1 \ge x \ge \varepsilon N$ or $x\leq \varepsilon N-1$. We start by the first case and we have 
\begin{eqnarray*}
\langle \eta(x)-\overrightarrow{\eta}^{\varepsilon N}(1),f\rangle_{\nu^{N}_\alpha} &=& \dfrac{1}{\varepsilon N}\sum _{y=1}^{1+\varepsilon N} \int (\eta(x)-\eta(y))f(\eta)\; d\nu^{N}_\alpha\\
&=& \dfrac{1}{1+\varepsilon N}\sum _{y=1}^{\varepsilon N}\sum_{z=y}^{x-1}\int (\eta(z+1)-\eta(z))f(\eta)\; d\nu^{N}_\alpha.
\end{eqnarray*}
By writing the previous term as its half plus its half and by performing in one of the terms the change of variables 
$\eta$ into $\eta^{z,z+1}$, for which the measure $\nu^N_\alpha$ is invariant, we write it as 
$$\dfrac{1}{2\varepsilon N}\sum _{y=1}^{1+\varepsilon N}\sum_{z=y}^{x-1}\int (f(\eta)-f(\eta^{z,z+1}))
(\eta(z+1)-\eta(z))\, d{\nu_\alpha^N}.$$
By using the fact that $(a-b)=(\sqrt a -\sqrt b)(\sqrt a+\sqrt b)$ for any $a,b \ge 0$ and since $ab \leq \dfrac{Aa^{2}}{2}+\dfrac{b^{2}}{2A}$ for all $A>0$, we have that 
\begin{equation}
\label{est_repla_term_1}
\begin{split}
&\sum_{x=\varepsilon N}^{N-1} \Theta_x^-\langle \eta(x)-\overrightarrow{\eta}^{\varepsilon N}(1),f\rangle_{\nu^{N}_\alpha}\\
 &\le \cfrac{A}{2} \, \sum _{x=\varepsilon N}^{N-1}\dfrac{ \Theta_x^-}{2\varepsilon N}\sum _{y=1}^{1+\varepsilon N}\sum_{z=y}^{x-1}\int (\sqrt {f(\eta)}-\sqrt{f(\eta^{z,z+1})})^2 d{\nu^{N}_\alpha}\\
&+ \cfrac{1}{2A} \, \sum _{x=\varepsilon N}^{N-1} \dfrac{\Theta_x^-}{2\varepsilon N}\sum _{y=1}^{1+\varepsilon N}\sum_{z=y}^{x-1}\int (\sqrt {f(\eta)}+\sqrt{f(\eta^{z,z+1})})^2 (\eta(z+1)-\eta(z))^2 d{\nu^{N}_\alpha}.
\end{split}
\end{equation}
By neglecting the jumps of size bigger than one, we see that
$$\sum_{z\in\Lambda_N}\int \Big(\sqrt {f(\eta)}-\sqrt{f(\eta^{z,z+1})}\Big)^2\; d{\nu_{\alpha}^N\; }\leq C \; \mc D_{N,0}(\sqrt f,\nu_{\alpha}^N).$$ 
Therefore, by using also (\ref{eq:thetaminus}), the first term at the right hand side of (\ref{est_repla_term_1}) can be bounded from above by 
\begin{eqnarray}
\label{eq:ouf}
\cfrac{A}{4}  \sum _{x =\varepsilon N}^{N-1} \Theta_x^-\; \sum_{z\in\Lambda_N}\int \Big(\sqrt {f(\eta)}-\sqrt{f(\eta^{z,z+1} )}\Big)^2\; \leq C A \mc D_{N,0}(\sqrt f,\nu^{N}_\alpha).
\end{eqnarray}
Recall (\ref{dir_est_const}) and observe that $$\mc D_N (\sqrt f,\nu_\alpha^N) \ge \mc D_{N,0} (\sqrt f,\nu_\alpha^N).$$ Then we choose the constant $A$ in the form $A=C N/B$ for some constant $C$. Moreover, for this choice of $A$,  we can bound from above the last term at the right hand side of \eqref{est_repla_term_1} by (use Lemma \ref{lem:densaf}) 
\begin{equation}
\begin{split}
\frac{B}{N}\; \sum _{x=\varepsilon N}^{N-1} \Theta_x^-\; &\dfrac{1}{2\varepsilon N}\sum _{y=1}^{\varepsilon N}\sum_{z=y}^{x-1}\int (\sqrt {f(\eta)}+\sqrt{f(\eta^{z,z+1})})^2
(\eta(z+1)-\eta(z))^2 d{\nu^{N}_\alpha}\; \\
&\leq C  \frac{B}{N}\sum _{x \in \Lambda_{N}}x\Theta_x^-
\end{split}
\end{equation}
which vanishes as $N \to \infty$ by (\ref{dir_est_const}). Note that the previous result holds for any  $\varepsilon>0.$ Now we analyse the case when $x\le \varepsilon N -1$. In that case, we write
\begin{equation*}
\begin{split}
\langle \eta(x)-\overrightarrow{\eta}^{\varepsilon N}(1),f\rangle_{\nu^{N}_\alpha} &=\dfrac{1}{1+\varepsilon N}\sum _{y=1}^{\varepsilon N} \int (\eta(x)-\eta(y))f(\eta)\; d\nu^{N}_\alpha\\
&= \dfrac{1}{\varepsilon N}\sum _{y=1}^{x-1}\sum_{z=y}^{x-1}\int (\eta(z+1)-\eta(z))f(\eta)\; d\nu^{N}_\alpha \\&-  \dfrac{1}{\varepsilon N}\sum _{y=x+1}^{1+\varepsilon N}\sum_{z=x}^{y-1}\int (\eta(z+1)-\eta(z))f(\eta)\; d\nu^{N}_\alpha .
\end{split}
\end{equation*}
and the same estimates as before give that there exists a constant $C>0$ such that for any $A>0$, 
\begin{equation*}
\begin{split}
\sum _{x=1}^{\varepsilon N-1} \Theta_x^-\langle \eta(x)-\overrightarrow{\eta}^{\varepsilon N}(1),f\rangle_{\nu^{N}_\alpha} &\le C \left[ A D_N (\sqrt f ,\nu_\alpha^N) + \cfrac{\varepsilon N}{A} \sum_{x=1}^{\varepsilon N -1} \Theta_x^-  \right].
\end{split}
\end{equation*}
Recall (\ref{dir_est_const}) and (\ref{eq:thetaminus}). Then, we choose $A=N/ \, 8CB$ and the result follows. \qed
\end{proof}

\begin{remark}\label{sec_rep_robin}
We note that above, if we change in the statement of the lemma $\Theta_x^-$ by $r_N^-$ (resp. $\Theta_x^+$ by $r_N^+$) then the same result holds by performing exactly the same estimates as above, because what we need is that
\begin{equation}
\sum _{x \in \Lambda_{N}}\Theta_x^\pm <+\infty \quad \textrm {and}\quad \frac{1}{N} \sum _{x \in \Lambda_{N}}x\Theta_x^\pm\to_{N\to+\infty}0
\end{equation}
which also holds for $r_N^\pm $ instead of $\Theta_x^\pm$ since $\gamma>2.$
\end{remark}

\begin{remark}\label{rep_for_ssep_robin}
Let us see now what the previous lemma says when $p(1)=p(-1)=\frac 12$. In this case we note that we have the same estimate as in \eqref{eq:errors_df}, see 2.  in Remark \ref{prices_for_DF} and also note  that $\Theta_x^-\neq 0$ for $x=1$ and $\Theta_x^-= 0$ for $x\neq1$. Moreover,  $\Theta_1^-=p(1)=\frac 12$, so that the result above reads as 
\begin{equation*}
\lim_{\varepsilon \to 0}\lim _{N\rightarrow \infty}  E_{\bbbp_{\mu_N}}\left[\Big|\int_0^t (\eta_{sN^2}(1)-\overrightarrow{\eta}^{\varepsilon N}_{sN^2}(1))\, ds\Big|\right] =0.
\end{equation*}
\begin{equation*}
\lim_{\varepsilon \to 0}\lim _{N\rightarrow \infty} E_{\bbbp_{\mu_N}}\left[\Big|\int_0^t (\eta_{sN^2}(N-1)-\overleftarrow{\eta}^{\varepsilon N}_{sN^2}(N-1))\, ds\Big|\right] =0.\\
\end{equation*}
\end{remark}

\subsection{Fixing the profile at the boundary}
\label{fix_profile}
Let $\bbbq$ be a limit point of the sequence $\lbrace\bbbq_{N}\rbrace_{N \geq 1}$ and assume, without lost of generality, that $\lbrace\bbbq_{N}\rbrace_{ N \ge 1}$ converges to $\bbbq$, as $N\to+\infty$.  In this section we prove that for the model of Section \ref{chap2} if $\theta \in[1-\gamma,1)$ (and also for the model of Section \ref{chap1} when $\theta<0$) that the profile satisfies $\rho_t(0)=\alpha$ and $\rho_t(1)=\beta$ for $t\in(0,T]$ a.e. We present the proof for $\rho_t(0)=\alpha$ but the other case is completely analogous. 

Recall \eqref{boxes_ssep}. Observe that 
$$ E_{\bbbp_{\mu_N}}\left[\Big|\int_0^t(\overrightarrow{\eta}^{\varepsilon N}_{sN^2}(1)-\alpha)\, ds\Big|\right] = {E}_{\bbbq_N}  \left[\Big|\int_0^t( \langle \pi_s, \iota^0_\varepsilon \rangle -\alpha)\, ds\Big|\right]$$
where $\iota_\varepsilon^0 (\cdot)  =\varepsilon^{-1} \, {\bf 1}_{(0,\varepsilon)} (\cdot)$. Therefore we have that for any $\delta>0$, 
$${\bbbq_N}  \left[\Big|\int_0^t( \langle \pi_s, \iota^0_\varepsilon \rangle -\alpha)\, ds\Big| >\delta \right] \; \le \; \delta^{-1} \,   E_{\bbbp_{\mu_N}}\left[\Big|\int_0^t(\overrightarrow{\eta}^{\varepsilon N}_{sN^2}(1)-\alpha)\, ds\Big|\right].$$ 
Note that $\iota_{\varepsilon}^0$ is not a continuous function so the set $$\Big\{ \pi \, ;\, \Big|\int_0^t( \langle \pi_s, \iota^0_\varepsilon \rangle -\alpha)\, ds\Big| >\delta \Big\}$$ is not an open set in the Skorohod topology, but, a simple argument as we did in Section \ref{sec:charac_limit_p} allows to overcome the problem. Therefore, by Portemanteau's Theorem  we conclude that
$${\bbbq}  \left[\Big|\int_0^t( \langle \pi_s, \iota^0_\varepsilon \rangle -\alpha)\, ds\Big| >\delta \right] \; \le \; \delta^{-1} \, \liminf_{N \to \infty} \,  E_{\bbbp_{\mu_N}}\left[\Big|\int_0^t(\overrightarrow{\eta}^{\varepsilon N}_{sN^2}(1)-\alpha)\, ds\Big|\right].$$
Now, if we are able to prove that the right hand side of the previous inequality is zero,  since we have that $\bbbq$ a.s. $\pi_s (dq) = \rho_s (q) dq$ with $\rho_s(\cdot)$ a continuous function in $0$ for a.e. $s$, by taking the limit $\varepsilon \to 0$, we can deduce that $\bbbq$ a.s. $\rho_s (0) =\alpha$ for $s$ a.e. The result follows from the next lemma.

\begin{lemma} \label{fix_prof}
For any $t\in[0,T]$ we have that 
\begin{equation*}
\begin{split}
&\lim_{\varepsilon \to 0} \, \lim _{N\rightarrow \infty} E_{\bbbp_{\mu_N}}\left[\Big|\int_0^t(\overrightarrow{\eta}^{\varepsilon N}_{sN^2}(1)-\alpha)\, ds\Big|\right] =0,\\
&\lim_{\varepsilon \to 0}  \, \lim _{N\rightarrow \infty}  E_{\bbbp_{\mu_N}}\left[\Big|\int_0^t(\overleftarrow{\eta}^{\varepsilon N}_{sN^2}(N-1)-\beta)\, ds\Big|\right] =0.
\end{split}
\end{equation*}
\end{lemma}
To prove last lemma we use a two step procedure. First we replace, when integrated in time, $\eta_{sN^2}(1)$ by $\alpha$ and then we replace $\eta_{sN^2}(1)$ by $\overrightarrow{\eta}_{sN^2}^{\varepsilon N}(1)$. This is the content of the next two lemmas.
\begin{lemma} \label{Rep-Dirichlet2}
For $\gamma>1$, for $1-\gamma\leq \theta<1$ and for $t\in[0,T]$ we have that 
\begin{equation*}
\begin{split}
&\lim _{N\rightarrow \infty}  E_{\bbbp_{\mu_N}}\left[\Big|\int_0^t(\eta_{sN^2}(1)-\alpha)\, ds\Big|\right] =0,\\
& \lim _{N\rightarrow \infty} E_{\bbbp_{\mu_N}}\left[\Big|\int_0^t(\eta_{sN^2}(N-1)-\beta)\, ds\Big|\right] =0.
\end{split}
\end{equation*}
\end{lemma}
\begin{proof} 
We give the proof for the first display, but  we note that for the other one it is similar. Fix a Lipschitz profile $\rho(\cdot)$ such that $\alpha=\rho(0) \le \rho(\cdot) \le \rho(1)=\beta$  and $\rho(\cdot)$ is $\frac \gamma 2$-H\"older at the boundary. From \eqref{dir_est_holder} that  we know that
\begin{equation}\label{eq:repl_imp}
\begin{split}
\frac{N}{B}\langle \mc L_{N}\sqrt{f},\sqrt{f} \rangle_{\nu_{\rho(\cdot)}^N} \leq& -\dfrac{N}{4B}\mc D_{N}(\sqrt{f},\nu_{\rho(\cdot)}^N) + \frac{C}{B} \sigma^2 + \frac{C\kappa\log(N)}{BN^{\gamma+\theta+1}}.
\end{split}
\end{equation}

 By the entropy inequality, for any $B>0$, the previous expectation is bounded from above by
\begin{equation*} \label{use1}
\begin{split}
\dfrac{H(\mu _{N}\vert {\nu_{\rho(\cdot)}^{N}}) }{BN}
+ \dfrac{1}{BN}\log {E}_{\bbbp_{\nu_{\rho(\cdot)}^{N}}}\Big[ e^{BN\Big| \int _{0}^{t}(\eta_{sN^2}(1)-\alpha)\, ds\Big|}\Big].
\end{split}
\end{equation*}
By \eqref{H}, Jensen's inequality and  the Feynman-Kac's formula and noting, as we did in the last proof, that we can remove the absolute value inside the exponential, last display can be estimated from above by 
\begin{equation} 
\label{F-K 333}
\frac {C_0} {B}+ t \sup _{f}\left\{\left\langle \eta(1)-{\alpha}, f\right\rangle_{{\nu_{\rho(\cdot)}^{N}}}   -\dfrac{N}{4B}\mc D_{N}(\sqrt{f},\nu_{\rho(\cdot)}^N) + \frac{C}{B} \sigma^2 + \frac{C\kappa}{BN^{\gamma+\theta-1}}\right\},
\end{equation}
where the supremum is carried over all the densities $f(\cdot)$ with respect to $\nu_{\rho(\cdot)}^N$. By Lemma \ref{bound}, since $\rho(\cdot)$ is $\frac \gamma 2$-H\"older at the boundaries, for any $A>0$, the first term in the supremum in  (\ref{F-K 333}) is bounded from above by 
\begin{equation*}
C \left[ \dfrac{1}{A}  \; I_{1}^{\alpha}(\sqrt{f},\nu_{\rho(\cdot)}^{N})+ {A}+\frac{1}{N^{\gamma/2}} \right]
\end{equation*}
for some constant $C>0$ independent of $f(\cdot)$ and $A$. Moreover from \eqref{dir_est_holder}, since 
$$\mc D_N (\sqrt{f}, \nu_{\rho(\cdot)}^N) \ge \mc D_{N,b} (\sqrt{f}, \nu_{\rho(\cdot)}^N)$$
and $\gamma+\theta -1 > 0$, by  
choosing $A =4C  (p(1))^{-1}B N^{\theta -1}$, we get then that the expression inside the brackets in (\ref{F-K 333}) is bounded from above by 
\begin{eqnarray*}
4 C^2  \dfrac{BN^{\theta-1}}{p(1)}+\cfrac{C}{N^{\gamma/2}}+\dfrac{C}{B}.
\end{eqnarray*}
Now if $p(1)\neq 0$,  then the proof follows by sending first $N\to \infty$  and then $B\to \infty$. 
For $\gamma+\theta-1=0$ the same proof as above holds, the only difference is that we use a Lipschitz profile $\rho(\cdot)$ such that $\alpha=\rho(0) \le \rho(\cdot) \le \rho(1)=\beta$  and $\rho(\cdot)$ is $\tfrac{\gamma+1}{2}$-H\"older at the boundaries. From \eqref{dir_est_holder_alt} that  we know that
\begin{equation}\label{eq:repl_imp_2}
\begin{split}
\frac{N}{B}\langle \mc L_{N}\sqrt{f},\sqrt{f} \rangle_{\nu_{\rho(\cdot)}^N} \leq& -\dfrac{N}{4B}\mc D_{N}(\sqrt{f},\nu_{\rho(\cdot)}^N) + \frac{C}{B} \sigma^2 + \frac{C\kappa}{B},
\end{split}
\end{equation}
and with last bound and the previous argument the proof ends.
\end{proof}

\begin{remark}\label{RL_theta_Dir}
The previous lemma tells us that for the model of Section \ref{chap1}  and 
for $\theta<1$ and $t\in[0,T]$ we have that 
\begin{equation*}
\begin{split}
&\lim _{N\rightarrow \infty}  E_{\bbbp_{\mu_N}}\left[\Big|\int_0^t(\eta_{sN^2}(1)-\alpha)\, ds\Big|\right] =0,\\
& \lim _{N\rightarrow \infty} E_{\bbbp_{\mu_N}}\left[\Big|\int_0^t(\eta_{sN^2}(N-1)-\beta)\, ds\Big|\right] =0.
\end{split}
\end{equation*}
Note that the previous proof follows since we have the bound \eqref{dir_est_lip_ssep} and in this model  $p(1)=\frac 12$.
\end{remark}

\begin{remark}
We note that for the case where $p(1)=0$ above what we have to do is to use the two step procedure with a point  $z$ such that $p(z)\neq 0$, from where we get that: 
\begin{equation*}
\begin{split}
&\lim _{N\rightarrow \infty}  E_{\bbbp_{\mu_N}}\left[\Big|\int_0^t(\eta_{sN^2}(z)-\alpha)\, ds\Big|\right] =0
\end{split}
\end{equation*}
and the same result holds by changing $\alpha$ to $\beta$.
\end{remark}
 Now we prove the second part of the two step procedure. 
\begin{lemma}
\label{repla_novo}
For $1-\gamma\leq \theta<1$ and  $t>0$ we have that 
\begin{equation}
\begin{split}
&\lim_{\varepsilon \to 0} \, \lim _{N\rightarrow \infty}   E_{\bbbp_{\mu_N}}\Big[\Big|\int_0^t\overrightarrow{\eta}^{\varepsilon N}_{sN^2}(1)-\eta_{sN^2}(1)\, ds\Big|\Big] =0,\\
&\lim_{\varepsilon \to 0} \, \lim _{N\rightarrow \infty}   E_{\bbbp_{\mu_N}}\Big[\Big|\int_0^t\overleftarrow{\eta}^{\varepsilon N}_{sN^2}(N-1)-\eta_{sN^2}(N-1)\, ds\Big|\Big] =0.
\end{split}
\end{equation}
\end{lemma}
\begin{proof}
We present the proof of the first item, but we note that for the second it is exactly the same. When $\gamma+\theta-1>0$, we fix a Lipcshitz profile $\rho(\cdot)$ such that $\alpha=\rho(0) \le \rho(\cdot) \le \rho(1)=\beta$, and $\rho(\cdot)$ is $\frac \gamma 2$-H\"older at the boundaries, when $\gamma+\theta-1=0$, the Holder regularity at the boundary is $\frac{\gamma+1}{2}$.  Since we imposed the same conditions as in the previous lemma in the profile $\rho(\cdot)$ then in this case  \eqref{eq:repl_imp} and \eqref{eq:repl_imp_2} holds. From now on we suppose that $\gamma+\theta-1>0$, the other case is completely analogous. By the entropy and Jensen's inequalities, for any $B>0$,  the previous expectation is bounded from above by
\begin{equation*}
\dfrac{H(\mu _{N}\vert \nu_{\rho(\cdot)}^{N}) }{BN}
+ \dfrac{1}{BN}\log {E}_{\bbbp_{\nu_{\rho(\cdot)}^{N}}}\Big[ e^{BN\Big|\int_0^t\overrightarrow{\eta}^{\varepsilon N}_{sN^2}(1)-\eta_{sN^2}(1)\, ds\Big|}\Big].
\end{equation*}
By \eqref{H},  the Feynman-Kac's formula, and using the same argument as in the proof of the previous  lemma, the estimate of the previous expression can be reduced to bound
\begin{equation}\label{E1} 
\begin{split}\frac {C_0} {B}+ t \sup _{f}\Big\{\frac{1}{\ell}\sum _{y=2}^{\ell+1} \vert \langle \eta(y)-\eta(1),f\rangle_{\nu_{\rho(\cdot)}^{N}}\vert -\dfrac{N}{4B}\mc D_{N}(\sqrt{f},\nu_{\rho(\cdot)}^N) + \frac{C}{B} \sigma^2 + \frac{C\kappa\log(N)}{BN^{\gamma+\theta-1}}\Big\},\\ 
\end{split}
\end{equation}
where $\ell=\varepsilon N $. As above, the supremum is carried over all the densities $f(\cdot)$ with respect to $\nu_{\rho(\cdot)}^N$. 
Note that since $y \in \Lambda_{N}$ we know that $$\eta(y)-\eta(1) = \sum _{z=1}^{y-1}(\eta(z+1)-\eta(z)).$$
Observe now that 
\begin{equation*}
\begin{split} 
\int(\eta(z+1)-\eta(z)) f(\eta)d\nu_{\rho(\cdot)}^N =& \dfrac{1}{2}\int(\eta(z+1)-\eta(z))(f(\eta)-f( \eta^{z, z+1}))d\nu_{\rho(\cdot)}^{N}\\
+&\dfrac{1}{2}\int(\eta(z+1)-\eta(z))(f(\eta)+f(\eta^{z, z+1}))d\nu_{\rho(\cdot)}^{N}.
\end{split}\end{equation*}
By using the fact that for any $a,b \ge 0$, $(a-b)=(\sqrt a -\sqrt b)(\sqrt a +\sqrt b)$ and Young's inequality, we have, for any positive constant $A$, that 
\begin{equation}\label{eq:term1}
\begin{split}  
\frac{1}{\ell}\sum _{y=2}^{\ell+1} \vert \langle& \eta(y)-\eta(1),f\rangle_{\nu_{\rho(\cdot)}^{N}}\vert  \; \\
& \le \;   \dfrac{1}{2A\ell} \sum _{y=2}^{\ell+1} \sum _{z=1}^{y-1}\int (\eta(z+1)-\eta(z))^2\; \Big(\sqrt{f(\eta)}+\sqrt{f( \eta^{z, z+1})}\Big)^2d\nu_{\rho(\cdot)}^{N}\\
&+\dfrac{A}{2\ell}\sum _{y=2}^{\ell+1}  \sum _{z=1}^{y-1}\int \Big(\sqrt{f(\eta)}-\sqrt{f( \eta^{z, z+1})}\Big)^2 d\nu_{\rho(\cdot)}^{N}\\
&+\dfrac{1}{2\ell} \sum _{y=2}^{\ell+1} \left| \sum _{z=1}^{y-1}\int \big( \eta(z+1)-\eta(z) \big) \; \Big(f(\eta)+f(\eta^{z, z+1})\Big) d\nu_{\rho(\cdot)}^{N} \right|.
\end{split}
\end{equation}
Now, we neglect jumps of size bigger than one as we did below \eqref{est_repla_term_1}, from where we get that the second term on the right hand side of (\ref{eq:term1}) is bounded from above by $  C A \mc D_N (\sqrt f,\nu_{\rho(\cdot)}^N) $
where $C$ is a positive constant independent of $A,\ell,f$. Then, for the choice $A=N(4BC)^{-1}$ and  since $\gamma +\theta -1 \ge 0$, we can bound from above \eqref{E1} by 
\begin{equation}
\label{eq:E1234567890}
\begin{split}
&\dfrac{2BC}{N \ell} \sum _{y=2}^{\ell+1} \sum _{z=1}^{y-1}\int (\eta(z+1)-\eta(z))^2\; \Big(\sqrt{f(\eta)}+\sqrt{f(\eta^{z, z+1} )}\Big)^2d\nu_{\rho(\cdot)}^{N}\\
+&\dfrac{1}{2\ell} \sum _{y=2}^{\ell+1} \left| \sum _{z=1}^{y-1}\int \big( \eta(z+1)-\eta(z) \big) \; \Big(f(\eta)+f(\eta^{z, z+1} )\Big) d\nu_{\rho(\cdot)}^{N} \right| + \frac{C' }{B}\\
\leq& C\Big( \frac{B\ell}{N} + \frac{1}{B} +\dfrac{1}{2\ell} \sum _{y=2}^{\ell+1} \Big| \sum _{z=1}^{y-1}\int \big( \eta(z+1)-\eta(z) \big) \; \Big(f(\eta)+f(\eta^{z, z+1} )\Big) d\nu_{\rho(\cdot)}^{N} \Big|\Big)
\end{split}
\end{equation}
for some constant $C$. For the last inequality we used Lemma \ref{lem:densaf}. Observe that $B\ell/N = B\varepsilon$ vanishes as $\varepsilon \to 0$. It remains to estimate the third term on the right hand side of the last inequality. For that purpose we make a similar computation to the one of Lemma \ref{bound} from where we get that
\begin{equation*}
\begin{split}
&\sum_{z=1}^{y-1}
\left\vert \int(\eta(z+1)-\eta(z))(f(\eta)+f(\eta^{z, z+1} )) d\nu_{\rho(\cdot)}^{N}\right\vert \leq C \; \sum_{z=1}^{y-1}\Big|\rho\Big(\tfrac{z+1}{N}\Big)-\rho\Big(\tfrac{z}{N}\Big)\Big)\Big|.
\end{split}
\end{equation*}
Since $\rho(\cdot)$ is Lipschitz, by (\ref{eq:E1234567890}), this estimate provides an upper bound for  \eqref{E1} which is in the form of a constant times 
$$  \frac{B\ell}{N} + \frac{1}{B} +\frac{1}{N\ell}\sum _{y=2}^{\ell+1} y \; \leq  B\varepsilon + B^{-1} + \varepsilon$$
 which vanishes, as $\varepsilon \to 0$ and then $B \to \infty$. This ends the proof. \qed
 \end{proof}

\paragraph{Acknowledgements}
This project has received funding from the European Research Council (ERC) under  the European Union's Horizon 2020 research and innovative programme (grant agreement   No 715734). \\

The author would like to express her gratitude to the organizers of the trimester "Stochastic dynamics out of equilibrium", namely Ellen Saada, Gabriel Stoltz,  Giambattista Giacomin,  Herbert Spohn, Stefano Olla, for the invitation to taught the mini-course and for the financial support to attend  the trimester. \\

Finally the author would like to thank C\'edric Bernardin, Byron Jim\'enez Oviedo and Adriana Neumann for the discussions around the problems described in these notes. 
%
%

\end{document}